\documentclass[12pt,a4paper]{amsart}
\makeatletter
\renewcommand\normalsize{%
    \@setfontsize\normalsize{11.7}{14pt plus .3pt minus .3pt}%
    \abovedisplayskip 10\p@ \@plus4\p@ \@minus4\p@
    \abovedisplayshortskip 6\p@ \@plus2\p@
    \belowdisplayshortskip 6\p@ \@plus2\p@
    \belowdisplayskip \abovedisplayskip}
\renewcommand\small{%
    \@setfontsize\small{9.5}{12\p@ plus .2\p@ minus .2\p@}%
    \abovedisplayskip 8.5\p@ \@plus4\p@ \@minus1\p@
    \belowdisplayskip \abovedisplayskip
    \abovedisplayshortskip \abovedisplayskip
    \belowdisplayshortskip \abovedisplayskip}
\renewcommand\footnotesize{%
    \@setfontsize\footnotesize{8.5}{9.25\p@ plus .1pt minus .1pt}
    \abovedisplayskip 6\p@ \@plus4\p@ \@minus1\p@
    \belowdisplayskip \abovedisplayskip
    \abovedisplayshortskip \abovedisplayskip
    \belowdisplayshortskip \abovedisplayskip}
\ifdefined\pdfpagewidth
\else
\fi
\calclayout
\makeatother

\usepackage{amsmath,amssymb,cite,mathrsfs,tikz-cd}
\usepackage[all]{xy}
\usepackage{hyperref} 

\usepackage{pdfcomment} 

\newtheorem{teo}{Theorem}[section]
\newtheorem{thm}[teo]{Theorem}
\newtheorem{prop}[teo]{Proposition}
\newtheorem{lemma}[teo]{Lemma}

\makeatletter
\newcommand{\neutralize}[1]{\expandafter\let\csname c@#1\endcsname\count@}
\makeatother

\newenvironment{thmrepeat}[1]
  {\renewcommand{\theteo}{\ref*{#1}}%
   \neutralize{teo}\phantomsection
   \begin{thm}}
  {\end{thm}}

\newenvironment{thmbis}[1]
  {\renewcommand{\theteo}{\ref*{#1}$'$}%
   \neutralize{teo}\phantomsection
   \begin{thm}}
  {\end{thm}}

\numberwithin{equation}{section}

\newenvironment{equationrepeat}[1]
  {\renewcommand{\theequation}{\ref*{#1}}%
   \neutralize{equation}\phantomsection
   \begin{equation}}
  {\end{equation}}

\newcommand{\IQbar}{\overline{\mathbb{Q}}}

\newcommand{\cA}{\mathcal{A}}
\newcommand{\cB}{\mathcal{B}}

\newcommand{\cC}{{\mathcal C}}

\newcommand{\cL}{{\mathcal L}}
\newcommand{\cM}{{\mathcal M}}

\newcommand{\pullbackcorner}[1][dr]{\save*!/#1-1.7pc/#1:(-1.5,1.5)@^{|-}\restore}

\usepackage{color}

\begin{document}
\title{The Uniform Mordell--Lang Conjecture}
\author{Ziyang Gao, Tangli Ge, Lars K\"{u}hne}

\address{Department of Mathematics, UCLA, Los Angeles, CA 90095, USA}
\email{ziyang.gao@math.ucla.edu}

\address{Department of Mathematics, Princeton University, Princeton, NJ 08544, USA}
\email{tangli@princeton.edu}

\address{School of Mathematics and Statistics, University College Dublin, Dublin 4, Ireland}
\email{lars.kuehne@ucd.ie}






\maketitle

\begin{abstract}
	
The Mordell--Lang conjecture for abelian varieties states that the intersection of an algebraic subvariety $X$ with a subgroup of finite rank is contained in a finite union of cosets contained in $X$. In this article, we prove a uniform version of this conjecture, meaning that the number of cosets necessary depend only on the dimension of the ambient variety and the degree (with respect to some polarization) of the subvariety $X$. To achieve this, we prove a general gap principle on algebraic points that extends the new gap principle for curves embedded into their Jacobians, previously obtained by Dimitrov--Gao--Habegger and K\"{u}hne. 
Our new gap principle also implies the full uniform Bogomolov conjecture in abelian varieties.
\end{abstract}

\tableofcontents

\section{Introduction}

Throughout this article, $F$ is an algebraically closed field of characteristic $0$; particularly relevant cases are $F=\IQbar$ and $F=\mathbb{C}$. Furthermore, we let $A$ be an abelian variety defined over $F$ and consider an ample line bundle $L$ on $A$. The translates of abelian subvarieties by closed points in $A$ are called \emph{cosets}.

The main result of our article is the following uniform version of the well-known Mordell--Lang conjecture (see \cite[p.\ 138]{Lang1962}). 

\begin{thm}[Uniform Mordell--Lang Conjecture]\label{MainThm2}
    For all integers $g,d \geq 0$, there exists a constant $c(g,d)>0$ with the following property. Let $X \subseteq A$ be an irreducible closed subvariety and $\Gamma\subseteq A(F)$ a subgroup of finite rank. Then the intersection $X(F)\cap\Gamma$ is covered by at most $$c(\dim A,\deg_L X)^{1+\mathrm{rk}\Gamma}$$ cosets contained in $X$.
\end{thm}

For readers' convenience, let us restate this theorem in a more explicit fashion. For every polarized abelian variety of dimension $g$, every irreducible closed subvariety $X \subseteq A$ of degree $d$ (with respect to the given polarization) and every subgroup $\Gamma$ of finite rank $\rho$, it claims the existence of cosets
\begin{equation*}
	x_i+B_i \subseteq X, \quad 1\leq i \leq N \leq c(g,d)^{1+\rho},
\end{equation*}
such that
\begin{equation}\label{EqHighDimML}
	X(F)\cap\Gamma = \bigcup_{i=1}^N (x_i+B_i)(F) \cap \Gamma.
\end{equation}

\medskip

The \textit{original} Mordell--Lang conjecture asserts that \textit{finitely} many cosets 
\begin{equation*}
	x_i + B_i \subseteq X, \ 1 \leq i \leq N,
\end{equation*}
are sufficient to cover the intersection $X(F) \cap \Gamma$ as in \eqref{EqHighDimML}, without suggesting any further quantitative control on their number. This original version, which combines both the Manin--Mumford conjecture \cite{Raynaud:MM} and the Mordell conjecture \cite{Faltings:ES}, was established by Faltings \cite{Faltings:DAAV,Faltings94}, following work of Hindry \cite{Hindry:Lang} and Vojta \cite{Vojta:siegelcompact}. An explicit upper bound for the number of cosets was obtained by R\'emond \cite{Remond:Decompte}, which additionally depends on the ambient abelian variety $A$ via its Faltings height $h_{\mathrm{Fal}}(A)$.

The novelty here is the complete removal of this very dependence on the ambient abelian variety $A$, confirming a  folklore expectation that can be found for example in \cite[Conj.~1.8]{DaPh:07}. It is known under the name of \textit{uniform Mordell--Lang conjecture} because the number $N$ in \eqref{EqHighDimML} must depend on $g$, $d$ and $\rho$. 
For curves embedded in their Jacobian, it dates back to a question of Mazur \cite[top of p.\ 234]{mazur1986arithmetic}, which has been answered affirmatively by work of Dimitrov, the first-named author, Habegger and the third-named author \cite{DGHUnifML,Kuehne:21}. Readers may profit from the survey of the first-named author \cite{GaoSurveyUML} for an overview of these previous works and their implications for rational points on algebraic curves. Let us remark that before these works, the only uniform results of Mordell-Lang type were obtained by David and Philippon \cite[Thm.~1.13]{DaPh:07} for subvarieties of self-products of an elliptic curve. It should be also noted that they give a completely explicit constant in this special case. In this regard, it is interesting to ask whether the present arguments can yield \textit{explicit} upper bounds on the number of cosets similar to those of David and Philippon \cite{DaPh:07} or if substantial new ideas are necessary.

Our proof of Theorem~\ref{MainThm2} relies on the ideas established in the series of work \cite{DGH1p, DGHUnifML, DGHBog, Kuehne:21} building upon Vojta's approach \cite{Vojta:siegelcompact} to the Mordell conjecture. However, several new difficulties arise in the higher-dimensional case considered here; see $\mathsection$\ref{SubsectionIdeaProof} for a short discussion. A comprehensive outline of all other sections can be found at the end of the introduction.

\subsection{The Ueno locus and first reductions}

The \textit{Ueno locus} of $X$ is the union of positive dimensional cosets contained in $X$. A result of Kawamata \cite[Thm.~4]{KawamataLocus} states that it is Zariski closed. Write $X^{\circ}$ for its complement in $X$; then $X^{\circ}$ is a Zariski open set. Note that the Ueno locus of a smooth, proper curve of genus $g\geq 2$ embedded into its Jacobian is empty.
 
A recursive argument in $\mathsection$\ref{SubsectionMainThm3toMainThm2} reduces Theorem~\ref{MainThm2} to the following weaker statement.

\begin{thmbis}{MainThm2}\label{MainThm3}
    For all integers $g,d \geq 0$, there exists a constant $c(g,d)>0$ with the following property. Let $X \subseteq A$ be an irreducible closed subvariety and $\Gamma\subseteq A(F)$ a subgroup of finite rank. Then 
    \begin{equation}\label{EqBoundLatticePointsOutsideUeno}
            \#X^\circ(F)\cap\Gamma \le c(\dim A,\deg_L X)^{1+\mathrm{rk}\Gamma}.
    \end{equation}
\end{thmbis}

Besides reducing to Theorem \ref{MainThm3}, we also use a specialization argument of Masser \cite{masser1989specializations} to assume $F=\IQbar$ in our main arguments. This is an essential reduction, since we will mostly work with
arithmetical arguments around the theory of heights. 
 Although Theorem \ref{MainThm3} could itself be understood as a purely geometric assertion (\textit{e.g.}\ for $F=\mathbb{C}$), we heavily rely on arithmetic tools throughout our proof.

\subsection{A generalized gap principle}
Our approach is modeled on Vojta's proof of the Mordell conjecture \cite{Vojta:siegelcompact} and its later refinements, notably the quantitative ones obtained by Rémond \cite{Remond:Decompte, Remond:Vojtasup}. His method leads to a dichotomy between algebraic points of \textit{large} and \textit{small} N\'eron--Tate height, which we call large and small points for simplicity.

A uniform count of large points can be done by using the work of R\'emond \cite{Remond:Decompte, Remond:Vojtasup}, which provides explicit, generalized versions of Mumford's and Vojta's inequalities. Compared to the case of curves in their Jacobians dealt with in \cite{DGH1p, DGHUnifML}, some extra work is actually needed to establish Mumford's inequality. In particular, we need to invoke the induction hypothesis twice to handle large points. We give a detailed account in Appendix~\ref{AppendixRemond} without claiming originality.

Our main contribution is a uniform count of small points, and we achieve this by establishing another kind of  
 \textit{gap principle} for algebraic points. In the case of curves in their Jacobians, 
 preceding work \cite{DGHUnifML,Kuehne:21} has been subsumed under such a New Gap Principle \cite[Thm.~4.1]{GaoSurveyUML}. In this article, we take the same perspective and generalize it as follows. We say that a subvariety $X \subseteq A$ \textit{generates} $A$ if the smallest abelian subvariety containing $X-X$ is $A$. Furthermore, we let $\hat{h}_{L\otimes [-1]^*L}$ denote the N\'eron--Tate height on $A(\IQbar)$ associated with the symmetric ample line bundle $L\otimes [-1]^*L$.

\begin{thm}[New Gap Principle]\label{ThmSmallPointHighDim}
Given two positive integers $g$ and $d$, there exist positive constants $c_1 = c_1(g , d)$ and $c_2 = c_2(g, d)$ with the following property: For any polarized abelian variety $(A,L)$ of dimension $g$ and any irreducible closed subvariety $X \subseteq A$ that generates $A$ of degree $\deg_L X \le d$, the set
\begin{equation}\label{EqNGP}
 \left\{ P \in X^{\circ}(\IQbar) : \hat{h}_{L\otimes [-1]^*L}(P) \le c_1\max\{1, h_{\mathrm{Fal}}(A)\} \right\}, 
\end{equation}
is contained in $X'(\IQbar)$ for a proper Zariski closed $X' \subsetneq X$ with $\deg_L(X') < c_2$.
\end{thm}

Note that this theorem has been predicted by \cite[Conj.~10.5']{GaoSurveyUML}. The two examples constructed at the end of \cite[$\mathsection$10.2]{GaoSurveyUML} show that, in contrast to the case of curves embedded into their Jacobians, one can neither get rid of the assumption that $X$ generates $A$ nor assert that \eqref{EqNGP} is a finite set of uniformly bounded cardinality.

\medskip

We remark that Theorem \ref{ThmSmallPointHighDim} implies a uniform version of the Bogomolov conjecture, generalizing \cite[Thm.~3]{Kuehne:21} in the case of curves embedded in their Jacobians. Since this result is of independent interest, we state it here explicitly.

\begin{thm}[Uniform Bogomolov Conjecture]\label{ThmUBC}
Given two positive integers $g$ and $d$, there exist positive constants $c_3 = c_3(g , d)$ and $c_3 = c_3(g, d)$ with the following property: For any abelian variety $(A,L)$ of dimension $g$ and any irreducible closed subvariety $X \subseteq A$ that generates $A$ of degree $\deg_L X \le d$, we have
\begin{equation}\label{EqUBC}
\# \left\{ P \in X^{\circ}(\IQbar) : \hat{h}_{L\otimes [-1]^*L}(P) \le c_3 \right\} < c_4.
\end{equation}
\end{thm}

We emphasize that as in the case of curves, to obtain Theorem~\ref{MainThm3} it is necessary to use the New Gap Principle for \eqref{EqNGP} instead of \eqref{EqUBC} because the dichotomy of large and small points is in comparison to $\max\{1,h_{\mathrm{Fal}}(A)\}$. 
Unlike the case of curves embedded in their Jacobians, this theorem is not formally a special case of Theorem~\ref{ThmSmallPointHighDim}, 
although a deduction from Theorem~\ref{ThmSmallPointHighDim} is easy and given in  $\mathsection$\ref{SectionUBC}.

An analogue of Theorem~\ref{ThmUBC} for curves in algebraic tori was proven by Bombieri--Zannier \cite{Bombieri1995} (compare also \cite{David1999, Amoroso2006}). For the case of abelian varieties considered here, Theorem~\ref{ThmUBC} was known only in some cases \cite{DaPh:07, DeMarcoKriegerYeUniManinMumford, Kuehne:21}. Let us note that all these results except for \cite{Kuehne:21} have rather explicit constants, in contrast to our general Theorem \ref{ThmUBC} here.

\medskip

After the first version of the current article appeared as a preprint, Yuan \cite{YuanUMLCurve} gave another proof of Theorem \ref{ThmSmallPointHighDim} in the case of curves embedded into their Jacobians previously considered in \cite{DGHUnifML, Kuehne:21}. His approach has the advantage to work in the function field case (in positive characteristic) as well. It relies on previous work of Zhang, Cinkir, and de Jong \cite{ZhangAdmissiblePairing, ZhangGrossSchoen, CinkirBogomolov, DeJong} for lower bounds on the self-intersection numbers of the admissible canonical bundles of curves over global fields. Little is known in this direction for subvarieties of higher dimension and it seems very hard to generalize Yuan's proof to the setting considered here. 
He also uses the new theory of Yuan--Zhang \cite{YuanZhangEqui} on \textit{adelic line bundles on quasi-projective varieties}.

\subsection{Ideas of the proof}\label{SubsectionIdeaProof}

Non-degenerate subvarieties of abelian schemes, a notion introduced by Habegger \cite{Hab:Special} and extensively studied by the first-named author \cite{GaoBettiRank}, have played a central role in previous work \cite{DGH1p, DGHUnifML, DGHBog, Kuehne:21} and continue to do so in the current article. They derive their importance from the fact that they are the natural setting for both
\begin{enumerate}
	\item the \textit{height inequality} of \cite[Thm.~1.6 and B.1]{DGHUnifML}, which allows a comparison of the N\'{e}ron--Tate height and the height on the base variety, and
	\item the \textit{equidistribution theorem} \cite[Thm.~1]{Kuehne:21}.
\end{enumerate}
More recently, Yuan and Zhang have reproven both these results using their general theory of adelic line bundles over quasi-projective varieties \cite{YuanZhangEqui}.

A starting problem in the current paper is hence the construction of an appropriate non-degenerate subvariety. In the case of curves embedded in their Jacobians, by the quasi-finiteness of the Torelli map, we could restrict ourselves to consider subvarieties of an abelian scheme $\mathcal{A}\rightarrow S$ of \textit{maximal variation}, that is, the moduli map from $S$ to the moduli space of abelian varieties is generically finite. In the current paper, we need a space parametrizing all subvarieties of a fixed degree in abelian varieties of a fixed dimension and polarization. While there is a natural candidate,  the \textit{Hilbert scheme}, the moduli map from it to the moduli space of abelian varieties  has \textit{positive dimensional fibers}. This makes the construction of the relevant non-degenerate subvarieties significantly harder than the case of curves.

We resolve this problem by showing that the moduli map \textit{on the total space}, when restricted to the universal family over an open subset of the Hilbert scheme \eqref{EqHilbSchemeOpenLocusFamily}, becomes still generically finite after taking a high enough fibered power (Lemma~\ref{LemmaGenericFiniteHilb}), as inspired by the second-named author's work \cite[$\mathsection$3]{GeQuadraticPoints}. Then \cite[Thm.~10.1]{GaoBettiRank} gives us the desired non-degeneracy (Proposition~\ref{PropNonDeg}). We expect the idea of this construction to be applicable in other settings, for example to study uniformity problems for semiabelian varieties, which would extend in particular the uniform results of Bombieri--Zannier \cite{Bombieri1995} on algebraic tori, and  even to study related problems in some families of dynamical systems as in \cite{GauthierVignyTaflin}. The need for non-degeneracy makes it necessary to shift back and forth to fibered products. In the course of this, we have to control the exceptional sets appearing. For this purpose, we provide a technical key Lemma~\ref{LemmaNogaAlon}.

To prove uniform bounds as in Theorem~\ref{MainThm3}, it is important to work with only \textit{finitely many} families at all times. While this is automatic in the setting for curves embedded into their Jacobians (by an induction on the dimension of the subvariety), we have to carefully handle the invariants involved, in particular the polarization type. Even if one is only interested in the case of principal polarization, our inductive proof generally invokes abelian varieties of \textit{any} polarization degree. We use various techniques to overcome these problems, including a classical result of Mumford and the Poincar\'{e} biextension on the universal abelian variety.

\subsection{Outline of the article}
In $\mathsection$\ref{SectionPreliminary}, we review basic facts on abelian varieties and polarization types. In particular, we give a bound Lemma~\ref{LemmaSmallestAbVarGenerated} on the degree of the abelian subvariety generated by $X$, using an argument suggested to us by Marc Hindry. This allows us to avoid any dependence on the polarization degree $\deg_L(A)$ in the final results.

In $\mathsection$\ref{SectionBasicSetupNonDegeneracy}, we construct the families of non-degenerate subvarieties used in our main arguments. We also deduce the key result to establish their non-degeneracy here (Proposition \ref{PropNonDeg}) from the results of \cite{GaoBettiRank}.

In $\mathsection$\ref{SectionAppHtIneq}, we apply the height inequality \cite[Thm.1.6 and B.1]{DGHUnifML} to these non-degenerate subvarieties. The main result of this section is Proposition~\ref{PropHtIneqStep1}, which roughly proves an analogue of Theorem~\ref{ThmSmallPointHighDim}, albeit invoking more invariants, with \eqref{EqNGP} replaced by the set
\begin{equation*}\label{EqnIntroHtIneq}
	\left\{ x \in X^{\circ}(\IQbar) : \hat{h}_{L\otimes [-1]^\ast L}(x) \le c_1'\max\{1, h_{\mathrm{Fal}}(A)\} - c_3' \right\}
\end{equation*}
for certain uniform $c_1',c_3'>0$. Note that this set is exactly \eqref{EqNGP} if $h_{\mathrm{Fal}}(A) \ge \max\{1,2c_3'/c_1'\}$ up to changing constants, but becomes trivially empty if $\max\{1,h_{\mathrm{Fal}}(A)\} < c_3'/c_1'$. 
For counting purposes, a key technical Lemma~\ref{LemmaNogaAlon} allows us to use the non-degenerate fibered products of the (in general degenerate) Hilbert schemes constructed in $\mathsection$\ref{SectionBasicSetupNonDegeneracy}.

The main result of $\mathsection$\ref{SectionEqDist} is Proposition \ref{PropEqDistStep1}, which again has the same form as Theorem~\ref{ThmSmallPointHighDim} but without $h_{\mathrm{Fal}}(A)$; here, the set \eqref{EqNGP} replaced by the set
\begin{equation*}\label{EqnIntroEqui}
\left\{ x \in X^{\circ}(\IQbar) : \hat{h}_{L\otimes [-1]^\ast L}(x) \le c_3'' \right\}
\end{equation*}
for a certain uniform $c_3''>0$. Its proof invokes the equidistribution theorem \cite[Thm.~1]{Kuehne:21} for the same non-degenerate subvarieties as in  $\mathsection$\ref{SectionAppHtIneq}. The proof follows the classical strategy of Ullmo \cite{Ullmo} and Zhang \cite{ZhangEquidist} with substantially new technical difficulties. Lemma~\ref{LemmaNogaAlon} is  used again.

In $\mathsection$\ref{SectionEndOfNGP}, we combine Propositions \ref{PropHtIneqStep1} and \ref{PropEqDistStep1} into the desired gap principle stated in Theorem ~\ref{ThmSmallPointHighDim}.

In $\mathsection$\ref{SectionUML}, we show how to deduce the uniform Mordell--Lang conjecture by combining the gap principle with a result of Rémond. In addition, we use a specialization argument of Masser to reduce to the case $F=\IQbar$ (Lemma \ref{LemmaSpecialization}). We also include an argument to deduce Theorem~\ref{MainThm2} from Theorem~\ref{MainThm3}.

In $\mathsection$\ref{SectionUBC}, we show how to deduce the uniform Bogomolov conjecture from the new gap principle (Theorem \ref{ThmSmallPointHighDim}).

In Appendix~\ref{AppendixRemond}, we give a more detailed account of R\'emond's result that is one of the two essential ingredients for the proof in $\mathsection$\ref{SectionUML}.

\subsection*{Acknowledgements} The authors would like to thank Marc Hindry for the argument for Lemma~\ref{LemmaSmallestAbVarGenerated}. The authors would like to thank Gabriel Dill for his valuable comments on a previous version of the manuscript, especially for a correction of the argument for Lemma~\ref{LemmaChowVar0Dim}, the reference \cite[Lemma 2.4]{DillTorsionIsogenousAV}, discussions on the specialization argument and on the optimality of our formulation of Theorem \ref{MainThm2}. We thank Dan Abramovich and Philipp Habegger for comments on a draft of this paper and Abbey Bourdon for a relevant discussion on the consequence of our main result on rational points. We would like to thank the referee for their careful reading and helpful comments. 

ZG received funding from the European Research Council (ERC) under the European Union’s Horizon 2020 research and innovation programme (grant agreement n$^{\circ}$ 945714). TG received funding from NSF grant DMS-1759514 and DMS-2100548. LK received funding from the European Union’s Horizon 2020 research and innovation programme under the Marie Sklodowska-Curie grant agreement No. 101027237. A revision of this article was supported by the National Science Foundation under Grant No.\ DMS-1928930, while the three authors were in residence at the SL-Math 
 in Berkeley, California, during its semester on ``Diophantine Geometry'' in 2023.

\section{Preliminaries on abelian varieties}\label{SectionPreliminary}

In this section, we include some basic results on abelian varieties. Throughout the section, let $A$ be an abelian variety defined over $\IQbar$ of dimension $g$, where $\IQbar$ is an algebraic closure of $\mathbb{Q}$ in $\mathbb{C}$.

\subsection{Polarizations}\label{SubsectionPolarizations}
Let $\mathrm{Pic}(A)$ be the Picard group of $A$ and $\mathrm{Pic}^0(A)$ be the connected component of the identity. For $L\in \mathrm{Pic}(A)$, we denote its Chern class by $c_1(L)\in H^2(A(\mathbb{C}),\mathbb{Z})$. Let $A^\vee$ be the dual abelian variety of $A$. Then $\mathrm{Pic}^0(A) = A^\vee(\IQbar)$.

For each $a \in A(\IQbar)$, let $t_a \colon A \rightarrow A$ be the translation-by-$a$ map. Given $L \in \mathrm{Pic}(A)$, it induces a group homomorphism between dual abelian varieties $\phi_L \colon A \rightarrow A^\vee$ by sending $a\in A(\IQbar)$ to $t_a^*L \otimes L^{\otimes -1}\in \mathrm{Pic}^0(A)$. 

When $L$ is ample, the homomorphism $\phi_L$ is moreover an isogeny, \textit{i.e.}\ a surjective group homomorphism with finite kernel, in which case we say $\phi_L$ is a \textit{polarization} of $A$. The polarization is called a \textit{principal polarization} if $\phi_L$ is an isomorphism.

The polarization is uniquely determined by the Chern class of the line bundle by the following lemma \cite[Theorem 6.10]{Debarre}:

\begin{lemma}\label{LemmaPolarizationChern}
Let $L$ and $L'$ be two ample line bundles on $A$. Then $L$ and $L'$ define the same polarization if and only if $c_1(L) = c_1(L')$.
\end{lemma}

A \textit{polarized abelian variety} $(A,L)$ is defined as an abelian variety $A$ equipped with an ample line bundle $L$ serving as a representative of the Chern class $c_1(L)$.

Write $A(\mathbb{C}) = \mathbb{C}^g/\Lambda$ for some lattice $\Lambda \subseteq \mathbb{C}^g$ of rank $2g$. There is a canonical isomorphism between $H^2(A(\mathbb{C}),\mathbb{Z})$ and $\mathrm{Alt}^2(\Lambda,\mathbb{Z})$, the group of $\mathbb{Z}$-bilinear  alternating forms $\Lambda \times \Lambda \rightarrow \mathbb{Z}$. Thus $c_1(L)$ defines a $\mathbb{Z}$-bilinear  alternating form
\begin{equation}\label{EqRiemannForm}
E \colon \Lambda \times \Lambda \rightarrow \mathbb{Z}.
\end{equation}

As $L$ is ample, $c_1(L)$ is positive definite, and hence $E$ is non-degenerate. So  there exists a basis $(\gamma_1,\ldots,\gamma_{2g})$ of $\Lambda$ under which the matrix of $E$ is
\begin{equation}\label{EqPolarizationType}
\begin{bmatrix}
0 & D \\
-D & 0
\end{bmatrix}
\end{equation}
where $D = \mathrm{diag}(d_1, \ldots, d_g)$ with $d_1|\cdots|d_g$ positive integers; see \cite[Proposition 6.1]{Debarre}. Moreover, the  matrix $D$ is uniquely determined. We say $D$ is the \textit{polarization type of $(A,L)$}. Define the \textit{Pfaffian} by
\[
\mathrm{Pf}(L) := \mathrm{det}(D).
\]

\begin{lemma}\label{LemmaAbIsogPPAV}
Let $L$ be an ample line bundle on $A$. Then
\begin{enumerate}
\item[(i)] $\dim H^0(A,L) = \mathrm{Pf}(L)$;
\item[(ii)] $\deg_L (A) = g! \cdot \mathrm{Pf}(L)$;
\item[(iii)] if $f \colon A' \rightarrow A$ be an isogeny, then $\dim H^0(A',f^*L) = \deg(f) \dim H^0(A,L)$;
\item[(iv)] there exists an abelian variety $A_0$, an ample line bundle $L_0$ on $A_0$ defining a principal polarization, and an isogeny $u_0 \colon A \rightarrow A_0$ such that $L\cong u_0^*L_0$; moreover, $\deg(u_0) = \deg_L(A) / g!$.
\end{enumerate}
\end{lemma}
\begin{proof}
For (i), (ii) and (iii), it suffices to prove the assertions over $\mathbb{C}$. Then (i) is \cite[Corollary 3.2.8]{CAV}, (ii) is just the Riemann--Roch theorem \cite[Section 3.6]{CAV} and (iii) is \cite[Cororollary 6 to Proposition 6.12]{Debarre}. 
(iv) is \cite[pp.\ 234, Corollary 1 and its proof]{MumfordAbVar}.
\end{proof}

We often work with \textit{symmetric} ample line bundles for Néron-Tate heights. For this purpose, we need the following lemma.
\begin{lemma}\label{LemmaSymAmpleFirstChernTorsion}
Let $L$ and $L'$ be two symmetric ample line bundles on $A$. If $c_1(L') = c_1(L)$, then $\hat{h}_L = \hat{h}_{L'}$.
\end{lemma}
\begin{proof}
We have $c_1(L' \otimes L^{\otimes -1}) = c_1(L') - c_1(L) = 0$. So $L' \otimes L^{\otimes -1} \in \mathrm{Pic}^0(A) = A^\vee(\IQbar)$. The homomorphism $\phi_L \colon A \rightarrow A^\vee$, $a \mapsto t_a^*L \otimes L^{\otimes -1}$, is an isogeny because $L$ is ample. So there exists $a \in A(\IQbar)$ such that $\phi_L(a) = L' \otimes L^{\otimes -1}$. Thus $L' \cong t_a^*L$. 
Therefore for each $x \in A(\IQbar)$, we have
\begin{equation}\label{EqHeightTranslation}
\hat{h}_{L'}(x) = \hat{h}_{t_a^*L}(x) = \hat{h}_L(a+x). 
\end{equation}
Taking $x=0$,
$x \in A(\IQbar)_{\mathrm{tor}}$, we get $\hat{h}_L(a) = 0$ for the symmetric ample line bundles $L$ and $L'$. But then $a\in A(\IQbar)_{\mathrm{tor}}$ since $L$ is symmetric ample, and hence \eqref{EqHeightTranslation} yields $\hat{h}_{L'}(x) = \hat{h}_{L}(a+x)=\hat{h}_L(x)$ for all $x \in A(\IQbar)$.
\end{proof}

\subsection{Degree estimates}
The degree of a closed subvariety $X$ of $A$ with respect to an ample line bundle $L$ is defined as follows. If $X$ is irreducible, set $\deg_L X := c_1(L)^{\dim X}\cap [X]$ where $[X]$ is the cycle of $A$ given by $X$. For general $X$, set $\deg_L X := \sum_i \deg_L X_i$ where $X = \bigcup X_i$ is the decomposition into irreducible components. We use the notation from \cite[Chapters 1 and 2]{Fulton} freely in the following.

\begin{lemma}
Assume that $L$ is very ample. Let $Y$ and $Y'$ be irreducible subvarieties of $A$. Then 
\begin{equation}\label{EqDegreeSum}
\deg_L(Y+Y') \le 4^{\dim Y + \dim Y'} \deg_L Y \cdot \deg_L Y'.
\end{equation}
\end{lemma}
\begin{proof}

Write $p_i \colon A \times A \rightarrow A$, $i = 1, 2$, for the natural projection to the $i$-th factor. Then $L^{\boxtimes 2}:=p_1^*L \otimes p_2^*L$ is an ample line bundle on $A\times A$ and we have $c_1(L^{\boxtimes 2}) = p_1^*c_1(L) + p_2^*c_1(L)$. For readability, write $d$ and $d^\prime$ for the dimension of $Y$ and $Y^\prime$, respectively. By \cite[Prop.~2.5 and Rmk.~2.5.3]{Fulton}, it follows that
\begin{align}
	\label{eqn::intersection_numbers}
	\deg_{L^{\boxtimes 2}}(Y \times Y')
	&= c_1(L^{\boxtimes 2})^{d + d'} \cap [Y\times Y'] \\ \nonumber
	&= \sum_{i=0}^{d + d'}\binom{d + d'}{i} c_1(p_1^* L)^i \cap c_1(p_2^*L)^{d+d'-i} \cap [Y\times Y'] \\ \nonumber
	&= \sum_{i=0}^{d + d'}\binom{d + d'}{i} (c_1(p_1^* L)^i \cap [Y]) \times (c_1(p_2^*L)^{d+d'-i} \cap [Y']).
\end{align}
For reasons of dimension, the only non-vanishing term in this sum is  $i = d$. So
\begin{equation}\label{EqDegreeProduct}
\deg_{L^{\boxtimes 2}}(Y \times Y') = \binom{d + d'}{d}\deg_L Y \cdot \deg_L Y'.
\end{equation}

Consider the isogeny
\[
\alpha \colon A \times A \longrightarrow A \times A, \quad (x,y) \longmapsto (x+y, x-y),
\]
of degree $2^{2g}$. We recall that $c_1(\alpha^* L^{\boxtimes 2}) = 2c_1(L^{\boxtimes 2})$ by \cite[Prop.~A.7.3.3]{Hindry2000}. By \eqref{EqDegreeProduct}, it follows that
\begin{equation*}\label{EqDegreeProduct2}
\deg_{\alpha^* L^{\boxtimes 2}}(Y \times Y') =2^{d + d'} \binom{d + d'}{d}\deg_L Y \cdot \deg_L Y'\leq 4^{d+ d'}\deg_L Y \cdot \deg_L Y'.
\end{equation*}
We are ready to prove \eqref{EqDegreeSum}. Indeed, using $p_1(\alpha(Y \times Y')) = Y+Y'$ we obtain
\begin{align*}
	\deg_L(Y+Y') 
	\le \deg_{p_1^*L \otimes p_2^*L}(\alpha(Y\times Y')) 
	\le \deg_{\alpha^* p_1^*L \otimes \alpha^* p_2^*L}(Y\times Y') 
	= \deg_{\alpha^* L^{\boxtimes 2}}(Y \times Y');
\end{align*}
the first inequality follows for example from \cite[Lem.~2.4]{DillTorsionIsogenousAV} and the second one from the projection formula \cite[Prop.~2.5(c)]{Fulton}. We conclude the proof by combining the last two inequalities and using $\binom{d+d^\prime}{d} \leq 2^{d+d^\prime}$.
\end{proof}

\begin{lemma}\label{LemmaSmallestAbVarGenerated}
Let $X$ be an irreducible subvariety of $A$ and let $A'$ denote the abelian subvariety generated by $X-X$. Then $\deg_L A' \ll_g  \deg_L(X)^{2g}$.
\end{lemma}

\begin{proof} Replacing $L$ by $L^{\otimes 3}$, we may and do assume that $L$ is very ample. We write $r = \dim X$. 
	
The ascending chain
\begin{equation*}
	(X - X) \subseteq (X-X) + (X-X) \subseteq (X-X) + (X-X) + (X-X) \subseteq \cdots
\end{equation*}
of closed irreducible subvarieties becomes stationary as soon as two consecutive elements are equal. By dimension reasons, we infer that
\begin{equation*}
	A^\prime = \underbrace{(X-X)+(X-X)+\cdots+(X-X)}_{\text{$g$ copies}}
\end{equation*}
so the chain certainly becomes stationary after at most $g$ steps.

To conclude the proof, we claim the inequality
\begin{equation*}
\deg_L( \underbrace{(X-X) + \cdots + (X-X)}_{\text{$k$ copies}} ) \le 8^{2^k r} \deg_L(X)^{2g}
\end{equation*}
for every $k \geq 1$. For $k=1$, we apply \eqref{EqDegreeSum} to $Y = X$ and $Y' = -X$ and get
\[
\deg_L(X-X) \le 8^{r} (\deg_L X)^2.
\]
Similarly, we get
\begin{equation*}
	\deg_L( \underbrace{(X-X) + \cdots + (X-X)}_{\text{$k$ copies}} )
	\leq
		8^{kr} \cdot \deg_L( \underbrace{(X-X) + \cdots + (X-X)}_{\text{$(k-1)$ copies}} ) \cdot \deg_L(X - X)	 
\end{equation*}
for every integer $k \geq 2$. Combining these inequalities, we obtain
\begin{equation*}
	\deg_L(A^\prime) \leq 8^{(g+\dots + 2)r} \deg_L(X - X)^g < 8^{(g+1)^2r}\deg_L(X)^{2g},
\end{equation*}
whence the assertion of the lemma.
\end{proof}

\section{Hilbert schemes and non-degeneracy}\label{SectionBasicSetupNonDegeneracy}

Let $r \ge 1$ and $d \ge 1$ be two integers. In this section, we work over $\IQbar$, \textit{i.e.} all objects and morphisms are defined over $\IQbar$ unless stated otherwise.

Our goal is to introduce the \textit{restricted} Hilbert schemes \eqref{EqHilbSchemeOpenLocusFamily} and prove a non-degeneracy result, Proposition~\ref{PropNonDeg}. This is one of the main new ingredients to prove the uniform Mordell--Lang Conjecture for higher dimensional subvarieties of abelian varieties compared to the case of curves.

As we will work with Hilbert schemes, we make the following convention. All schemes are assumed to be separated and of finite type over the base. By a \textit{variety} defined over $\IQbar$, we mean a reduced scheme over $\IQbar$. Hence an \textit{integral scheme} (\textit{i.e.\ }a reduced irreducible scheme) is the same as an \textit{irreducible variety}.

For general information on Hilbert schemes and Hilbert polynomials, we refer to \cite{FGAHilb}, \cite[Chap.~9]{ACG:Curve} and \cite[$\mathsection$I.1]{kollar1999rational}.

\subsection{Parameterizing subvarieties of an abelian variety}\label{SubsectionHilbAbVar}
Let $A$ be an abelian variety and let $L$ be a very ample line bundle. All degrees below will be with respect to $L$.

The following result can be deduced from \cite[Thm.~2.1(b) and Lem.~2.4]{FGAHilb}, which itself is proved by comparing the Hilbert scheme and the Chow variety: There is a finite set $\Xi$ of polynomials such that the Hilbert polynomial of any irreducible subvariety of $A$ of dimension $r$ and degree $d$ is an element of $\Xi$. 
Indeed, in the notation of \cite[Thm.~2.1(b) and Lem.~2.4]{FGAHilb}, we can take $X = A$, $S = \mathrm{Spec}(\IQbar)$, $\mathcal O_X(1) = L$, $K = \IQbar$, $Y \subseteq A$ the irreducible subvarieties of degree $d$, $\mathfrak F$ the structure sheaf of such a subvariety $Y$, and $E$ the set of the classes of all these structure sheaves.

Let $\mathbf{H}_{r,d}(A) := \bigcup_{P\in \Xi} \mathbf{H}_P(A)$, where $\mathbf{H}_P(A)$ is the Hilbert scheme that parameterizes subschemes of $A$ with Hilbert polynomial $P$ and is constructed in \cite[Thm.~3.2]{FGAHilb}. As $\Xi$ is a finite set, $\mathbf{H}_{r,d}(A)$ is a projective variety over $\IQbar$.
By construction, there exists a universal family $\mathscr{X}_{r,d}(A) \rightarrow \mathbf{H}_{r,d}(A)$ 
endowed with a natural closed $\mathbf{H}_{r,d}(A)$-immersion
\begin{equation}\label{EqUnivFamily}
\xymatrix{
\mathscr{X}_{r,d}(A) \ \ar@{^(->}[r] \ar[rd] & A \times \mathbf{H}_{r,d}(A) \ar[d]^-{\pi_A} \\
& \mathbf{H}_{r,d}(A)
}
\end{equation}
where $\pi_A$ is the projection to the second factor. Over each point $s \in \mathbf{H}_{r,d}(A)(\IQbar)$, the fiber $\mathscr{X}_{r,d}(A)_s$ is precisely the subscheme of $A$ parametrized by $s$, and the horizontal immersion restricts to its natural closed immersion in $A$. 

By the definition of Hilbert schemes, the morphism $\pi_A|_{\mathscr{X}_{r,d}(A)}$ is flat and proper. By \cite[Thm.~12.2.4.(viii)]{EGAIV},
\begin{equation}\label{EqHilbOpenLocus}
\{s \in \mathbf{H}_{r,d}(A) : \mathscr{X}_{r,d}(A)_s\text{ is geometrically integral}\}
\end{equation}
is a Zariski open subset of $\mathbf{H}_{r,d}(A)$, which we endow with the induced subscheme structure. We write $\mathbf{H}_{r,d}(A)^{\circ}$ for the maximal reduced subscheme of this scheme. Note that $\mathbf{H}_{r,d}(A)^{\circ}$ is a quasi-projective variety defined over $\IQbar$. Its points $\mathbf{H}_{r,d}(A)^{\circ}(\IQbar)$ parametrize all irreducible $\IQbar$-subvarieties $X$ of $A$ with $\dim X = r$ and $\deg_L X = d$ (by our choice of $\Xi$ above). 
For each irreducible component $V$ of $\mathbf{H}_{r,d}(A)$, the intersection $V \cap \mathbf{H}_{r,d}(A)^{\circ}$ is either a dense Zariski open in $V$ or  empty.

\begin{lemma}\label{LemmaChowVar0Dim}
Let $V$ be a (not necessarily irreducible) subvariety of $\mathbf{H}_{r,d}(A)^{\circ}$. Then there exist $m_0 = m_0(V)$ points $P_1, \ldots, P_{m_0} \in A(\IQbar)$ such that the Zariski closed subset of $V$ defined by
\[
\{[X] \in V(\IQbar) : P_1 \in X(\IQbar), \ldots, P_{m_0} \in X(\IQbar) \}
\]
has dimension $0$, \textit{i.e.}\ is a non-empty finite set.
\end{lemma}
\begin{proof}
Fix $[X_0] \in V(\IQbar)$. For each $k \in \{0, \ldots, \dim V\}$, we claim the following statement: \textit{There are finitely many points $P_{k,1},\ldots,P_{k, n_k}$ in $X_0(\IQbar)$ such that $$V_k := \{[X] \in V(\IQbar) : P_{k,1} \in X(\IQbar), \ldots, P_{k,n_k} \in X(\IQbar) \}$$ has dimension $\le \dim V-k$.}

\medskip

Taking $k = \dim V$, this claim immediately yields the lemma; the set $V_{\dim V}$ thus obtained is non-empty since it contains $[X_0]$.

We prove this claim by induction on $k$. The base step $k = 0$ trivially holds true because we can take any finite set of points in $X_0(\IQbar)$.

Assume the claim holds true for $0, \ldots, k-1$. We have thus obtained points $P_{k-1,1},\ldots,$ $P_{k-1, n_{k-1}} \in X_0(\IQbar)$ such that $$V_{k-1} := \{[X] \in V(\IQbar): P_{k-1, i} \in X(\IQbar)\text{ for all }i=1,\ldots,n_{k-1} \}$$ has dimension $\le \dim V-k+1$.

For each irreducible component $W$ of $V_{k-1}$ with $\dim W > 0$, there exists some $[X_W] \in W(\IQbar)$ such that $X_W \not= X_0$ as (irreducible) subvarieties of $A$. Take $P_W \in (X_0 \setminus X_W)(\IQbar)$. Then $\{[X] \in W(\IQbar) : P_W \in X(\IQbar)\}$ has dimension $\le \dim W - 1 \le \dim V-k+1-1 = \dim V -k$.

Thus it suffices to take $\{P_{k,1},\ldots,P_{k,n_k}\} := \{P_{k-1,1},\ldots,$ $P_{k-1, n_{k-1}}\} \cup (\bigcup_W \{P_W\})$ with $W$ running over all positive dimensional irreducible components of $V_{k-1}$.
\end{proof}

For each $m \ge 1$, set $\mathscr{X}_{r,d}(A)^{[m]} := \mathscr{X}_{r,d}(A) \times_{\mathbf{H}_{r,d}(A)} \ldots \times_{\mathbf{H}_{r,d}(A)} \mathscr{X}_{r,d}(A)$. Then \eqref{EqUnivFamily} induces
\begin{equation}\label{EqUnivFamilyPower}
\xymatrix{
\mathscr{X}_{r,d}(A)^{[m]} \times_{\mathbf{H}_{r,d}(A)} \mathbf{H}_{r,d}(A)^{\circ} \ \ar@{^(->}[r] \ar[rd] & A^m \times \mathbf{H}_{r,d}(A)^{\circ} \ar[d]^-{\pi_A^{[m]}}  \ar[r]^-{\iota_A^{[m]}} & A^m \\
& \mathbf{H}_{r,d}(A)^{\circ}
}
\end{equation}
where $\iota_A^{[m]}$ is the natural projection. 

\begin{lemma}\label{LemmaGenericallyFiniteFiber}
Let $V$ be a (not necessarily irreducible) subvariety of $\mathbf{H}_{r,d}(A)^{\circ}$. Then there exists $m_0 = m_0(V) \ge 1$ with the following property. For each $m \ge m_0$, there exists $\underline{P} \in A^m(\IQbar)$ such that 
$$(\iota_A^{[m]}|_{\mathscr{X}_{r,d}(A)^{[m]}\times_{\mathbf{H}_{r,d}(A)}V})^{-1}(\underline{P})$$ has dimension $0$, \textit{i.e.}\ is a non-empty finite set.
\end{lemma}
\begin{proof}
Let  $P_1,\ldots,P_{m_0} \in A(\IQbar)$ be from Lemma~\ref{LemmaChowVar0Dim}. For each $k \geq m_0$, we set $P_k = P_{m_0}$.

Set $\underline{P} = (P_1,\ldots,P_m) \in A^m(\IQbar)$. To prove the lemma, it suffices to prove that $(\iota_A^{[m]}|_{\mathscr{X}_{r,d}(A)^{[m]}\times_{\mathbf{H}_{r,d}(A)}V})^{-1}(\underline{P})$ is a non-empty finite set.

It is clear that $\pi_A^{[m]}|_{\{\underline{P}\} \times V}$ is an isomorphism. Therefore, $\pi_A^{[m]}$ induces an isomorphism
\[
(\{\underline{P}\} \times V) \cap \mathscr{X}_{r,d}(A)^{[m]} \cong \pi_A^{[m]}\left((\{\underline{P}\} \times V) \cap \mathscr{X}_{r,d}(A)^{[m]} \right).
\]
The left hand side of this isomorphism is precisely $(\iota_A^{[m]}|_{\mathscr{X}_{r,d}(A)^{[m]}\times_{\mathbf{H}_{r,d}(A)}V})^{-1}(\underline{P})$. So it suffices to prove that the right hand side of this isomorphism is a non-empty finite set. A direct computation shows that the right hand side is
\[
\{[X] \in V(\IQbar) : P_1 \in X(\IQbar),  \ldots, P_{m_0} \in X(\IQbar) \},
\]
which is non-empty and finite by  Lemma~\ref{LemmaChowVar0Dim}.
\end{proof}

\subsection{Parameterizing subvarieties in families}\label{SubsectionHilbFamily}
We next extend the setting of the previous subsection to the case of families of abelian varieties.
Let $\pi^{\mathrm{univ}} \colon \mathfrak{A}_g \rightarrow \mathbb{A}_g$ be the universal abelian variety over the fine moduli space of principally polarized abelian varieties of dimension $g$ with level-$4$-structure. For each $b \in \mathbb{A}_g(\IQbar)$, the abelian variety parametrized by $b$ is $(\mathfrak{A}_g)_b = (\pi^{\mathrm{univ}})^{-1}(b)$.

We fix a symmetric relatively ample line bundle $\mathfrak{L}_g$ on $\mathfrak{A}_g / \mathbb{A}_g$ satisfying the following property: for each principally polarized abelian variety $(A,L)$ parametrized by $b \in \mathbb{A}_g(\IQbar)$, we have $c_1(\mathfrak{L}_g|_{(\mathfrak{A}_g)_b}) = 2c_1(L)$; see \cite[Prop.~6.10]{MFK:GIT94} for the existence of $\mathfrak{L}_g$. 
The line bundle $\mathfrak{L}_g^{\otimes 4}$ is relatively very ample on $\mathfrak{A}_g / \mathbb{A}_g$.

Again by \cite[Thm.~2.1(b) and Lem.~2.4]{FGAHilb}, there are finitely many possibilities for the Hilbert polynomials of irreducible subvarieties of $(\mathfrak{A}_g)_b$ (for all $b \in \mathbb{A}_g(\IQbar)$) of dimension $r$ and degree $d$ with respect to $\mathfrak{L}_g^{\otimes 4}|_{(\mathfrak{A}_g)_b}$.\footnote{In the notation of \cite[Thm.~2.1(b) and Lem.~2.4]{FGAHilb}, $X = \mathfrak{A}_g$, $S= \mathbb{A}_g$, $\mathcal O_X(1) = \mathfrak{L}_g^{\otimes 4}$, $K = \IQbar$, $Y$ the irreducible subvarieties of degree $d$, $\mathfrak F$ the structure sheaf of such a subvariety $Y$, and $E$ the set of the classes of all these structure sheaves.} Let $\Xi$ be this finite set of polynomials.

Consider the $\mathbb{A}_g$-scheme
\begin{equation}\label{EqHilbSch}
\mathbf{H} := \bigsqcup_{P \in \Xi} \mathbf{H}_P(\mathfrak{A}_g / \mathbb{A}_g) \xrightarrow{\iota_{\mathbf{H}}} \mathbb{A}_g
\end{equation}
with $\mathbf{H}_P(\mathfrak{A}_g / \mathbb{A}_g)$ the $\mathbb{A}_g$-scheme representing the functor from $\mathbb{A}_g$-schemes to sets associating with an $\mathbb{A}_g$-scheme $T$ the set of proper subschemes $W \subseteq \mathfrak{A}_g\times_{\mathbb{A}_g}T$ that are flat over $T$ and have Hilbert polynomial $P$. Each $\mathbf{H}_P(\mathfrak{A}_g / \mathbb{A}_g)$ is projective 
over $\mathbb{A}_g$. Since $\Xi$ is a finite set, the structure morphism $\iota_{\mathbf{H}}$ is of finite type.

Similarly to the previous subsection, there is a universal family endowed with a closed $\mathbf{H}$-immersion 
\begin{equation}\label{EqUnivFamilyOverHilbRel}
\xymatrix{
\mathscr{X} := \mathscr{X}_{r,d}(\mathfrak{A}_g / \mathbb{A}_g)  \ \ar@{^(->}[r] \ar[rd] & \mathfrak{A}_g \times_{\mathbb{A}_g} \mathbf{H} =: \cA 
\ar[d]^-{\pi} \\
& \mathbf{H}
}
\end{equation}
with $\pi|_{\mathscr{X}}$ flat and proper.  

\medskip

To ease notation, we set $\mathfrak{S}_S := \mathfrak{S} \times_{\mathbf{H}} S$ for every $\mathbf{H}$-scheme $\mathfrak{S} \rightarrow \mathbf{H}$ and every morphism $S \rightarrow \mathbf{H}$. In particular, we do so for $\mathscr{X}$ and $\cA$ as well as for their fibered self-products over $\mathbf{H}$. If $S$ is a variety, then $\cA_S \rightarrow S$ is an abelian scheme, and $\mathscr{X}_S$ is a subvariety of $\cA_S$ which dominates $S$.

\medskip

As for \eqref{EqHilbOpenLocus}, the set
\begin{equation}\label{EqHilbSchemeOpenLocusFamily}
\{s \in \mathbf{H}: \mathscr{X}_{r,d}(\mathfrak{A}_g / \mathbb{A}_g)_s\text{ is geometrically integral}\}
\end{equation}
is Zariski open in $\mathbf{H}$ and hence has an induced structure as an open subscheme. We write $\mathbf{H}_{r,d}(\mathfrak{A}_g / \mathbb{A}_g)^{\circ}$ for the maximal reduced subscheme of this subscheme, which is a quasi-projective  variety defined over $\IQbar$. In the sequel, $\mathbf{H}_{r,d}(\mathfrak{A}_g / \mathbb{A}_g)^{\circ}$ is called the \textit{restricted Hilbert scheme}. In addition, we denote it simply by $\mathbf{H}^{\circ}$ if $r$, $d$ and $g$ are clear from context.

By our choice of $\Xi$ and the construction \eqref{EqHilbSch}, $\mathbf{H}^{\circ}(\IQbar)$ parametrizes all pairs $((A,L),X)$ consisting of a principally polarized abelian variety defined over $\IQbar$ and an irreducible subvariety $X$ of $A$ defined over $\IQbar$ with $\dim X = r \ge 1$ and $\deg_{\mathfrak{L}_g^{\otimes 4}|_A} X = d$.\footnote{Here $(A,L)$ gives rise to a point $b \in \mathbb{A}_g(\IQbar)$, and we identify $A$ with $(\mathfrak{A}_g)_b$.} 
For each irreducible component $S$ of $\mathbf{H}$, the intersection $S \cap \mathbf{H}^{\circ}$ is either a dense Zariski open set in $S$ or empty. For convenience, we set $\iota_{\mathbf{H}^{\circ}} := \iota_{\mathbf{H}}|_{\mathbf{H}^{\circ}}$ for the restriction of the structural morphism. The fiber $\iota_{\mathbf{H}^{\circ}}^{-1}(b)$ is precisely the subvariety $\mathbf{H}_{r,d}(A_b)^{\circ}$ defined in the previous subsection.

\medskip

For each $m \ge 1$, let $\mathscr{X}_{\mathbf{H}^{\circ}}^{[m]}$ (resp. $\cA_{\mathbf{H}^{\circ}}^{[m]}$) denote the $m$-fold fibered power of $\mathscr{X}_{\mathbf{H}^{\circ}}$ over $\mathbf{H}^{\circ}$ (resp.\ of $\cA_{\mathbf{H}^{\circ}}$ over $\mathbf{H}^{\circ}$).  
There is a commutative diagram 
\begin{equation}\label{EqHilbFiberPowerToUse}
\xymatrix{
\mathscr{X}^{[m]}_{\mathbf{H}^{\circ}} \ \ar@{^(->}[r] \ar[rd] & 
\cA^{[m]}_{\mathbf{H}^{\circ}} \ar[d]^-{\pi^{[m]}} \ar[r]^-{\iota^{[m]}} \pullbackcorner & \mathfrak{A}_g^{[m]} \ar[d]^-{\pi^{\mathrm{univ}, [m]}}  \\
& \mathbf{H}^{\circ} \ar[r]^-{\iota_{\mathbf{H}^{\circ}}} & \mathbb{A}_g.
}
\end{equation}
Taking fibers in this diagram over a point $b \in \mathbb{A}_g(\IQbar)$, we obtain again the diagram \eqref{EqUnivFamilyPower} for the abelian variety $A = (\mathfrak{A}_g)_b$. 

\begin{lemma}\label{LemmaGenericFiniteHilb}
Let $S$ be an irreducible subvariety of $\mathbf{H}^{\circ}$. Then there exists $m_0  = m_0(S) \ge 1$ such that the restriction of $\iota^{[m]}$ to $\mathscr{X}_{S}^{[m]}$ is generically finite onto its image for every $m \ge m_0$.
\end{lemma}
\begin{proof}
Choose some $b \in \iota_{\mathbf{H}^{\circ}}(S)(\IQbar)$. We use the notation from \eqref{EqUnivFamilyPower} with $A = A_b$. For any $\underline{P} \in (\pi^{\mathrm{univ}, [m]})^{-1}(b)(\IQbar) = A_b^m(\IQbar)$, we have
\begin{equation}\label{EqInverseImageFamilyFiber}
(\iota^{[m]})^{-1}(\underline{P}) \cap \mathscr{X}_S^{[m]} = (\iota_{A_b}^{[m]}|_{\mathscr{X}_{r,d}(A_b)^{[m]} \times_{\mathbf{H}_{r,d}(A_b)} V})^{-1}(\underline{P}),
\end{equation}
where $V = (\iota_{\mathbf{H}^{\circ}}|_{S})^{-1}(b) \subseteq \iota_{\mathbf{H}^{\circ}}^{-1}(b) = \mathbf{H}_{r,d}(A_b)^{\circ}$. 
Thus there exists $\underline{P}$ such that $(\iota^{[m]}|_{\mathscr{X}^{[m]}_S})^{-1}(\underline{P})$ has dimension $0$ by Lemma~\ref{LemmaGenericallyFiniteFiber}. 

The asserted generic finiteness follows if  ${\mathscr{X}^{[m]}_S}$ is irreducible. To prove this irreducibility, we use that $\mathscr{X}_{\mathbf{H}^\circ} \rightarrow \mathbf{H}^\circ$ is flat by construction, and so are the morphisms $\mathscr{X}_{\mathbf{H}^\circ}^{[m]} \rightarrow \mathbf{H}^\circ$ and ${\mathscr{X}^{[m]}_S} \rightarrow S$. Openness of flat morphisms \cite[Thm.~2.4.6]{EGAIV} implies that every irreducible component of ${\mathscr{X}^{[m]}_S}$ dominates $S$, \textit{i.e.}\ its projection contains the generic point of $S$. But since $\mathscr{X}_s$ is geometrically irreducible for all $s \in \mathbf{H}^{\circ}(\IQbar)$, there can be only one such irreducible component of ${\mathscr{X}^{[m]}_S}$ by \cite[Prop.~9.7.8]{EGAIV}.
\end{proof}

\subsection{Non-degeneracy}
We keep the notation from the previous subsection, in particular those in \eqref{EqUnivFamilyOverHilbRel}. For later applications of both the height inequality of  \cite[Thm.~1.6 and B.1]{DGHUnifML} and the equidistribution result \cite[Thm.~1]{Kuehne:21}, we need to work with appropriate non-degenerate subvarieties in certain families of abelian varieties. Here, we give the main proposition providing such non-degenerate subvarieties, inspired by a result of the second-named author \cite[Prop.~3.4]{GeQuadraticPoints} and using a theorem of the first-named author \cite[Thm.~10.1]{GaoBettiRank}.

\begin{prop}\label{PropNonDeg}
Let $S \subseteq \mathbf{H}^{\circ}$ be an irreducible subvariety and consider $\mathscr{X}_S \subseteq \cA_S \rightarrow S$. Writing $\bar{\eta}$ for the geometric generic point of $S$, assume that the following hypotheses hold true:
\begin{enumerate}
\item[(i)] $\mathscr{X}_{\bar{\eta}}$ is an irreducible subvariety of $\cA_{\bar{\eta}}$;
\item[(ii)] $\mathscr{X}_s$ generates $\cA_s$ for each $s \in S(\mathbb{C})$;
\item[(iii)] the subvariety $\mathscr{X}_{\bar{\eta}}$ has finite stabilizer.
\end{enumerate}
Then there exists $m_0 = m_0(S) \ge 1$ such that $\mathscr{X}^{[m]}_S$ is a non-degenerate subvariety of $\cA_S^{[m]}$ for each $m \ge m_0$.
\end{prop}

Non-degenerate subvarieties over $\mathbb{C}$ are defined as in 
\cite[Defn.~1.5 and B.4]{DGHUnifML}. 
A subvariety defined over $\IQbar$ is said to be non-degenerate if its base change to $\mathbb{C}$ is non-degenerate.

\begin{proof}[Proof of Proposition \ref{PropNonDeg}]
We wish to invoke \cite[Thm.~10.1.(i)]{GaoBettiRank} with $t=0$. 
In fact, as the conventions of \cite{GaoBettiRank} deviate a litte bit from standard terminology, we apply it in the formulation of \cite[Thm.~6.5.(i)]{GaoSurveyUML}.

All hypotheses of \cite[Thm.~6.5]{GaoSurveyUML} are satisfied: indeed, hypothesis (a) clearly holds true because we fixed $r \ge 1$ at the beginning of this section, hypothesis (b) holds true by (ii), and hypothesis (c) holds true by (iii). Thus we can apply \cite[Thm.~6.5.(i)]{GaoSurveyUML} to the abelian scheme $\cA_S \rightarrow S$ and the subvariety $\mathscr{X}_S$. It yields that $\mathscr{X}^{[m]}_S$ is non-degenerate if $m \ge \dim S$ and $\iota^{[m]}|_{\mathscr{X}^{[m]}_S}$ is generically finite. By Lemma~\ref{LemmaGenericFiniteHilb} there exists $m_0 \ge 1$ such that $\iota^{[m]}|_{\mathscr{X}^{[m]}_S}$ is generically finite for each $m \ge m_0$. So the conclusion holds true with $m_0$ replaced by $\max\{m_0, \dim S\}$.
\end{proof}

We conclude with a further lemma useful in later sections.

\begin{lemma}\label{LemmaConstructibleSet}
For a (not necessarily irreducible) subvariety $S \subseteq \mathbf{H}^{\circ}$, we define the subset 
\small
\begin{align}
\label{EqConstructibleSet}
S_{\mathrm{gen}}:=\{ s\in S(\IQbar): \text{the} & \text{ stabilizer of } \mathscr{X}_s \text{ in }\cA_s \text{ is finite, and } \mathscr{X}_s\text{ generates }\cA_s\}. 
\end{align}
\normalsize
Endow each irreducible component $S'$ of $\overline{S_{\mathrm{gen}}}$ (the Zariski closure of $S_{\mathrm{gen}}$ in $S$) with the reduced induced subscheme structure. Then $\mathscr{X}_{S'} \subseteq \cA_{S'} \rightarrow S'$ satisfies the hypotheses (i)--(iii) of Proposition~\ref{PropNonDeg}.
\end{lemma}
\begin{proof} 
Let $\bar{\eta}$ be the geometric generic point of $S'$.

To verify the condition (i), we can use again \cite[Prop.~9.7.8]{EGAIV} as in the proof of Lemma \ref{LemmaGenericFiniteHilb} above. As there, we simply use the fact that $\mathscr{X}_{s'}$ is (geometrically) irreducible for all $s' \in \mathbf{H}^{\circ}(\IQbar)$.

The condition (ii) clearly holds true, for example by Chevalley's Theorem on the upper semicontinuity of the fiber dimension \cite[Thm.~13.1.3]{EGAIV}.

To prove (iii), let $C$ be the identity component of $\mathrm{Stab}_{\cA_{\bar{\eta}}}(\mathscr{X}_{\bar{\eta}})$. There exists a quasi-finite dominant morphism $\rho \colon S'' \rightarrow S'$ such that $C$ extends to an abelian subscheme $\cC$ of $\cA_{S''} = \cA_{S'} \times_{S'} S'' \rightarrow S''$. Clearly, $\cC_{s''} \subseteq \mathrm{Stab}_{\cA_{s''}}(\mathscr{X}_{s''})$ for each $s'' \in S(\IQbar)$. Since $\rho(S'')$ contains a Zariski open dense subset of $S'$, there exists a point $s' \in \rho(S'')(\IQbar) \cap S_{\mathrm{gen}}$ as in the previous paragraph. 
Choosing a point $s'' \in S''(\IQbar)$ over $s' \in S'(\IQbar)$, the fiber $\cA_{s''}$ is naturally identified with $\cA_{s'}$, which also induces an identification of $\mathscr{X}_{s''}$ with $\mathscr{X}_{s'}$. 
It follows that $\cC_{s''} \subseteq \mathrm{Stab}_{\cA_{s''}}(\mathscr{X}_{s''}) = \mathrm{Stab}_{\cA_{s'}}(\mathscr{X}_{s'})$. By the definition of $S_{\mathrm{gen}}$ \eqref{EqConstructibleSet}, $\mathrm{Stab}_{\cA_{s'}}(\mathscr{X}_{s'})$ is finite and hence $\cC_{s''}$ is the identity element of $\cA_{s''}$. Thus $\cC$ is the zero section of $\cA_{S''} \rightarrow S''$, and $C = \cC_{\bar{\eta}}$ is the origin of $\cA_{\bar{\eta}}$, whence (iii). 
\end{proof}

\section{Applying the height inequality}\label{SectionAppHtIneq}

\begin{prop}\label{PropHtIneqStep1}
Let $g$, $l$, $r\le g$ and $d$ be positive integers. There exist constants $c_1' = c_1'(g,l,r,d) > 0$, $c_2' = c_2'(g,l,r,d) > 0$ and $c_3' = c_3'(g,l,r,d)>0$ satisfying the following property: For 
\begin{itemize}
\item every polarized abelian variety $(A,L)$ over $\IQbar$ of dimension $g$ and degree $\deg_L A = l$, and
\item every irreducible subvariety $X$ of $A$ with $\dim X = r$, $\deg_L X = d$ and such that $X$ generates $A$, 
\end{itemize}
the set
\begin{equation}
\left\{ x \in X^{\circ}(\IQbar) : \hat{h}_{L\otimes [-1]^*L}(x) \le c_1'\max\{1, h_{\mathrm{Fal}}(A)\} - c_3' \right\}, 
\end{equation}
is contained in $X'(\IQbar)$ for a proper Zariski closed $X' \subsetneq X$ with $\deg_L(X') < c_2'$.
\end{prop}

For the proof of the proposition, we start by providing a similar result (Proposition~\ref{PropHtIneqStep1Aux}) for the universal families over the \textit{restricted} Hilbert schemes constructed in $\mathsection$\ref{SubsectionHilbFamily}. Next, we deduce Proposition~\ref{PropHtIneqStep1} in the case where $(A,L)$ is principally polarized. Finally, we handle the polarization type and prove the full Proposition~\ref{PropHtIneqStep1}.

\subsection{A preliminary result for restricted Hilbert schemes}\label{SubsectionHtIneqFamilies}
We use the notation from $\mathsection$\ref{SubsectionHilbFamily}. In particular, we have a symmetric relatively ample line bundle $\mathfrak{L}_g$ on the universal abelian variety $\mathfrak{A}_g \rightarrow \mathbb{A}_g$ 
with level-$4$-structure 
satisfying the following property: for each principally polarized abelian variety $(A,L)$ parametrized by $b \in \mathbb{A}_g(\IQbar)$, we have $c_1(\mathfrak{L}_g|_{(\mathfrak{A}_g)_b}) = 2c_1(L)$. 
If we identify $A = (\mathfrak{A}_g)_b$, then $c_1(\mathfrak{L}_g|_A) = c_1(L \otimes [-1]^*L)$.

For each $1\le r \le g$ and each $d\ge 1$, we have the  \textit{restricted} Hilbert scheme $\mathbf{H}^\circ:=\mathbf{H}_{r,d}(\mathfrak{A}_g / \mathbb{A}_g)^{\circ}$  from 
 \eqref{EqHilbSchemeOpenLocusFamily}. Recall that it is a variety defined over $\IQbar$ whose closed points parametrizes all pairs $(X, (A,L))$ of a principally polarized abelian variety $(A,L)$ and an irreducible subvariety $X$ of $A$ with $\dim X = r$ and $\deg_{\mathfrak{L}_g^{\otimes 4}|_A}X = d$.

For $\cA_{\mathbf{H}^{\circ}} := \mathfrak{A}_g \times_{\mathbb{A}_g} \mathbf{H}^{\circ}$, we have a commutative diagram (\eqref{EqHilbFiberPowerToUse} with $m=1$)
\begin{equation}\label{EqSubsectionHtIneqFamiliesHilbDiagram}
\xymatrix{
\mathscr{X}_{\mathbf{H}^{\circ}} \ar@{^(->}[r] \ar[rd] & 
\cA_{\mathbf{H}^{\circ}} \ar[d]^-{\pi} \ar[r]^-{\iota} \pullbackcorner & \mathfrak{A}_g \ar[d]^-{\pi^{\mathrm{univ}}}  \\
& \mathbf{H}^{\circ} \ar[r]^-{\iota_{\mathbf{H}^{\circ}}} & \mathbb{A}_g
}
\end{equation}
where $\pi|_{\mathscr{X}_{\mathbf{H}^{\circ}}} \colon \mathscr{X}_{\mathbf{H}^{\circ}} \rightarrow \mathbf{H}^{\circ}$ is the universal family. Every object in the diagram \eqref{EqSubsectionHtIneqFamiliesHilbDiagram} is a \textit{variety} over $\IQbar$, and every morphism is defined over $\IQbar$.

Set $\cL := \iota^*\mathfrak{L}_g^{\otimes 4}$. Notice that $\cL$ is relatively very ample with respect to $\cA_{\mathbf{H}^{\circ}} \rightarrow \mathbf{H}^{\circ}$.

Fix an ample line bundle $\mathfrak{M}$ on a compactification $\overline{\mathbb{A}_g}$ of $\mathbb{A}_g$. 
The morphism $\iota_{\mathbf{H}^{\circ}}$ extends to a morphism $\overline{\iota_{\mathbf{H}^{\circ}}} \colon \overline{\mathbf{H}^{\circ}} \rightarrow \overline{\mathbb{A}_g}$ for some compactification $\overline{\mathbf{H}^{\circ}}$ of $\mathbf{H}^{\circ}$. We can do these constructions such that $\overline{\mathbf{H}^{\circ}}$ is a projective variety defined over $\IQbar$ which contains $\mathbf{H}^{\circ}$ as a Zariski open dense subset. Set $\cM := \overline{\iota_{\mathbf{H}^{\circ}}}^*\mathfrak{M}$.

\begin{prop}\label{PropHtIneqStep1Aux}
Let $S \subseteq \mathbf{H}^{\circ}$ be a  (not necessarily irreducible) subvariety defined over $\IQbar$. Let $S_{\mathrm{gen}} \subseteq S(\IQbar)$ be the subset defined by \eqref{EqConstructibleSet}.

There exist constants $c_1' = c_1'(r,d,S)>0$, $c_2' = c_2'(r,d,S) >0$ and $c_3'= c_3'(r,d,S) >0$ with the following property. 
For each $s \in S_{\mathrm{gen}}$, the set
\begin{equation}\label{EqNGPHtIneq}
\Sigma_s := \left\{ x \in \mathscr{X}_s^{\circ} : \hat{h}_{\cL}(x) \le c_1'\max\{1, h_{\overline{\mathbf{H}^{\circ}}, \cM}(s)\} - c_3' \right\}
\end{equation}
is contained in $X'(\IQbar)$ for a proper Zariski closed $X' \subsetneq \mathscr{X}_s$ with $\deg_{\cL_s}(X') < c_2'$.
\end{prop}

This proposition is a generalization of the corresponding result for the universal curve \cite[Prop.~7.1]{DGHUnifML}, though our formulation here is closer to \cite[Prop.~7.2]{GaoSurveyUML}. 
Notice  that the case $S=\mathbf{H}^\circ$ contains the full strength of the proposition. Furthermore, it is the only case we use below in the proof of Proposition \ref{PropHtIneqPPAV}. However, the above formulation is more convenient for the proof by induction.

\begin{proof}[Proof of Proposition \ref{PropHtIneqStep1Aux}]
It suffices to prove the proposition for $S$ irreducible, which we assume from now on.

For the proof, we use an induction on $\dim S$. The base case $\dim S = 0$ is straightforward as we may choose $c_3^\prime$ large enough to render $\Sigma_s$ empty.

Endow each irreducible component $S'$ of $\overline{S_{\mathrm{gen}}}$ with the reduced induced subscheme structure.  
Lemma~\ref{LemmaConstructibleSet} allows us to invoke Proposition~\ref{PropNonDeg} for $\mathscr{X}_{S'} \subseteq \cA_{S'} \rightarrow S'$. 
Let $m_0(S') > 0$ be as in Proposition~\ref{PropNonDeg}, and let $m = \max_{S'}\{m_0(S')\}$ with $S'$ running over all irreducible components of $S_{\mathrm{gen}}$. Then $\mathscr{X}^{[m]}_{S'}$ is a non-degenerate subvariety of $\cA_{S'}^{[m]}$. Note that $m = m(S) > 0$ depends only on $S$. 

Let $\mathscr{X}^{[m],*}_{S'}$ be $\mathscr{X}^{[m]}_{S'}$ deprived of its $0$-th degeneracy locus; a definition of the $0$-th degeneracy can be found in \cite[Defn.~1.6]{GaoBettiRank}. By \cite[Thm.~1.8]{GaoBettiRank}, it is a Zariski open subset of $\mathscr{X}^{[m]}_{S'}$. It is defined over $\IQbar$ since $\mathscr{X}^{[m]}_{S'}$ is and it is dense in $\mathscr{X}^{[m]}_{S'}$ because of the non-degeneracy of $\mathscr{X}^{[m]}_{S'}$.

Consider the abelian scheme $\pi^{[m]} \colon \cA_{\mathbf{H}^{\circ}}^{[m]} \rightarrow \mathbf{H}^{\circ}$.

Consider $S' \setminus \pi^{[m]}(\mathscr{X}_{S'}^{[m],*})$ endowed with the reduced induced subscheme structure; it has dimension $\le \dim S' - 1 \le \dim S - 1$. Let $S_1,\ldots,S_k$ be the irreducible components of $\bigcup_{S'} S' \setminus \pi^{[m]}(\mathscr{X}_{S'}^{[m],*})$ with $S'$ running over all irreducible components of $\overline{S_{\mathrm{gen}}}$. The set $\{S_1,\ldots,S_k\}$ is uniquely determined by $S$ and $m$. As $m = m(S) > 0$ depends only on $S$, it is actually already determined by $S$.

Let $s \in S_{\mathrm{gen}}(\IQbar)$. Then either $s \in \bigcup_{i=1}^k S_i(\IQbar)$ or $s \in \pi^{[m]}(\mathscr{X}_{S'}^{[m],*})(\IQbar)$ for some irreducible component $S'$ of $\overline{S_{\mathrm{gen}}}$. For the remaining proof of the proposition, we distinguish between these two cases.

\smallskip
\noindent\underline{\textbf{Case (i):} $s \in \bigcup_{i=1}^k S_i(\IQbar)$.}
\smallskip

For each $i \in \{1,\ldots,k\}$, we can apply the induction hypothesis to $S_i$ as $\dim S_i \le \dim S - 1$. In this way, we obtain constants $c_{i,1}' = c_{i,1}'(r,d,S_i) > 0$, $c_{i,2}' = c_{i,2}'(r,d,S_i) > 0$, and $c_{i,3}' = c_{i,3}'(r,d,S_i) >0$ such that for each $s \in S_i(\IQbar)$, the set
\begin{equation}
\Sigma_{i,s}:=\left\{ x \in \mathscr{X}_s^{\circ}(\IQbar) : \hat{h}_{\cL}(x) \le c_{i,1}'\max\{1, h_{\overline{\mathbf{H}^{\circ}}, \cM}(s)\} - c_{i,3}' \right\}
\end{equation}
is contained in $X'(\IQbar)$ for some proper Zariski closed $X' \subsetneq \mathscr{X}_s$ with $\deg_{\cL_s}(X') < c_{i,2}'$.

Let $c_{\mathrm{deg},1}' = \min_{1\le i\le k} \{c_{i,1}'\} > 0$, $c_{\mathrm{deg},2}' = \max_{1\le i\le k} \{c_{i,2}'\}> 0$ and $c_{\mathrm{deg},3}' = \max_{1\le i\le k} \{c_{i,3}'\}>0$. 
As the set $\{S_1,\ldots,S_k\}$ is uniquely determined by $S$, the constants $c_{\mathrm{deg},1}'$, $c_{\mathrm{deg},2}'$ and $c_{\mathrm{deg},3}'$ depend only on $g$, $d$, $r$ and $S$. Moreover, for each $s \in \bigcup_{i=1}^k S_i(\IQbar)$, the set
\begin{equation}\label{EqDegHtIneqSet}
\Sigma_{\mathrm{deg},s} := \left\{ x \in \mathscr{X}_s^{\circ}(\IQbar) : \hat{h}_{\cL}(x) \le c_{\mathrm{deg},1}'\max\{1, h_{\overline{\mathbf{H}^{\circ}}, \cM}(s)\} - c_{\mathrm{deg},3}' \right\}
\end{equation}
must be contained in $\Sigma_{i,s}$ for some $i$. So $\Sigma_{\mathrm{deg},s}$ is contained in $X'(\IQbar)$ for some proper Zariski closed $X' \subsetneq \mathscr{X}_s$ with $\deg_{\cL_s}(X') \le c_{\mathrm{deg},2}'$. This concludes the induction step in Case (i). 

\smallskip
\noindent\underline{\textbf{Case (ii):}
$s \in \pi^{[m]}(\mathscr{X}_{S'}^{[m],*})(\IQbar)$ for some irreducible component $S'$ of $\overline{S_{\mathrm{gen}}}$.}
\smallskip

By the height inequality \cite[Thm.~1.6]{DGHUnifML} (in the version of \cite[Thm.~7.1]{GaoSurveyUML}), there exist constants $c = c(S') > 0$ and $c' = c'(S')$ such that 
\begin{equation}\label{EqHtInequalityFamily}
\hat{h}_{\cL}(x_1) + \cdots + \hat{h}_{\cL}(x_m) \ge c h_{\overline{\mathbf{H}^{\circ}},\cM}(s) - c'
\end{equation}
for all  $(x_1,\ldots,x_m) \in (\mathscr{X}_{S'}^{[m],*})_s(\IQbar)$. 
We set $c_{S',1}' = c/m$ and $c_{S',3}' = (c+c')/m$; these constants depend only on $S'$ and $m$. Consider
\begin{equation}\label{EqNonDegHtIneqSet}
\Sigma_{S',s} = \left\{ x \in \mathscr{X}_s^{\circ}(\IQbar) : \hat{h}_{\cL}(x) \le c_{S',1}'\max\{1, h_{\overline{\mathbf{H}^{\circ}}, \cM}(s)\} - c_{S',3}' \right\} \subseteq \mathscr{X}_s(\IQbar). 
\end{equation}
If $\Sigma_{S',s}^m \cap \mathscr{X}_{S'}^{[m],*}(\IQbar)$ were non-empty, then there would exist  some $$(x_1,\ldots,x_m) \in \Sigma_{S',s} ^m \cap \mathscr{X}_{S'}^{[m],*}(\IQbar),$$ in contradiction of the height inequality \eqref{EqHtInequalityFamily}. Hence $\Sigma_{S',s} ^m \subseteq (\mathscr{X}_{S'}^{[m]} \setminus \mathscr{X}_{S'}^{[m],*})_s(\IQbar)$.
Our case assumption implies $(\mathscr{X}_{S'}^{[m],*})_s \not= \emptyset$ and thus $$\mathscr{X}_s^m = (\mathscr{X}_{S'}^{[m]})_s \not= (\mathscr{X}_{S'}^{[m]} \setminus \mathscr{X}_{S'}^{[m],*})_s.$$
We may hence apply Lemma~\ref{LemmaNogaAlon} below to $X = \mathscr{X}_s$, $L=\cL_s|_{\mathscr{X}_s}$,  $Z = (\mathscr{X}_{S'}^{[m]} \setminus \mathscr{X}_{S'}^{[m],*})_s$ and $\Sigma = \Sigma_{S',s}$. It yields a Zariski closed subset $X'$ of $\mathscr{X}_s$ such that
\begin{enumerate}
\item[(i)] $X' \subsetneq \mathscr{X}_s$,
\item[(ii)] $\deg_{\cL_s} (X') < c(m, r, d, \deg_{\cL_s^{\boxtimes m}}(\mathscr{X}_{S'}^{[m]} \setminus \mathscr{X}_{S'}^{[m],*})_s)$, and
\item[(iii)] $\Sigma_{S',s} \subseteq X'(\IQbar)$.
\end{enumerate}

The degree $\deg_{\cL_s^{\boxtimes m}}(\mathscr{X}_{S'}^{[m]} \setminus \mathscr{X}_{S'}^{[m],*})_s$ is bounded from above solely in terms of $S'$ and $m$. Hence (ii) implies a bound of the form
$\deg_{\cL_s} (X') < c_{S',2}'$ with $c_{S',2}'$ depending only on $m$, $r$, $d$ and $S'$. In summary, the set $\Sigma_{S',s}$ defined in \eqref{EqNonDegHtIneqSet} is contained in $X'(\IQbar)$ for a proper Zariski closed $X' \subsetneq \mathscr{X}_s$ with $\deg_{\cL_s}(X') < c_{S',2}'$.

Letting $S'$ runs over all irreducible components of $\overline{S_{\mathrm{gen}}}$, we set $c_1^{*\prime} := \min_{S'} \{c_{S',1}'\} > 0$, $c_2^{*\prime} := \max_{S'} \{c_{S',2}'\}> 0$ and $c_3^{*\prime} := \max_{S'} \{c_{S',3}'\}>0$. 
As the Zariski closed subset $\overline{S_{\mathrm{gen}}}$ of $S$ is determined by $S$, so are its irreducible components $S^\prime$. Moreover, $m = m(S) > 0$ depends only on $S$. Therefore, the constants $c_1^{*\prime}$, $c_2^{*\prime}$ and $c_3^{*\prime}$ can be expressed only in terms of $r$, $d$ and $S$. 

Let us summarize the result of the above discussion: For each $s \in \bigcup_{S'} \pi^{[m]}(\mathscr{X}_{S'}^{[m],*})(\IQbar)$, the set
\begin{equation}\label{EqNonDegHtIneqUnionSet}
\Sigma_s^* := \left\{ x \in \mathscr{X}_s^{\circ}(\IQbar) : \hat{h}_{\cL}(x) \le c_1^{*\prime}\max\{1, h_{\overline{\mathbf{H}^{\circ}}, \cM}(s)\} - c_3^{*\prime} \right\}
\end{equation}
must be contained in $\Sigma_{S',s}$ for some $S'$; thus $\Sigma_s^*$ is contained in $X'(\IQbar)$ for some proper Zariski closed $X' \subsetneq \mathscr{X}_s$ with $\deg_{\cL_s}(X') < c_2^{*\prime}$. This concludes the proof in Case (ii).
\end{proof}

\begin{lemma}\label{LemmaNogaAlon}
    Let $k$ be a field, $X$ an irreducible projective variety over $k$ and $L$ a very ample line bundle on $X$. Consider a closed subvariety $Z\subsetneq X^M$. Then there exists a constant
	\[c=c(M,\dim X, \deg_L X, \deg_{L^{\boxtimes M}}(Z))>0\] 
	such that, for any subset $\Sigma\subseteq X(k)$ satisfying $\Sigma^M\subseteq Z(k)$, there exists a proper Zariski closed $X' \subsetneq X$ with $\Sigma \subseteq X'(k)$ and $\deg_L(X') < c$.
\end{lemma}

\begin{proof}
	We prove this lemma by induction on $M$. The base case $M=1$ is trivial as one may take $X':=Z$. 
	
	Assume next that the lemma is proved for $1, \ldots, M-1 \ge 1$. Let $q \colon X^M \rightarrow X$ be the projection to the first factor. We write $Z'$ for the union of all irreducible components $Y$ of $Z$ such that $q(Y)=X$ and $Z''$ for the union of the other components. 
Setting $\Sigma'' := q\left(\Sigma^M \cap Z''(k)\right)$, we notice that $(\Sigma \setminus \Sigma'') \times \Sigma^{M-1} \subseteq Z'(k)$.	
	By \cite[Thm.~13.1.3 and Prop.~13.2.3]{EGAIV}, the fiber $q|_{Z'}^{-1}(P)$ over a  generic point $P \in X$ has dimension $$\dim Z' - \dim X < M \cdot \dim X - \dim X = (M-1) \dim X,$$ and the set
	\[W := \{P \in X : \{P\} \times X^{M-1} \subseteq Z' \} \subsetneq X\] 
    is Zariski closed in $X$.
    
    \smallskip
    \noindent\underline{\textbf{Case (i): $\Sigma \setminus \Sigma'' \not\subseteq W(k)$}.} 
    \smallskip
    
    Take a point $P \in \Sigma\setminus \Sigma''$ such that $P \notin W(k)$ and set $Z_1:= Z' \cap (\{P\}\times X^{M-1})$. 
    Since $\{P\} \times \Sigma^{M-1} \subseteq Z'(k)$, we can apply the induction hypothesis for $M-1$, $\Sigma \subseteq X(k)$ and $Z_1 \subseteq X^{M-1}$, tacitly identifying $\{P\}\times X^{M-1}$ with $X^{M-1}$. It implies the existence of a closed subvariety $X' \subsetneq X$ such that $\Sigma \subseteq X'(k)$ and
	\begin{equation}\label{DegreeBoundCase1}
		\deg_L(X') < c(M-1, \dim X, \deg_L X, \deg_{L^{\boxtimes {(M-1)}}}(Z_1)).
	\end{equation}
    We claim that the constant on the right hand side can be expressed solely in terms of $M$, $\dim X$, $\deg_L(X)$ and $\deg_{L^{\boxtimes M}} (Z)$; this evidently suffices to complete the proof in this case. Indeed, we have
    $$\deg_{L^{\boxtimes (M-1)}}(Z_1) = \deg_{L^{\boxtimes M}}(Z' \cap (\{P\}\times X^{M-1}))$$  and 
	\[\deg_{L^{\boxtimes M}}(Z' \cap (\{P\}\times X^{M-1})) \le \deg_{L^{\boxtimes M}}(Z') \deg_{L^{\boxtimes (M-1)}}(X^{M-1})\]
	by B\'{e}zout's theorem. Using \eqref{EqDegreeProduct} inductively by taking $Y:=X$ and $Y^\prime:=X^i$ for $i=1,2,...,M-2$, we have 
	\[
	\deg_{L^{\boxtimes (M-1)}}(X^{M-1})=\frac{((M-1)\dim X)!}{((\dim X )!)^{M-1}} (\deg_L X)^{M-1}.
	\]Moreover $\deg_{L^{\boxtimes M}} (Z') \le \deg_{L^{\boxtimes M}} (Z)$. In conclusion, we get
	\[
	\deg_{L^{\boxtimes (M-1)}}(Z_1) \le \frac{((M-1)\dim X)!}{((\dim X )!)^{M-1}} \deg_{L^{\boxtimes M}} (Z) (\deg_L X)^{M-1}.
	\] 
    Hence the right hand side of \eqref{DegreeBoundCase1} depends only on $M$, $\dim X$, $\deg_L X$ and $\deg_{L^{\boxtimes M}} (Z)$. This proves the claim and ends the proof.
    
    \smallskip
    \noindent\underline{\textbf{Case (ii): }$\Sigma \setminus \Sigma'' \subseteq W(k)$.}
    \smallskip
    
    In this case, $(\Sigma \setminus \Sigma'') \times X^{M-1} \subseteq Z'$. We also assume that $k$ is algebraically closed, which we may by enlarging the base field if needed.  
    As $Z' \subsetneq X^M$, there exists a point $\underline{x} \in X^{M-1}(k)$ such that $(X \times \{\underline{x}\}) \cap Z' \not= X \times \{\underline{x}\}$. Thus $(X \times \{\underline{x}\}) \cap Z' = X'' \times \{\underline{x}\}$ for some closed subset $X'' \subsetneq X$. We have $\Sigma \setminus \Sigma'' \subseteq X''(k)$ by construction, and 
	\[\deg_L X'' \le \deg_{L^{\boxtimes M}}(X\times \{\underline{x}\}) \deg_{L^{\boxtimes M}}(Z) = \deg_L X \deg_{L^{\boxtimes M}}(Z)\]
	by B\'{e}zout's theorem.

    To control the points in $\Sigma''$, note that the closed subset $q(Z'')$ is properly contained in $X$ by definition. Moreover, we have
	\[
	\deg_L q(Z'')\le \deg_{L^{\boxtimes M}} (Z'')\le  \deg_{L^{\boxtimes M}} (Z);
	\]
    the first inequality follows again from \cite[Lem.~2.4]{DillTorsionIsogenousAV}.
    In summary, it suffices to take $X' := X'' \cup q(Z'') \subsetneq X$ in this case.
\end{proof}

\subsection{The case of principally polarized abelian varieties}

We next prove Proposition~\ref{PropHtIneqStep1} in the special case of principally polarized abelian varieties.

\begin{prop}\label{PropHtIneqPPAV}
Let $g$, $r$ and $d$ be positive integers. There exist constants $c_1' = c_1'(g,r,d) > 0$, $c_2' = c_2'(g,r,d) > 0$ and $c_3' = c_3'(g,r,d) >0$ satisfying the following property: For 
\begin{itemize}
\item every principally polarized abelian variety $(A,L)$ defined over $\IQbar$ of dimension $g$,
\item and every irreducible subvariety $X$ of $A$
with $\dim X = r$ and $\deg_L X = d$ generating $A$,
\end{itemize}
the set
\[
\Sigma:=\left\{ x \in X^{\circ}(\IQbar) : \hat{h}_{L\otimes [-1]^*L}(x) \le c_1'\max\{1, h_{\mathrm{Fal}}(A)\} - c_3' \right\}
\]
is contained in  
$X'(\IQbar)$ for some proper Zariski closed $X' \subsetneq X$ with $\deg_L(X') < c_2'$.
\end{prop}

\begin{proof}
For simplicity denote by  $L_{-} := [-1]^*L$.

Let $\mathfrak{A}_g \to \mathbb{A}_g$ and $\mathfrak{L}_g$ be as from the beginning of $\mathsection$\ref{SubsectionHtIneqFamilies}. Then each principally polarized abelian variety $(A,L)$  defined over $\IQbar$ gives rise to a point $b \in \mathbb{A}_g(\IQbar)$, such that $A \cong (\mathfrak{A}_g)_b$ and $c_1(\mathfrak{L}_g|_{(\mathfrak{A}_g)_b}) = 2c_1(L) = c_1(L\otimes L_{-})$. 
 In the rest of the proof, we identify $A$ with $(\mathfrak{A}_g)_b$.

Let $X$ be an irreducible subvariety of $A$ generating $A$ such that $d = \deg_L X$ and $r = \dim X$. We may assume $r \ge 1$ because otherwise the result is trivial. We have 
\begin{equation}\label{EqDegreeWRTLineBundleOverUnivFamily}
\deg_{\mathfrak{L}_g^{\otimes 4} |_A}X = (4c_1(\mathfrak{L}_g|_A))^r \cdot [X] = (8c_1(L))^r \cdot [X] = 8^r \deg_L X = 8^r d.
\end{equation}

Consider the restricted Hilbert scheme $\mathbf{H}^{\circ}:=\mathbf{H}^{\circ}_{r, 8^r d}(\mathfrak{A}_g/\mathbb{A}_g)$ from \eqref{EqHilbSchemeOpenLocusFamily} and 
recall the commutative diagram \eqref{EqSubsectionHtIneqFamiliesHilbDiagram}
\[
\xymatrix{
\mathscr{X}_{\mathbf{H}^{\circ}} \ar@{^(->}[r] \ar[rd] & 
\cA_{\mathbf{H}^{\circ}} \ar[d]^-{\pi} \ar[r]^-{\iota} \pullbackcorner & \mathfrak{A}_g \ar[d]^-{\pi^{\mathrm{univ}}}  \\
& \mathbf{H}^{\circ} \ar[r]^-{\iota_{\mathbf{H}^{\circ}}} & \mathbb{A}_g
}
\]
of varieties and morphisms defined over $\IQbar$. Note that $\pi|_{\mathscr{X}_{\mathbf{H}^{\circ}}} \colon \mathscr{X}_{\mathbf{H}^{\circ}} \rightarrow \mathbf{H}^{\circ}$ is the universal family. 

The pair $(X, (A,L))$ is parametrized by a point $s \in \mathbf{H}^{\circ}(\IQbar)$ such that $X = \mathscr{X}_s$, $A = \cA_s = (\mathfrak{A}_g)_b$ and $\iota_{\mathbf{H}^{\circ}}(s) = b$.

\medskip

For the line bundle $\cM = \overline{\iota_{\mathbf{H}^{\circ}}}^*\mathfrak{M}$ on $\overline{\mathbf{H}^{\circ}}$ defined above Proposition~\ref{PropHtIneqStep1Aux}, we have $$h_{\overline{\mathbf{H}^{\circ}}, \cM}(s) = h_{\overline{\mathbb{A}_g}, \mathfrak{M}}(\iota_{\mathbf{H}^{\circ}}(s)) = h_{\overline{\mathbb{A}_g}, \mathfrak{M}}(b).$$ By fundamental work of Faltings \cite[Thm.~1.1]{MoretBailly:Hauteurs}, 
$h_{\mathrm{Fal}}(A)$ is bounded from above in terms of
$h_{\overline{\mathbb{A}_g},\mathfrak{M}}(b) = h_{\overline{\mathbf{H}^{\circ}}, \cM}(s)$ and $g$ only. 
More precisely, 
\begin{align}\label{EqFalHtModuliHt}
\max\{1, h_{\mathrm{Fal}}(A))\} 
&\le c(g) h_{\overline{\mathbf{H}^{\circ}}, \cM}(s) + c'(g) \log ( h_{\overline{\mathbf{H}^{\circ}}, \cM}(s) + 2) \\
&\le (c(g)+c'(g))h_{\overline{\mathbf{H}^{\circ}}, \cM}(s) + 2c'(g) \nonumber
\end{align}
\normalsize
for some constants $c(g) > 0$ and $c'(g) > 0$.

For the line bundle $\cL = \iota^*\mathfrak{L}_g^{\otimes 4}$ on $\cA_{\mathbf{H}^{\circ}}$ defined above Proposition~\ref{PropHtIneqStep1Aux}, we have $\cL|_{\cA_s} = \mathfrak{L}_g^{\otimes 4}|_{(\mathfrak{A}_g)_{\iota(s)}} = \mathfrak{L}_g^{\otimes 4}|_{(\mathfrak{A}_g)_b}$, and hence $\cL|_A = \mathfrak{L}_g^{\otimes 4}|_A$ via the identification $A = \cA_s = (\mathfrak{A}_g)_b$ fixed above. Thus 
$c_1(\cL|_A) = 4c_1(L\otimes L_{-})$. Moreover, both $\cL|_A$ and $(L \otimes L_{-})^{\otimes 4}$ are symmetric and ample on $A$, so Lemma~\ref{LemmaSymAmpleFirstChernTorsion} implies
\begin{equation}\label{EqNTHeightChangetoL}
\hat{h}_{\cL|_A}(x) = 4\hat{h}_{L\otimes L_{-}}(x)\text{ for each } x \in A(\IQbar).
\end{equation}

\medskip

Let $\mathbf{H}^{\circ}_{\mathrm{gen}}$ be the subset of $\mathbf{H}^{\circ}(\IQbar)$ as defined by \eqref{EqConstructibleSet}. If the stabilizer $\mathrm{Stab}_A(X)$ of $X$ in $A$ has positive dimension, then $X^{\circ} = \emptyset$ by the definition of the Ueno locus and hence the proposition trivially holds true. Therefore, we may and do assume that $\mathrm{Stab}_A(X)$ is finite. Then the point $s \in \mathbf{H}^{\circ}(\IQbar)$ parameterizing $(X, (A,L))$ is in $\mathbf{H}_{\mathrm{gen}}^{\circ}$. 

Invoking Proposition~\ref{PropHtIneqStep1Aux} with $S = \mathbf{H}^{\circ}$, we obtain constants $c_1'$, $c_2'$ and $c_3'$ depending only on $r$, $8^r d$ and $\mathbf{H}^{\circ}$ (and hence ultimately only on $g$, $d$ and $r$) and a proper Zariski closed $X' \subsetneq X$ with $\deg_{\cL|_A}(X') < c_2'$ such that the set $$\Sigma_s := \left\{ x \in X^{\circ}(\IQbar) : \hat{h}_{\cL|_A}(x) \le c_1'\max\{1, h_{\overline{\mathbf{H}^{\circ}}, \cM}(s)\} - c_3' \right\}$$ 
is contained in $X^\prime(\IQbar)$.

Notice that $\deg_L(X') <\deg_{\cL|_A}(X')$ because of $8c_1(L) = c_1(\cL|_A)$ and hence $\deg_L(X') < c_2'$.

By \eqref{EqFalHtModuliHt} and \eqref{EqNTHeightChangetoL}, we have
\[
\left\{ x \in X^{\circ}(\IQbar) : \hat{h}_{L\otimes L_{-}}(x) \le \frac{c_1'}{4(c(g)+c'(g))} \max\{1, h_{\mathrm{Fal}}(A)\} - \frac{c_3' + 2c_1'}{4}\right\} \subseteq \Sigma_s
\]
We are done on replacing $c_1'$ with $c_1'/4(c(g)+c'(g))$ and $c_3'$ with $(c_3' + 2c_1')/4$.
\end{proof}

\subsection{Proof of Proposition~\ref{PropHtIneqStep1}}\label{SubsectionProofOfHtIneqApp}

Let $(A,L)$ be a polarized abelian variety with $\deg_L A = l$ and let $X$ be an irreducible subvariety generating $A$ and such that $\deg_L X = d$. 

By Lemma~\ref{LemmaAbIsogPPAV}.(iv), there exists a principally polarized abelian variety $(A_0,L_0)$ and an isogeny $u_0 \colon A \rightarrow A_0$ such that $L = u_0^*L_0$; moreover, $\deg(u_0) = l / g!$. A basic property of the Faltings height \cite[Prop.~1.4.1]{Raynaud:HI} yields the bound
\begin{equation}\label{EqFaltingsHeight}
|h_{\mathrm{Fal}}(A) - h_{\mathrm{Fal}}(A_0)| \le \frac{1}{2}\log\deg u_0 = \frac{1}{2} \log (l/g!).
\end{equation}

It is well-known that there exists an isogeny $u \colon A_0 \rightarrow A$ such that $u_0 \circ u = [\deg u_0]$ on $A_0$. So $u^*L= (u_0\circ u)^*L_0 = [l/g!]^*L_0$, and $\deg u = (\deg u_0)^{2g -1} = (l/g!)^{2g-1}$.

Let $X_0$ be an irreducible component of $u^{-1}(X)$. Then $u(X_0) = X$ as $u$ is étale and $X_0$ generates $A_0$. By the projection formula \cite[Prop.~2.5]{Fulton}, we have
\begin{multline}\label{EqDegreeChangeIsogeny}
d' 
:=\deg_{u^*L} X_0 = c_1(u^*L)^{\dim X}  \cdot [X_0] \\
= \deg(u|_{X_0}) \cdot (c_1(L)^{\dim X} \cdot [X])
\le \deg(u) \cdot \deg_L X
= d(l/g!)^{2g-1}.
\end{multline}
It is easy to infer from the definition of the Ueno locus that $X^{\circ} = u(X_0)^{\circ} \subseteq u(X_0^{\circ})$.

Let $c_{0,1}'(g,r,d') > 0$, $c_{0,2}'(g,r,d') > 0$ and $c_{0,3}'(g,r,d') > 0$ be the constants from Proposition~\ref{PropHtIneqPPAV} with $d$  replaced by $d'$. 
Set $$c_1' := \min_{1\le d' \le  d(l/g!)^{2g-1}}\{c_{0,1}'(g,r,d')\} > 0,$$ $$c_2' := \max_{1\le d' \le d(l/g!)^{2g-1}}\{c_{0,2}'(g,r,d')\} > 0,$$ and $$c_3':=\max_{1\le d' \le d(l/g!)^{2g-1}} \{c_{0,3}'(g,r,d')\} > 0.$$ Then $c_1'$, $c_2'$ and $c_3'$ depend only on $g$, $r$ and $d$. 
Proposition~\ref{PropHtIneqPPAV} yields the following assertion: the set $$\Sigma_0 := \left\{ x_0 \in X_0^{\circ}(\IQbar) : \hat{h}_{L_0 \otimes [-1]^*L_0}(x_0) \le c_1' \max\{1,h_{\mathrm{Fal}}(A_0) \} - c_3' \right\}$$ is contained in $X_0'(\IQbar)$ for a proper Zariski closed $X_0' \subsetneq X_0$ with $\deg_{L_0}(X'_0) \le c_2'$.

Let $$\Sigma := \left\{ x \in X^{\circ}(\IQbar) : (g!/l)^2 \hat{h}_{L\otimes [-1]^*L}(x) \le c_1' \max\{1, h_{\mathrm{Fal}}(A)\} - c_3' - (1/2)\log (l/g!) \right\}.$$ As $X^{\circ} \subseteq u(X_0^{\circ})$ and $u^*(L\otimes [-1]^*L) \cong (L_0\otimes [-1]^*L_0)^{\otimes (l/g!)^2}$, 
\eqref{EqFaltingsHeight} yields $\Sigma \subseteq u(\Sigma_0)$. 

Set $X' := u(X_0')$. Then $X' \subsetneq X$, $\Sigma \subseteq X'(\IQbar)$, and
\begin{equation*}
\deg_L X' = c_1(L)^{\dim X'} \cdot [X'] 
 \le c_1(u^*L)^{\dim X'} \cdot [X_0']  
 = \deg_{u^*L}X_0' 
\le c_2'.
\end{equation*}
Replacing $c_1'$ with $(l/g!)^2 c_1'$ and $c_3'$ with $(l/g!)^2 (c_3' + (1/2)\log (l/g!))$, we obtain the assertion of Proposition~\ref{PropHtIneqStep1}. \qed

\section{Applying equidistribution}\label{SectionEqDist}
\begin{prop}\label{PropEqDistStep1}
Let $g$, $l$, $r$ and $d$ be positive integers. There exist constants  $c_2'' = c_2''(g,l,r,d) > 0$ and $c_3'' = c_3''(g,l,r,d) > 0$ with the following property. For 
\begin{itemize}
\item every polarized abelian variety $(A,L)$ of dimension $g$ defined over $\IQbar$ with $\deg_L A = l$, and
\item every irreducible subvariety $X \subseteq A$ generating $A$ with $\dim X = r$ and $\deg_L X = d$,
\end{itemize}
the set
\begin{equation}
\Sigma:=\left\{ x \in X^{\circ}(\IQbar) : \hat{h}_{L\otimes [-1]^*L}(x) \le c_3'' \right\}
\end{equation}
is contained in $X'(\IQbar)$ for a proper Zariski closed $X' \subsetneq X$ with $\deg_L(X') < c_2''$.
\end{prop}
The main idea of the proof this proposition is similar to that of the proof of Proposition~\ref{PropHtIneqStep1}: We put the various 
pairs of polarized abelian varieties $(A,L)$ and subvarieties $X \subset A$ into finitely many families over the components of the \textit{restricted} Hilbert scheme constructed in \eqref{EqHilbSchemeOpenLocusFamily}. The technical core of this section is a related result for each such family (Proposition ~\ref{PropEqDistStep1Aux}). To achieve the uniformity as stated, we need to do however also a careful bookkeeping.

\subsection{A preliminary result for the restricted Hilbert scheme}\label{SubsectionEqDistFamilies}
We retain the notation from $\mathsection$\ref{SubsectionHtIneqFamilies}. In particular $\pi^{\mathrm{univ}} \colon \mathfrak{A}_g \rightarrow \mathbb{A}_g$ is the universal abelian variety over the fine moduli space of principally polarized abelian varieties with level-$4$-structure, and $\mathfrak{L}_g$ is a symmetric relatively ample line bundle $\mathfrak{L}_g$ on $\mathfrak{A}_g / \mathbb{A}_g$ satisfying the following property: for each principally polarized abelian variety $(A,L)$ parametrized by $b \in \mathbb{A}_g(\IQbar)$, we have $c_1(\mathfrak{L}_g|_{(\mathfrak{A}_g)_b}) = 2c_1(L)$. 

For each $1\le r \le g$ and each $d\ge 1$, let $\mathbf{H}^{\circ}:=\mathbf{H}_{r,d}(\mathfrak{A}_g / \mathbb{A}_g)^{\circ}$ be the restricted Hilbert scheme  from \eqref{EqHilbSchemeOpenLocusFamily}; recall that its $\IQbar$-points parametrize all pairs $(X, (A,L))$ of a principally polarized abelian variety $(A,L)$ and an irreducible subvariety $X \subseteq A$ with $\dim X = r$ and $\deg_{\mathfrak{L}_g^{\otimes 4}|_A}X = d$.

We have a commutative diagram over $\IQbar$
(\eqref{EqHilbFiberPowerToUse} with $m=1$)
\begin{equation}\label{EqSubsectionEqDistFamiliesHilbDiagram}
\xymatrix{
\mathscr{X}_{\mathbf{H}^{\circ}} \ar@{^(->}[r] \ar[rd] & 
\cA_{\mathbf{H}^{\circ}} \ar[d]^-{\pi} \ar[r]^-{\iota} \pullbackcorner & \mathfrak{A}_g \ar[d]^-{\pi^{\mathrm{univ}}}  \\
& \mathbf{H}^{\circ} \ar[r]^-{\iota_{\mathbf{H}^{\circ}}} & \mathbb{A}_g
}
\end{equation}
where $\pi|_{\mathscr{X}_{\mathbf{H}^{\circ}}} \colon \mathscr{X}_{\mathbf{H}^{\circ}} \rightarrow \mathbf{H}^{\circ}$ is the universal family. 
Set $\cL := \iota^*\mathfrak{L}_g^{\otimes 4}$.

\begin{prop}\label{PropEqDistStep1Aux}
Let $S \subseteq \mathbf{H}^{\circ}$ be a (not necessarily irreducible) subvariety and let $S_{\mathrm{gen}} \subseteq S(\IQbar)$ be the subset defined by \eqref{EqConstructibleSet}.

There exist constants $c_2'' = c_2''(r,d,S) >0$ and $c_3'= c_3''(r,d,S) >0$ such that the following property holds true. 
For each $s \in S_{\mathrm{gen}}(\IQbar)$, the set
\begin{equation}\label{EqNGPHtIneq2}
\Sigma_s := \left\{ x \in \mathscr{X}_s^{\circ}(\IQbar) : \hat{h}_{\cL}(x) \le c_3'' \right\}
\end{equation}
is contained in a proper Zariski closed subset $X' \subsetneq \mathscr{X}_s$ with $\deg_{\cL_s}(X') < c_2''$.
\end{prop}

This proposition is a generalization of \cite[Prop.~21]{Kuehne:21} on families of curves. 
For convenience of the reader, we divide the proof into five steps. 
The goal of the first four steps is to run a family version of the classical approach of Ullmo \cite{Ullmo} and Zhang \cite{ZhangEquidist} to the Bogomolov conjecture in order to obtain a \textit{generic} lower bound on heights that are closely related to $\hat{h}_{\cL}$, which is our height of interest. 
In the final step, we then deduce Proposition~\ref{PropEqDistStep1Aux} from this height bound. As in the proof of Proposition~\ref{PropHtIneqStep1Aux}, we use Lemma~\ref{LemmaNogaAlon} to deal with the \textit{non-generic} points for which we do not obtain a height bound; this gives rise to the exceptional set $X'$. 

This family version of the Ullmo--Zhang approach differs from the classical one mainly in two aspects. First, a version of equidistribution is used that applies to \textit{families} of abelian varieties; for our purpose, \cite[Thm.~1]{Kuehne:21} or the more general results \cite[Thm.~6.7]{YuanZhangEqui} and \cite{GauthierGood} are sufficient replacements of the classical equidistribution result of Szpiro--Ullmo--Zhang \cite{SUZ}. Second and in contrast to \cite{SUZ}, a new condition, namely non-degeneracy as defined in \cite[Defn.~B.4]{DGHUnifML}, has to be verified for the subvarieties under consideration. 
Notice that every subvariety is non-degenerate in the case of a single abelian variety. So this extra verification is a genuinely new aspect of working in families and  in practice this step is often very difficult.

\begin{proof}[Proof of Proposition~\ref{PropEqDistStep1Aux}]
Decomposing $S$ into its irreducible components, it suffices to prove the proposition for irreducible $S$. 

We prove the proposition by induction on $\dim S$. The base case $\dim S = 0$ is in fact contained in the first three steps of the following proof. In contrast to Proposition \ref{PropHtIneqStep1Aux} it is a highly non-trivial statement, namely the classical Bogomolov conjecture. One could alternatively cite \cite{Ullmo,ZhangEquidist} for the base case and hence we do not state its proof separately as an induction start here. The induction hypothesis is only utilized in Step~5 and is unnecessary in case $\dim S=0$.

\medskip
\noindent\underline{\textbf{Step~1:} Construction of non-degenerate subvarieties}
\smallskip

Endow each irreducible component $S'$ of $\overline{S_{\mathrm{gen}}}$ with the reduced induced subscheme structure.  Lemma~\ref{LemmaConstructibleSet} allows us to invoke Proposition~\ref{PropNonDeg} for $\mathscr{X}_{S'} \subseteq \cA_{S'} \rightarrow S'$.

Let $m_0(S') > 0$ be from Proposition~\ref{PropNonDeg}, and let $m = \max_{S'}\{m_0(S')\}$ with $S'$ running over all irreducible components of $\overline{S_{\mathrm{gen}}}$. Then $\mathscr{X}^{[m]}_{S'}$ is a non-degenerate subvariety of $\cA_{S'}^{[m]}$ by Proposition~\ref{PropNonDeg}. Notice that $m > 0$ depends only on $S$. Let us write $\omega_m$ for the Betti form on $\cA_{S'}^{[m]}$. By \cite[Prop.~2.2.(iii)]{DGHUnifML}, there exists a point $\mathbf{z} \in (\mathscr{X}^{[m]}_{S'})^{\mathrm{sm}}(\mathbb{C})$ such that 
\begin{equation}\label{EqNonDegPoint}
(\omega_m|_{\mathscr{X}^{[m]}_{S'}(\mathbb{C})}^{\wedge (rm +\dim(S^\prime))})_\mathbf{z} \not= 0.
\end{equation}
Since this condition on $\mathbf{z}$ is open (in the usual topology), we may and do assume that $\mathbf{z}$ lies over a smooth point of $S'$.

For each $M \ge 1$,
we consider the proper $S'$-morphism 
\begin{align}\label{EqDDMA}
\mathscr{D}_0 \colon \ \ &(\cA_{S'}^{[m]})^{[M+1]}& &\longrightarrow& &(\cA_{S'}^{[m]})^{[M]}& \\
&(\mathbf{a}_0,\mathbf{a}_1,\ldots,\mathbf{a}_M)& &\longmapsto& &(\mathbf{a}_1-\mathbf{a}_0,\ldots,\mathbf{a}_M-\mathbf{a}_0),& \nonumber
\end{align}
using the fiberwise group structure.

Let $\bar{\eta}$ be the geometric generic point of $S'$. By the proof of \cite[Lem.~3.1]{ZhangEquidist}, there exists $M_0(S') > 0$ with the following property: For each $M \ge M_0(S')$, the generic fiber of the restriction of $\mathscr{D}_0$ to $\mathscr{X}_{\bar{\eta}}^{m(M+1)}$ is an orbit under the action of $\mathrm{Stab}_{\cA_{\bar{\eta}}^m}(\mathscr{X}_{\bar{\eta}}^m) \subseteq \cA_{\bar{\eta}}^m$ diagonally embedded into $(\cA_{\bar{\eta}}^m)^{M+1}$.
Notice that $\mathrm{Stab}_{\cA_{\bar{\eta}}^m}(\mathscr{X}_{\bar{\eta}}^m) = \mathrm{Stab}_{\cA_{\bar{\eta}}}(\mathscr{X}_{\bar{\eta}})^m$ is finite by the definition of $S_{\mathrm{gen}}$. We set $M = \max_{S'} \{M_0(S')\} > 0$ where $S'$ runs over all irreducible components of $S_{\mathrm{gen}}$. Then $M = M(S) > 0$ depends only on $S$. 

We further define the $S'$-morphism 
\begin{equation*}
\mathscr{D}_0^\prime := (\mathrm{id}, \mathscr{D}_0) \colon \ \ \cA_{S'}^{[m]} \times_{S'} (\cA_{S'}^{[m]})^{[M+1]} \longrightarrow \cA_{S'}^{[m]} \times_{S'} (\cA_{S'}^{[m]})^{[M]}    
\end{equation*}
and its restriction 
\begin{equation}\label{EqFZFinal}
\mathscr{D} : \ \ \mathscr{X}^{[m]}_{S'} \times_{S'} (\mathscr{X}^{[m]}_{S'})^{[M+1]} = \mathscr{X}^{[m(M+2)]}_{S'} \longrightarrow \mathscr{D}_0^\prime(\mathscr{X}^{[m(M+2)]}_{S'}).
\end{equation}
(Note that we also (co-)restrict on the range to simplify our notation.) For later use, we set
\begin{equation*}
    G_{\bar{\eta}}:=\{0\}\times \mathrm{Stab}_{\cA_{\bar{\eta}}}(\mathscr{X}_{\bar{\eta}})^m \subseteq \cA_{\bar{\eta}}^{m} \times (\cA_{\bar{\eta}}^m)^{M+1} = \cA_{\bar{\eta}}^{m(M+2)};
\end{equation*}
then the fibers of $\mathscr{D}|_{\overline{\eta}}$ are naturally orbits under $G_{\overline{\eta}}$.

The subvarieties
\begin{equation*}
\mathscr{X}^{[m(M+2)]}_{S'}\subseteq \cA_{S'}^{[m(M+2)]}     \ \ \text{and} \ \ 
\mathscr{D}(\mathscr{X}^{[m(M+2)]}_{S'}) \subseteq \cA_{S'}^{[m(M+1)]}
\end{equation*}
are both non-degenerate. For the former one, this follows directly from our choice of $m$. For the latter one, we remark that
\begin{equation*}
    \mathscr{D}(\mathscr{X}^{[m(M+2)]}_{S'})) = \mathscr{X}^{[m]}_{S'} \times_{S'} \mathscr{D}_0((\mathscr{X}^{[m]}_{S'})^{[M+1]}),
\end{equation*}
whose non-degeneracy follows from the non-degeneracy of $\mathscr{X}^{[m]}_{S'}$; see for example \cite[Lem.~24]{Kuehne:21} or \cite[Lem.~6.2]{GaoSurveyUML}.

\medskip
\noindent\underline{\textbf{Step~2:} Defining non-proportional measures $\mu_{S',1}$ and $\mathscr{D}^*\mu_{S',2}$ on $\cA_{S'}^{[m(M+2)]}(\mathbb{C})$}
\smallskip

Write $\omega_{m(M+2)}$ (resp.\ $\omega_{m(M+1)}$) for the Betti form on $\cA_{S'}^{[m(M+2)]}(\mathbb{C})$ (resp.\ $\cA_{S'}^{[m(M+1)]}(\mathbb{C})$). Furthermore, let $\mu_{S',1}$ (resp.\ $\mu_{S',2}$) on $\mathscr{X}^{[m(M+2)]}_{S'}(\mathbb{C})$ (resp.\ $\mathscr{D}(\mathscr{X}^{[m(M+2)]}_{S'})(\mathbb{C})$) be the equilibrium measure for which \cite[Thm.~1]{Kuehne:21} holds. By its proof (compare the end of Section 3 in \cite{Kuehne:21}), the measure $\mu_{S',1}$ (resp.\ $\mu_{S',2}$) is proportional to the restriction of
\begin{equation*}
    (\omega_{m(M+2)})^{\wedge (rm(M+2)+\dim(S^\prime))} \ \ \text{(resp.\ } (\omega_{m(M+1)})^{\wedge (rm(M+2)+\dim(S^\prime))}).)
\end{equation*}
We use this to prove that the measures $\mu_{S',1}$ and $\mathscr{D}^*\mu_{S',2}$ are non-proportional. In particular, $\mu_{S',1} \not= \mathscr{D}^*\mu_{S',2}$. 

For every point $\mathbf{t} \in (\mathscr{X}^{[m]}_{S'})(\mathbb{C})$,  write
\begin{equation*}
\Delta_{\mathbf{t}} := 
(\mathbf{t},\ldots,\mathbf{t}) \in (\mathscr{X}^{[m]}_{S'})^{[M+1]}(\mathbb{C}) = \mathscr{X}^{[m(M+1)]}_{S'}(\mathbb{C}).
\end{equation*}
for its $(M+1)$-fold self product. For the point $\mathbf{z} \in (\mathscr{X}^{[m]}_{S'})^{\mathrm{sm}}(\mathbb{C})$ chosen above \eqref{EqNonDegPoint}, the point $(\mathbf{z},\Delta_\mathbf{z})$ is a smooth point of $(\mathscr{X}^{[m]}_{S'} \times_{S'} (\mathscr{X}^{[m]}_{S'})^{[M+1]})(\mathbb{C})$ and
\[
(\mu_{S',1})_{(\mathbf{z}, \Delta_\mathbf{z})} \not= 0;
\]
see for example \cite[Lem.~11 and 25]{Kuehne:21}.

By definition, we have $\mathscr{D}((\mathbf{z},\Delta_{\mathbf{z}})) = (\mathbf{z},0,\dots,0)$ and hence
\begin{equation*}
    \mathscr{D}^{-1}\mathscr{D}((\mathbf{z},\Delta_{\mathbf{z}}))
     =  \{ (\mathbf{z},\Delta_{\mathbf{t}}) \ | \ \mathbf{t} \in \mathscr{X}^{[m]}_{S'}(\mathbb{C}) \}.
\end{equation*}
As this is of dimension $rm + \dim S^\prime > 0$ locally at $(\mathbf{z},\Delta_{\mathbf{z}})$, the differential
\begin{equation*}
d\mathscr{D} \colon T_{(\mathbf{z},\Delta_\mathbf{z})}  \mathscr{X}^{[m(M+2)]}_{S'} \longrightarrow T_{\mathscr{D}(\mathbf{z}, \Delta_{\mathbf{z}})}\mathscr{D}(\mathscr{X}^{[m(M+2)]}_{S'})
\end{equation*}
has a non-trivial kernel. So
\[
(\mathscr{D}^*\mu_{S',2})_{(\mathbf{z}, \Delta_{\mathbf{z}})} = 0,
\]
and, consequently, $\mu_{S',1} \not= \mathscr{D}^*\mu_{S',2}$.

\medskip
\noindent\underline{\textbf{Step~3:} Choice of test functions $f_{S',1},f_{S',2}$ and $\varepsilon_{S'} > 0$}
\smallskip

As $\mu_{S',1} \not= \mathscr{D}^*\mu_{S',2}$, there exists a constant $\epsilon_{S'} > 0$ as well as a function $f_{S',1} \in \mathscr{C}_{\mathrm{c}}^0(\mathscr{X}^{[m(M+2)]}_{S'}(\mathbb{C}))$ such that
\begin{equation}\label{EqIntegralDiff}
\left| \int_{\mathscr{X}^{[m(M+2)]}_{S'}(\mathbb{C})} f_{S',1}\mu_{S',1} - \int_{\mathscr{X}^{[m(M+2)]}_{S'}(\mathbb{C})} f_{S',1} \mathscr{D}^*\mu_{S',2}  \right| > 2\varepsilon_{S'}.
\end{equation}
This is, however, not enough for our purposes. We want $f_{S',1}$ to be of the form $$f_{S',1} =f_{S',2} \circ \mathscr{D}$$ for some $f_{S',2} \in \mathscr{C}_{\mathrm{c}}^0( \mathscr{D}(\mathscr{X}^{[m(M+2)]}_{S'})(\mathbb{C}))$. 

To achieve this goal, let $G_{\overline{\eta}}$ be the finite group defined in Step~1. By abuse of notation denote the projection by $\pi \colon \cA_{S'}^{[m(M+2)]} \rightarrow S'$. There exists a Zariski open dense $U\subseteq S^\prime$ such that $G_{\overline{\eta}}$ extends to a flat group scheme $G$ over $U$. Furthermore, there exists a Zariski open dense subset $V \subseteq \mathscr{D}(\mathscr{X}^{[m(M+2)]}_{U})$ such that $\mathscr{D}|_{\mathscr{D}^{-1}(V)} \colon \mathscr{D}^{-1}(V) \rightarrow V$ is finite \'{e}tale and each of its fibers $\mathscr{D}^{-1}(\mathbf{y})$, $\mathbf{y} \in V(\mathbb{C})$, is a $G_{\pi(\mathbf{y})}$-orbit, where $G_{\pi(\mathbf{y})}$ is the restriction of $G$ to the fiber over $\pi(\mathbf{y})$. We may and do assume that $f_{S^\prime,1}$ is supported in $V \subseteq \pi^{-1}(U)$ without compromising \eqref{EqIntegralDiff}. Furthermore, we can replace $f_{S^\prime,1}$ by 
\begin{align*}
f_{S^\prime,1} \colon \ \ &\mathscr{X}^{[m(M+2)]}_{U}& &\longrightarrow& &\mathbb{R}& \\
&(\mathbf{z}_0,\mathbf{z}_1,\ldots,\mathbf{z}_M)& &\longmapsto& &\sum_{\mathbf{t} \in G_{s}} f_{S',1}(\mathbf{z}_0 + \mathbf{t}, \dots, \mathbf{z}_M + \mathbf{t})&
\end{align*}
where $G_s$ is the restriction of $G$ to the fiber over $s = \pi(\mathbf{z}_0,\dots,\mathbf{z}_M)$ and \eqref{EqIntegralDiff} remains valid; indeed, both $\mu_{S^\prime,1}$ and $\mu_{S^\prime,2}$ are translation invariant since they are exterior powers of Betti forms, which are translation invariant.

In summary, $f_{S^\prime,1}$ is constant on the fibers of $\mathscr{D}$. Since its support is contained in the locus where $\mathscr{D}$ is étale, there exists a unique function $f_{S^\prime,2}$ such that $f_{S^\prime,1} = f_{S',2} \circ \mathscr{D}$.

\medskip
\noindent\underline{\textbf{Step~4:} Obtaining a height lower bound via equidistribution}
\smallskip

We apply the equidistribution result \cite[Thm.~1]{Kuehne:21} in the form of \cite[Lem.~23]{Kuehne:21} twice. Applied to $\mathscr{X}_{S^\prime}^{[m(M+2)]}$, $f_{S',1}$ and $\varepsilon_{S'}$, we obtain a constant $\delta_{\varepsilon_{S'},1}>0$ and a Zariski closed subset $Z_{S',1}\subsetneq \mathscr{X}^{[m(M+2)]}_{S'}$ such that
\begin{equation}
\label{equation::integral1}
    \left|  \int_{\mathscr{X}^{[m(M+2)]}_{S'}(\mathbb{C})} f_{S',1}\mu_{S',1} - \frac{1}{\#\mathbf{O}(\mathbf{x})}\sum_{\mathbf{y}\in \mathbf{O}(\mathbf{x})}f_{S',1}(\mathbf{y}) \right| < \varepsilon_{S^\prime}
\end{equation}
for all $\mathbf{x} \in (\mathscr{X}^{[m(M+2)]}_{S'} \setminus Z_{S',1})(\IQbar)$ with $\hat{h}_{\cL^{\boxtimes m(M+2)}}(\mathbf{x}) \ge \delta_{\varepsilon_{S'},1}$; here $\mathbf{O}(\mathbf{x})$ is the Galois orbit of $\mathbf{x}$ (over $\mathbb{Q}$). Applying the lemma to $\mathscr{D}(\mathscr{X}^{[m(M+2)]}_{S'})$, $f_{S',2}$ and $\varepsilon_{S'}$, we obtain similarly a constant $\delta_{\varepsilon_{S'},2}>0$ and a Zariski closed subset $Z_{S',2}\subsetneq \mathscr{D}(\mathscr{X}_{S^\prime}^{[m(M+2)]})$ such that
\begin{equation}
\label{equation::integral2}
    \left|  \int_{\mathscr{D}(X_{S^\prime}^{[m(M+2)]})(\mathbb{C})} f_{S',2}\mu_{S',2} - \frac{1}{\#\mathbf{O}(\mathbf{x}')}\sum_{\mathbf{y}'\in \mathbf{O}(\mathbf{x}')}f_{S',2}(\mathbf{y}') \right| < \varepsilon_{S^\prime}
\end{equation}
for all $\mathbf{x}' \in (\mathscr{D}(\mathscr{X}_{S^\prime}^{[m(M+2)]}) \setminus Z_{S',2})(\IQbar)$ with $\hat{h}_{\cL^{\boxtimes m(M+1)}}(\mathbf{x}') \ge \delta_{\varepsilon_{S'},2}$.

Set $\delta_{S'}:= \min\{\delta_{\varepsilon_{S'},1}, \delta_{\varepsilon_{S'},2}\} > 0$ and define the Zariski closed set $Z_{S'} = Z_{S',1} \cup \mathscr{D}^{-1}(Z_{S',2})$. 
Note that $Z_{S^\prime} \neq \mathscr{X}^{[m(M+2)]}_{S'}$ as  $\mathscr{D}$ is generically finite.

\medskip
\noindent{\textbf{Claim:}} For each point $\mathbf{x} \in (\mathscr{X}^{[m(M+2)]}_{S'}\setminus Z_{S'})(\IQbar)$, we have
\begin{enumerate}
\item[(i)] either $\hat{h}_{\cL^{\boxtimes m(M+2)}}(\mathbf{x}) \ge \delta_{S'}$, 
\item[(ii)] or $\hat{h}_{\cL^{\boxtimes m(M+1)}}(\mathscr{D}(\mathbf{x})) \ge \delta_{S'}$.
\end{enumerate}
\medskip

Indeed, assume both conditions were violated. Then 
\eqref{equation::integral1} would hold true, 
and 
\eqref{equation::integral2} with $\mathbf{x}' = \mathscr{D}(\mathbf{x})$ would also hold true.

But $$\frac{1}{\#\mathbf{O}(\mathbf{x})}\sum_{\mathbf{y}\in \mathbf{O}(\mathbf{x})}f_{S',1}(\mathbf{y}) = \frac{1}{\#\mathbf{O}(\mathscr{D}(\mathbf{x}))}\sum_{\mathbf{y}^\prime\in \mathbf{O}(\mathscr{D}(\mathbf{x}))}f_{S',2}(\mathbf{y}^\prime)$$ because $f_{S',1} = f_{S',2} \circ \mathscr{D}$. So we would have
\[
\left|\int_{\mathscr{X}^{[m(M+2)]}_{S'}(\mathbb{C})} f_{S',1}\mu_{S',1}-  \int_{\mathscr{D}(\mathscr{X}^{[m(M+2)]}_{S'})(\mathbb{C})} f_{S',2} \mu_{S',2} \right| \le 2\varepsilon_{S'},
\]
which contradicts \eqref{EqIntegralDiff} as $f_{S',1} = f_{S',2} \circ \mathscr{D}$. This finishes the proof of the claim.

\smallskip

Before moving on, let us take a closer look at the constants and objects obtained up to now. In Step~1, we introduced the integers $m= m(S) > 0$ and $M = M(S) > 0$, which depend only on $S$. The measures $\mu_{S',1}$ and $\mu_{S',2}$ from Step~2 clearly depend only on $S^\prime$, $m$ and $M$, and so do the test functions $f_{S',1}$, $f_{S',2}$ and the constant $\varepsilon_{S'} > 0$ chosen in Step~3. In this step, we obtained a closed subvariety $Z_{S'} \subsetneq \mathscr{X}^{[m(M+2)]}_{S'}$ and a constant $\delta_{S'}$ for each irreducible component $S'$ of $\overline{S_{\mathrm{gen}}}$; these depend only on $f_{S',1}$, $f_{S',2}$ and $\varepsilon_{S'} > 0$ and hence ultimately only on $S^\prime$, $m$ and $M$.

\medskip
\noindent\underline{\textbf{Step~5:} Conclusion and induction step}
\smallskip

The argument of this step is similar to the proof of 
Proposition~\ref{PropHtIneqStep1Aux}. The main difference is that the the height inequality \cite[Thm.~1.6 and B.1]{DGHUnifML} is replaced by the height lower bound from Step~4.

By abuse of notation, let $\pi \colon \cA_{S'} \rightarrow S'$ denote the structural morphism.
Consider the complement $S' \setminus \pi(\mathscr{X}_{S'}^{[m(M+2)]}\setminus Z_{S'})$ endowed with its reduced induced subscheme structure; it has dimension $\le \dim S' - 1 \le \dim S - 1$. Let $S_1,\ldots,S_k$ be the irreducible components of $\bigcup_{S'} S' \setminus \pi(\mathscr{X}_{S'}^{[m(M+2)]}\setminus Z_{S'})$ with $S'$ running over all irreducible components of $\overline{S_{\mathrm{gen}}}$. The set $\{S_1,\ldots,S_k\}$ is uniquely determined by $S$, $m$ and $M$. As both $m$ and $M$ depend only on $S$, the set $\{S_1,\ldots,S_k\}$ is  actually completely determined by $S$.

As in the proof of Proposition \ref{PropHtIneqStep1Aux}, we need to distinguish two cases in order to prove the proposition for a given $s \in S_{\mathrm{gen}}(\overline{\mathbb{Q}})$.

\smallskip
\noindent\underline{\textbf{Case (i): $s \in \bigcup_{i=1}^k S_i(\IQbar)$.}} 
\smallskip

Since $\dim S_i \le \dim S - 1$ for all $i \in \{1,\ldots,k\}$, we can apply the induction hypothesis for each subvariety $S_i$ separately. This yields constants $c_{i,2}'' = c_{i,2}''(r,d,S_i) > 0$ and $c_{i,3}'' = c_{i,3}''(r,d,S_i) >0$ such that, for each $s \in S_i(\IQbar)$, the set
\begin{equation}
\Sigma_{i,s}:=\left\{ x \in \mathscr{X}_s^{\circ}(\IQbar) : \hat{h}_{\cL}(x) \le c_{i,3}'' \right\}
\end{equation}
is contained in $X'(\IQbar)$ for a proper Zariski closed  $X' \subsetneq \mathscr{X}_s$ with $\deg_{\cL_s}(X') < c_{i,2}''$.

Set $c_{\mathrm{deg},2}'' := \max_{1\le i\le k} \{c_{i,2}''\}> 0$ and $c_{\mathrm{deg},3}'' := \min_{1\le i\le k} \{c_{i,3}''\}>0$. 
As the set $\{S_1,\ldots,S_k\}$ is completely determined by $S$, the constants $c_{\mathrm{deg},2}''$ and $c_{\mathrm{deg},3}''$ depend only on $r$, $d$ and $S$. 

For each $s \in \bigcup_{i=1}^k S_i(\IQbar)$, the set
\begin{equation}\label{EqDegEqDistSet}
\Sigma_{\mathrm{deg},s} := \left\{ x \in \mathscr{X}_s^{\circ}(\IQbar) : \hat{h}_{\cL}(x) \le  c_{\mathrm{deg},3}'' \right\}
\end{equation}
must be contained in $\Sigma_{i,s}$ for some $i$. Therefore, $\Sigma_{\mathrm{deg},s}$ is contained in $X'(\IQbar)$ for some proper Zariski closed $X' \subsetneq \mathscr{X}_s$ with $\deg_{\cL_s}(X') < c_{\mathrm{deg},2}''$. This proves the assertion of the proposition in this case.

\smallskip
\noindent\ul{\textbf{Case (ii): $s \in \pi(\mathscr{X}^{[m(M+2)]}_{S'}\setminus Z_{S'})(\IQbar)$ for some irreducible component $S'$ of $\overline{S_{\mathrm{gen}}}$.}} 
\medskip

Set $c_{S',3}'' = \delta_{S'}/4m(M+2)$ and consider
\begin{equation}\label{EqNonDegEqDistSet}
\Sigma_{S',s} = \left\{ x \in \mathscr{X}_s^{\circ}(\IQbar) : \hat{h}_{\cL}(x) \le c_{S',3}'' \right\} \subseteq \mathscr{X}_s(\IQbar). 
\end{equation}

We claim that 
\begin{equation}
\label{equation::claim}
\Sigma_{S',s}^{m(M+2)} \subseteq Z_s(\IQbar).
\end{equation}
Assume to the contrary that there exists some $\mathbf{x} = (x_1,\ldots,x_{m(M+2)}) \in (\Sigma_{S',s}^{m(M+2)} \setminus Z_s)(\IQbar)$. Then,
\begin{equation*}
\hat{h}_{\cL^{\boxtimes m(M+2)}}(\mathbf{x}) = \sum_{i=1}^{m(M+2)}\hat{h}_{\cL}(x_i) \le m(M+2) c_{S',3}'' < \delta_{S'}   
\end{equation*}
and
\begin{equation*}
\hat{h}_{\cL^{\boxtimes m(M+1)}}(\mathscr{D}(\mathbf{x})) = \sum_{i=1}^m \hat{h}_{\cL}(x_i) + \sum_{i=2}^{M} \sum_{j=1}^m \hat{h}_{\cL}(x_{i\cdot m + j} - x_{m + j}) \leq 4 m(M+1) c_{S',3}'' < \delta_{S^\prime}.
\end{equation*}
This contradicts the alternative height bounds obtained in Step~4, establishing the claim.

By the assumption of this case, $\mathscr{X}_s^{m(M+2)} \not= Z_s$. Therefore, we can apply Lemma~\ref{LemmaNogaAlon} to $X = \mathscr{X}_s$, $L = \cL_s|_{\mathscr{X}_s}$, $Z = Z_s$ and $\Sigma = \Sigma_{S',s}$. In this way, we obtain a proper Zariski closed  $X'\subsetneq \mathscr{X}_s$ such that
\begin{enumerate}
\item[(i)] $\deg_{\cL_s} (X') \le c(m(M+2), r, d, \deg_{\cL_s^{\boxtimes m}}Z_s)$, and
\item[(ii)] $\Sigma_{S',s} \subseteq X'(\IQbar)$.
\end{enumerate}

As the degree $\deg_{\cL_s^{\boxtimes m}}Z_s$ depends only on $S'$, $m$ and $M$, property (i) simplifies to $\deg_{\cL_s} (X') \le c_{S',2}''$ with a constant $c_{S',2}''$ depending only on $r$, $d$, $S'$, $m$ and $M$. In summary, the set $\Sigma_{S',s}$ 
is contained in $X'(\IQbar)$ for some Zariski closed  $X' \subsetneq \mathscr{X}_s$ with $\deg_{\cL_s}(X') \le c_{S',2}''$.

Set $c_2^{*\prime\prime} := \max_{S'} \{c_{S',2}''\}> 0$ and $c_3^{*\prime\prime} := \min_{S'} \{c_{S',3}''\}>0$ where $S'$ runs over all irreducible components of $\overline{S_{\mathrm{gen}}}$. 
The subset $S_{\mathrm{gen}}$ of $S$ is uniquely determined by $S$. Hence the set $\{S'\}$ of irreducible components of $\overline{S_{\mathrm{gen}}}$ is uniquely determined by $S$. Moreover $m = m(S) > 0$ and $M = M(S) > 0$ depend only on $S$. Consequently, the constants $c_2^{*\prime\prime}$ and $c_3^{*\prime\prime}$ depend only on $r$, $d$ and $S$. 

For each $s \in \bigcup_{S'} \pi(\mathscr{X}^{[m(M+2)]}_{S'}\setminus Z_{S'})(\IQbar)$, the set
\begin{equation}\label{EqNonDegEqDistUnionSet}
\Sigma_s^* := \left\{ x \in \mathscr{X}_s^{\circ}(\IQbar) : \hat{h}_{\cL}(x) \le c_3^{*\prime\prime} \right\}
\end{equation}
must be contained in $\Sigma_{S',s}$ for some $S'$. So $\Sigma_s^*$ 
is contained in $X'(\IQbar)$ for a proper Zariski closed  $X' \subsetneq \mathscr{X}_s$ with $\deg_{\cL_s}(X') < c_2^{*\prime\prime}$. This concludes the induction step in this case.
\end{proof}

\subsection{Proof of Proposition~\ref{PropEqDistStep1}}
Proposition~\ref{PropEqDistStep1} can be deduced from Proposition~\ref{PropEqDistStep1Aux} by an almost verbal copy of the proof of Proposition~\ref{PropHtIneqPPAV} and the argument in $\mathsection$\ref{SubsectionProofOfHtIneqApp}. Even more, the arguments needed here are simpler because we no longer need to deal with the Faltings height $h_{\mathrm{Fal}}(A)$. Instead of repeating the proof, we hereby give a brief summary and leave it to the reader to supplement details from the previous sections.

Let $(A,L)$ be a polarized abelian variety defined over $\IQbar$ with $\deg_L A = l$. Let $X$ be an irreducible subvariety of $A$ with $\dim X = r$ and $\deg_L X = d$ such that $X$ generates $A$. We may  assume $r \ge 1$; otherwise the  proposition is trivial.

We start with the case where $(A,L)$ is principally polarized, \textit{i.e.}\ $l = g!$. 

Recall the universal abelian variety $\mathfrak{A}_g \rightarrow \mathbb{A}_g$ and the symmetric relatively ample  line bundle $\mathfrak{L}_g$ on $\mathfrak{A}_g/\mathbb{A}_g$ from the beginning of $\mathsection$\ref{SubsectionEqDistFamilies}. 
The pair $(A,L)$ gives rise to a point $b \in \mathbb{A}_g(\IQbar)$ such that $(\mathfrak{A}_g)_b \cong A$ and $c_1(\mathfrak{L}_g|_A) = 2c_1(L) = c_1(L \otimes L_{-})$ for $L_{-} = [-1]^*L$. By \eqref{EqDegreeWRTLineBundleOverUnivFamily}, $\deg_{\mathfrak{L}_g^{\otimes 4}|_A}X = 8^r d$. For the line bundle $\cL = \iota^*\mathfrak{L}_g^{\otimes 4}$, we have seen in \eqref{EqNTHeightChangetoL} that $\hat{h}_{\cL|_A} = 4\hat{h}_{L\otimes L_{-}}$ as height functions on $A(\IQbar)$.

Consider the restricted Hilbert scheme $\mathbf{H}^{\circ}:=\mathbf{H}^{\circ}_{r, 8^r d}(\mathfrak{A}_g/\mathbb{A}_g)$ from \eqref{EqHilbSchemeOpenLocusFamily} and retain the commutative diagram \eqref{EqSubsectionEqDistFamiliesHilbDiagram}
\[
\xymatrix{
\mathscr{X}_{\mathbf{H}^{\circ}} \ar@{^(->}[r] \ar[rd] & 
\cA_{\mathbf{H}^{\circ}} \ar[d]^-{\pi} \ar[r]^-{\iota} \pullbackcorner & \mathfrak{A}_g \ar[d]^-{\pi^{\mathrm{univ}}}  \\
& \mathbf{H}^{\circ} \ar[r]^-{\iota_{\mathbf{H}^{\circ}}} & \mathbb{A}_g
}
\]
where $\pi|_{\mathscr{X}_{\mathbf{H}^{\circ}}} \colon \mathscr{X}_{\mathbf{H}^{\circ}} \rightarrow \mathbf{H}^{\circ}$ is the universal family. All varieties and morphisms in this diagram are defined over $\IQbar$. The pair $(X, (A,L))$ is parametrized by a point $s \in \mathbf{H}^{\circ}(\IQbar)$, which means that $X = \mathscr{X}_s$, $A = \cA_s = (\mathfrak{A}_g)_b$, and $\iota_{\mathbf{H}^{\circ}}(s) = b$.

Let $\mathbf{H}^{\circ}_{\mathrm{gen}}$ be the subset of $\mathbf{H}^{\circ}$ as defined by \eqref{EqConstructibleSet}. If the stabilizer $\mathrm{Stab}_A(X)$) of $X$ in $A$ has positive dimension, then $X^{\circ} = \emptyset$ by  definition of the Ueno locus and the proposition trivially holds true. Thus we may and do assume that $\mathrm{Stab}_A(X)$ is finite. Then the point $s \in \mathbf{H}^{\circ}(\IQbar)$ which parametrizes $(X, (A,L))$ is in $\mathbf{H}^{\circ}_{\mathrm{gen}}$

We apply Proposition~\ref{PropEqDistStep1Aux} to $S = \mathbf{H}^{\circ}$, and hence obtain constants $c_2''$ and $c_3''$ that depend only on $r$, $8^r d$ and $\mathbf{H}^{\circ}$ (so only on $g$, $d$ and $r$)  such that the set 
\[
\Sigma := \left\{ x \in \mathscr{X}_s^{\circ}(\IQbar) : \hat{h}_{\cL|_A}(x) \le  c_3'' \right\} = \left\{ x \in X^{\circ}(\IQbar) : \hat{h}_{L\otimes L_{-}}(x) \le  c_3''/4 \right\}
\] 
is contained in some proper Zariski closed  $X' \subsetneq X$ with $\deg_{\cL|_A}(X') \le c_2''$. Thus Proposition~\ref{PropEqDistStep1} for this case holds true because $\deg_L X' < \deg_{\cL|_A}(X')$ (since $c_1(\cL|_A) = 4c_1(\mathfrak{L}_g|_A) = 8c_1(L)$).

Now let us turn to arbitrary polarized abelian varieties $(A,L)$. By Lemma~\ref{LemmaAbIsogPPAV}.(iv), there exist a principally polarized abelian variety $(A_0,L_0)$ defined over $\IQbar$ and an isogeny $u \colon A_0 \rightarrow A$ such that $u^*L = [l/g!]^*L_0$; see below \eqref{EqFaltingsHeight}. In particular, $u^*(L\otimes L_{-}) \cong (L_0\otimes (L_0)_{-})^{\otimes (l/g!)^2}$, where $(L_0)=[-1]^*L_0$.

Let $X_0$ be an irreducible component of $u^{-1}(X)$. Then $\deg_{u^*L}X_0 \le d(l/g!)^{2g-1}$ by \eqref{EqDegreeChangeIsogeny}, and $X_0$ is not contained in any proper subgroup of $A_0$. Thus we can apply the conclusion for the principally polarized case to $X_0 \subseteq A_0$ and get two constants $c_2''$ and $c_3''$ depending only on $g$, $r$, $l$ and $d$.\footnote{Again, we first of all get constants $c_{0,2}''(g,r,d') > 0$ and $c_{0,3}''(g,r,d') > 0$ for each $1 \le d' \le d$, and then set $c_2'' := \max_{1\le d' \le d(l/g!)^{2g-1}}\{c_{0,2}''(g,r,d')\}$ and $c_3'' := \min_{1 \le d' \le d(l/g!)^{2g-1}}\{c_{0,3}''(g,r,d')\}$.}

By  definition of the Ueno locus, we have $X^{\circ} = u(X_0)^{\circ} \subseteq u(X_0^{\circ})$. 
So we have $\Sigma \subseteq u(\Sigma_0)$ for $$\Sigma:=\left\{ x \in X^{\circ}(\IQbar) : \hat{h}_{L\otimes L_{-}}(x) \le  c_3'' \right\}$$ and $$\Sigma_0 = \left\{ x_0 \in X_0^{\circ}(\IQbar) : \hat{h}_{L_0 \otimes (L_0)_{-}}(x_0) \le (l/g!)^2 c_3'' \right\}.$$ The conclusion follows from the principally polarized case upon replacing $c_3''$ by $(l/g!)^2 c_3''$. \qed

\section{Proof of the gap principle (Theorem \ref{ThmSmallPointHighDim})}\label{SectionEndOfNGP}
In this section we combine Proposition~\ref{PropHtIneqStep1} and Proposition~\ref{PropEqDistStep1} to finish the proof of the \textit{generalized New Gap Principle} (Theorem~\ref{ThmSmallPointHighDim}) with the same argument for curves \cite[Prop.~9.2]{GaoSurveyUML}; this argument is eventually related to \cite[Prop.~2.3]{DGHBog}.
\begin{prop}\label{ThmSmallPointHighDimEqMainBody}
Let $g$, $l$, $r$, and $d$ be positive integers. There exist constants $c_1 = c_1(g,l,r,d) > 0$ and $c_2 = c_2(g,l,r,d) > 0$  with the following property. For 
\begin{itemize}
\item each polarized abelian variety $(A,L)$ of dimension $g$ defined over $\IQbar$ with $\deg_L A = l$,
\item and each irreducible subvariety $X$ of $A$ with $\dim X = r$ and $\deg_L X = d$, such that $X$ generates $A$,
\end{itemize}
the set
\begin{equation}\label{EqSmallPointHighDimEq}
\left\{P \in X^{\circ}(\IQbar): \hat{h}_{L\otimes [-1]^*L}(P) \le c_1 \max\{1,h_{\mathrm{Fal}}(A)\} \right\}
\end{equation}
is contained in $X'(\IQbar)$ for a proper Zariski closed $X' \subsetneq X$ with $\deg_L X' < c_2$.
\end{prop}

\begin{proof}
Let $(A,L)$ and $X$ be as in the proposition. Denote for simplicity by $L_{-} := [-1]^*L$.

By Proposition~\ref{PropHtIneqStep1}, there exist constants $c_1' = c_1'(g,l,r,d) > 0$ ,$c_2' = c_2'(g,l,r,d) > 0$ and $c_3'= c_3'(g,l,r,d)>0$ such that
\begin{equation}\label{EqSet1}
\left\{P \in X^{\circ}(\IQbar): \hat{h}_{L\otimes L_{-}}(P) \le c_1' \max\{1,h_{\mathrm{Fal}}(A)\} - c_3' \right\}
\end{equation}
is contained in a proper Zariski closed $X' \subseteq X$ with $\deg_L X' < c_2'$.

By Proposition~\ref{PropEqDistStep1}, there exist constants $c_2'' = c_2''(g,l,r,d)>0$ and $c_3'' = c_3''(g,l,r,d) > 0$ such that
\begin{equation}\label{EqSet2}
\left\{P \in X^{\circ}(\IQbar): \hat{h}_{L\otimes L_{-}}(P) \le  c_3'' \right\}
\end{equation}
is contained in a proper Zariski closed $X' \subseteq X$ with $\deg_L X' < c_2''$.

Now set
\begin{equation}
c_1 := \min\left\{\frac{c_3''}{\max\{1,2c_3'/c_1'\}}, \frac{c_1'}{2}\right\} \quad \text{ and } \quad c_2:= \max\{c_2',c_2''\}.
\end{equation}
We will prove that these are the desired constants.

To prove this, it suffices to prove the following claim.

\smallskip
\noindent\ul{\textbf{Claim: If $P \in X^{\circ}(\IQbar)$ satisfies $\hat{h}_{L\otimes L_{-}}(P) \le c_1 \max\{1,h_{\mathrm{Fal}}(A)\} $, then $P$ is in either the set {\eqref{EqSet1}} or the set {\eqref{EqSet2}}.}}
\smallskip

Let us prove this claim. 
Suppose $P \in X^{\circ}(\IQbar)$ is not in \eqref{EqSet1} or \eqref{EqSet2}, \textit{i.e.}, $\hat{h}_{L\otimes L_{-}}(P) > c_1' \max\{1,h_{\mathrm{Fal}}(A)\} - c_3'$ and $\hat{h}_{L\otimes L_{-}}(P) > c_3''$. We wish to prove $\hat{h}_{L\otimes L_{-}}(P) > c_1 \max\{1,h_{\mathrm{Fal}}(A)\} $.

We split up to two cases on whether or not $\max\{1,h_{\mathrm{Fal}}(A)\}  \le \max\{1,2c_3'/c_1'\}$.

In the first case, \textit{i.e.}, $\max\{1,h_{\mathrm{Fal}}(A)\}  \le \max\{1,2c_3'/c_1'\}$, we have
\[
\hat{h}_{L\otimes L_{-}}(P) > c_3'' \ge c_3'' \frac{\max\{1,h_{\mathrm{Fal}}(A)\}}{\max\{1,2c_3'/c_1'\}} = \frac{c_3''}{\max\{1,2c_3'/c_1'\}} \max\{1,h_{\mathrm{Fal}}(A)\} \ge c_1 \max\{1,h_{\mathrm{Fal}}(A)\}.
\]
In the second case,  \textit{i.e.}, $\max\{1,h_{\mathrm{Fal}}(A)\} > \max\{1,2c_3'/c_1'\}$, we have $c_1' \max\{1,h_{\mathrm{Fal}}(A)\} -c_3' \ge (c_1'/2) \max\{1,h_{\mathrm{Fal}}(A)\}$ and hence
\[
\hat{h}_{L\otimes L_{-}}(P) > \frac{c_1'}{2}\max\{1,h_{\mathrm{Fal}}(A)\} \ge c_1 \max\{1,h_{\mathrm{Fal}}(A)\}.
\]
Hence we are done.
\end{proof}

\begin{proof}[Proof of Theorem~\ref{ThmSmallPointHighDim}]
Let $A$ be an abelian variety of dimension $g$, let $L$ be an ample line bundle, and let $X$  be an irreducible subvariety which generates $A$. Assume all these objects are defined over $\IQbar$. Then $(A,L)$ is a polarized abelian variety.

Write $d = \deg_L X$, $r = \dim X$, $l = \deg_L A$. Then $l \le c(g,d)$ by Lemma~\ref{LemmaSmallestAbVarGenerated}.

Since $X$ generates $A$, we have that $r \ge 1$. 
Thus we can apply Proposition~\ref{ThmSmallPointHighDimEqMainBody} to $(A,L)$ and $X$. Then we obtain constants $c_1 = c_1(g,l,r,d) > 0$ and $c_2 = c_2(g,l,r,d)>0$ such that the set
\[
\Sigma:=\left \{ P \in X^{\circ}(\IQbar) : \hat{h}_{L\otimes [-1]^*L}(P) \le c_1\max\{1, h_{\mathrm{Fal}}(A)\} \right \}
\]
is contained in a proper Zariski closed $X' \subsetneq X$ with $\deg_L X' < c_2$.

Now we can conclude by replacing $c_1$ by $\min_{1 \le r \le g, 1 \le l \le c(g,d)}\{c_1(g,l,r,d)\} > 0$ and replacing $c_2$ by $\max_{1\le r\le g, 1 \le l \le c(g,d)} \{c_2(g,l,r,d)\} > 0$.
\end{proof}

\section{Proof of the uniform Mordell--Lang conjecture (Theorem \ref{MainThm2})}\label{SectionUML}

\subsection{A theorem of R\'{e}mond}
In this subsection, we work over $\IQbar$.

We start by recalling the following result, which is a consequence of R\'{e}mond's generalized  Vojta's Inequality \cite[Thm.~1.2]{Remond:Vojtasup} for points in  $X^{\circ}(\IQbar)$, the generalized Mumford's Inequality \cite[Thm.~3.2]{Remond:Decompte} for points in $X^{\circ}(\IQbar) \cap \Gamma$, and the technique to remove the height of the subvariety \cite[$\mathsection$3.b)]{Remond:Decompte}.

\medskip

Let $A$ be an abelian variety of dimension $g$, and $L$ be a \textit{symmetric} ample line bundle on $A$.

Let $X$ be an irreducible subvariety of $A$, and $\Gamma$ be a finite rank subgroup of $A(\IQbar)$. We say that \textit{the assumption {\tt (Hyp pack)} holds true for $(A,L)$, $X$ and $\Gamma$}, if there exists a constant $c_0 = c_0(g, \deg_L X) > 0$ satisfying the following property: for each $P_0 \in X(\IQbar)$,
\begin{equation}\label{EqRemondAssumption}
\#\left\{P - P_0 \in (X^{\circ}(\IQbar)-P_0) \cap \Gamma : \hat{h}_L(P - P_0) \le c_0^{-1} \max\{1,h_{\mathrm{Fal}}(A)\} \right\} \le c_0^{\mathrm{rk}\Gamma + 1}.
\end{equation}

\begin{thm}[R\'{e}mond]\label{ThmRemond}
Assume that {\tt (Hyp pack)} holds true for all $(A,L)$, $X$, $\Gamma$ (as above) such that $X$ generates $A$.

Then for each polarized abelian variety $(A,L)$ with $L$ symmetric, each irreducible subvariety $X$ of $A$ and each  finite rank subgroup $\Gamma$ of $A(\IQbar)$,  we have
\begin{equation}\label{EqRemond}
\#X^{\circ}(\IQbar) \cap \Gamma \le c(g,\deg_L X, \deg_L A)^{\mathrm{rk}\Gamma + 1}.
\end{equation}
\end{thm}

A more detailed proof of this theorem can be found in Appendix~\ref{AppendixRemond}. We hereby give a brief explanation by taking David--Philippon's formulation of R\'{e}mond's result.

\begin{proof}
We start with a d\'{e}vissage. 
More precisely, we reduce to the case where $X$ generates $A$. 
Indeed, let $A'$ be the abelian subvariety of $A$ generated by $X-X$. Then $X \subseteq A' + Q$ for some $Q \in A(\IQbar)$. The subgroup $\Gamma'$ of $A(\IQbar)$ generated by $\Gamma$ and $Q$ has rank $\le \mathrm{rk}\Gamma + 1$. We have $(X-Q)^{\circ} = X^{\circ} - Q$ by definition of the Ueno locus, $(X^{\circ}(\IQbar)-Q) \cap \Gamma  \subseteq  (X^{\circ}(\IQbar) - Q) \cap \Gamma' = X^{\circ}(\IQbar) \cap \Gamma'$ and $\deg_L(X-Q) = \deg_L X$. By Lemma~\ref{LemmaSmallestAbVarGenerated}, we have $\deg_L A' \le c'(g,\deg_L X)$. Therefore 
if \eqref{EqRemond} holds true for $X-Q \subseteq A'$, $L|_{A'}$ and $\Gamma' \cap A'(\IQbar)$, then $\#X^{\circ}(\IQbar) \cap \Gamma \le c(g,\deg_L X)^{\mathrm{rk}\Gamma' + 1} \le  c(g,\deg_L X)^{\mathrm{rk}\Gamma + 2}$. So we can conclude by replacing $c$ with $c^2$. Thus we are reduced to the case where $X$ generates $A$.

From now on, assume $X$ generates $A$. We take the formulation of David--Philippon \cite[Thm.~6.8]{DaPh:07} of R\'{e}mond's result \cite{Remond:Decompte}.

It follows  from Tate's construction  that there exists a constant $c_{\mathrm{NT}}(A)\ge 0$, which depends on $A$, such that $|\hat h_{L}(P) - h(P)|\le c_{\mathrm{NT}}(A)$ for all $P\in A(\IQbar)$. 

Let $h_1(A)$ denote the Weil height of the polynomials defining the addition and the substraction on $A$.

It is known that $c_{\mathrm{NT}}(A), h_1(A) \le c' \max\{1,h_{\mathrm{Fal}}(A)\}$ for some $c' = c'(g,\deg_L A) >0$; see \cite[equation (6.41)]{DaPh:07}. Alternatively this can be deduced from \cite[(8.4) and (8.7)]{DGHUnifML}.

Let $\eta \ge 1$ be a real number. Then 
\begin{multline}\label{EqRemondDP2}
\#\left\{ P \in X^{\circ}(\IQbar) \cap \Gamma :  \hat{h}_L(P) \le \eta \max\{1,c_{\mathrm{NT}}(A), h_1(A)\} \right\} \\
\le \#\left\{ P \in X^{\circ}(\IQbar) \cap \Gamma :  \hat{h}_L(P) \le \eta c' \max\{1,h_{\mathrm{Fal}}(A)\} \right\}.
\end{multline}

Since $X$ generates $A$,  we can apply \eqref{EqRemondAssumption} to $X - P_0$ for each $P_0 \in X(\IQbar)$.

Set $R = (\eta c' \max\{1, h_{\mathrm{Fal}}(A)\})^{1/2}$ and $r = (c_0^{-1} \max\{1,h_{\mathrm{Fal}}(A)\})^{1/2}$ with $c_0$ from \eqref{EqRemondAssumption}.

Consider the real vector space $\Gamma\otimes_{\mathbb{Q}}\mathbb{R}$ endowed with the Euclidean norm $| \cdot | = \hat{h}_L^{1/2}$. 
By an elementary ball packing argument, any subset of $\Gamma \otimes\mathbb{R}$ contained in a closed ball of radius $R$ centered at $0$ is covered by at most $(1+2R/r)^{\mathrm{rk}\Gamma}$ closed balls of radius $r$ centered at the elements $P - P_0$ of the given subset \eqref{EqRemondAssumption}; see \cite[Lem.~6.1]{Remond:Decompte}. Thus the number of balls in the covering is at most $(1+2\sqrt{\eta c' c_0})^{\mathrm{rk}\Gamma}$. 
 But each closed ball of radius $r$ centered at some $P - P_0$ in \eqref{EqRemondAssumption} contains at most $c$ elements by \eqref{EqRemondAssumption}. So
\begin{equation}\label{EqRemondDP3}
 \#\left\{ P \in X^{\circ}(\IQbar) \cap \Gamma : \hat{h}_L(P) \le \eta c' \max\{1, h_{\mathrm{Fal}}(A)\} \right\} \le c_0 (1+2\sqrt{\eta c' c_0})^{\mathrm{rk}\Gamma}.
\end{equation}
 
 So \eqref{EqRemondDP2} and \eqref{EqRemondDP3} yield, for each real number $\eta \ge 1$, 
\begin{equation}
\#\left\{ P \in X^{\circ}(\IQbar) \cap \Gamma :  \hat{h}_L(P) \le \eta \max\{1,c_{\mathrm{NT}}(A), h_1(A)\} \right\} \le c_0 (1+2\sqrt{\eta c' c_0})^{\mathrm{rk}\Gamma}.
\end{equation}
Thus  \cite[Thm.~6.8]{DaPh:07} implies
\begin{equation}
\#X^{\circ}(\IQbar) \cap \Gamma \le (c'')^{\mathrm{rk}\Gamma + 1} \cdot c_0 (1+2\sqrt{c'' c' c_0})^{\mathrm{rk}\Gamma}
\end{equation}
for some $c'' = c''(g, \deg_L X) > 0$. Therefore \eqref{EqRemond} holds true by letting $c = (c'' c_0(1+2\sqrt{c'' c' c_0}))^2$.
\end{proof}

\subsection{Proof of Theorem~\ref{MainThm3} over $\IQbar$}\label{SubsectionUMLIQbar}
Now we are ready to prove Theorem~\ref{MainThm3} over $\IQbar$. In view of R\'{e}mond's result (Theorem~\ref{ThmRemond}) cited above, the most important ingredient is the following proposition.
\begin{prop}\label{PropHypPack}
Let $A$ be an abelian variety of dimension $g$ and $L$ be an ample line bundle on $A$. 
Then with $L$ replaced  by $L\otimes [-1]^*L$, {\tt (Hyp pack)} holds true for each  irreducible subvariety $X$ of $A$ which generates $A$ and each finite rank subgroup $\Gamma$ of $A(\IQbar)$.
\end{prop}
\begin{proof}
Write $d = \deg_L X$. Denote for simplicity by $L_-:= [-1]^*L$.

We prove this result by induction on $r := \dim X$.

The base step is $r = 0$, in which case trivially holds true.

For an arbitrary $r \ge 1$. Assume the theorem is proved for $0, 1, \ldots, r-1$.

We wish to prove \eqref{EqRemondAssumption} with $L$ replaced by $L\otimes L_{-}$. Let $P_0 \in X(\IQbar)$. Then $\deg_L (X-P_0) = \deg_L X = d$. Moreover, $(X-P_0)^{\circ}(\IQbar) = X^{\circ}(\IQbar) - P_0$ by definition of the Ueno locus. Notice that $X-P_0$ still generates $A$ because $(X-P_0) - (X-P_0) = X-X$.

Apply Theorem~\ref{ThmSmallPointHighDim} to $X-P_0$. Then we have constants $c_1 = c_1(g,d) > 0$ and $c_2 = c_2(g,d) > 0$ such that for $\left\{P - P_0 \in X^{\circ}(\IQbar) - P_0 : \hat{h}_{L\otimes L_{-}}(P - P_0) \le c_1 \max\{1,h_{\mathrm{Fal}}(A)\} \right\}$ is contained in a proper Zariski closed $X' \subsetneq X - P_0$ with $\deg_L X' < c_2$. In particular, the number of irreducible components of $X'$ is $< c_2$.

Let $X^{\dagger}$ be an irreducible component of $X'$. Then $\dim X^{\dagger} \le \dim X - 1 \le r-1$. Let $A^{\dagger}$ be the abelian subvariety of $A$ generated by $X^{\dagger} - X^{\dagger}$. Then $X^{\dagger} \subseteq A^{\dagger} + Q$ for some $Q \in A(\IQbar)$. Now $\dim (X^{\dagger} - Q) = \dim X^{\dagger} \le r-1$.

Let $\Gamma^{\dagger}$ be the subgroup of $A(\IQbar)$ generated by $\Gamma$ and $Q$. Then $\mathrm{rk}\Gamma^{\dagger} \le \mathrm{rk}\Gamma + 1$. 
Apply the induction hypothesis to $A^{\dagger}$, $L|_{A^{\dagger}}$, the irreducible subvariety $X^{\dagger} - Q$ and $\Gamma^{\dagger} \cap A^{\dagger}(\IQbar)$. Then we have
\begin{equation}\label{EqUMLBoundIrreCompo}
\# (X^{\dagger} - Q)^{\circ}(\IQbar) \cap \Gamma^{\dagger} \le  (c^\dagger)^{\mathrm{rk}\Gamma^{\dagger} + 1} \le (c^{\dagger})^{\mathrm{rk}\Gamma + 2}
\end{equation}
for some $c^{\dagger} = c^\dagger(\dim A^{\dagger}, \deg_L(X^{\dagger}-Q)) > 0$. But $\dim A^{\dagger} \le \dim A = g$ and $\deg_L(X^\dagger -Q) = \deg_L X^{\dagger} \le \deg_L X' < c_2 = c_2(g,d)$.

But $(X^{\dagger} - Q)^{\circ}(\IQbar) = (X^{\dagger})^{\circ}(\IQbar) - Q$ by definition of the Ueno locus, and $Q \in \Gamma^{\dagger}$. So \eqref{EqUMLBoundIrreCompo} yields
\begin{equation}\label{EqUMLBoundIrreCompo2}
\# (X^{\dagger})^{\circ}(\IQbar) \cap \Gamma^{\dagger} \le  (c^\dagger)^{\mathrm{rk}\Gamma + 2}
\end{equation}

Now that $(X')^{\circ}(\IQbar) \cap \Gamma \subseteq \bigcup_{X^{\dagger}}(X^{\dagger})^{\circ}(\IQbar) \cap \Gamma^{\dagger}$ 
with $X^{\dagger}$ running over all irreducible components of $X'$ and $\Gamma^{\dagger}$ constructed accordingly,  \eqref{EqUMLBoundIrreCompo2} implies
\begin{equation}\label{EqDiscretePointsOnUenoLocus}
\# (X')^{\circ}(\IQbar) \cap \Gamma \le c_2 \max_{X^{\dagger}}\{c^{\dagger}\}^{\mathrm{rk}\Gamma+2} \le c_3^{\mathrm{rk}\Gamma + 1}
\end{equation}
where $c_3 = (\max_{X^{\dagger}}\{c_2, c^{\dagger}\})^3 > 0$ depends only on $g$ and $d$.

But $\left\{P - P_0 \in X^{\circ}(\IQbar) - P_0: \hat{h}_{L\otimes L_{-}}(P - P_0) \le c_1 \max\{1,h_{\mathrm{Fal}}(A)\} \right\} \subseteq X'(\IQbar)$ by construction of $X'$. Moreover, $(X')^{\circ} \supseteq X^{\circ} \cap X'$ by definition of the Ueno locus. So \eqref{EqDiscretePointsOnUenoLocus} yields
\begin{equation}
\left\{P \in (X^{\circ} (\IQbar) - P_0) \cap \Gamma : \hat{h}_{L\otimes L_{-}}(P- P_0) \le c_0^{-1} \max\{1,h_{\mathrm{Fal}}(A)\} \right\} \le c_0^{\mathrm{rk}\Gamma + 1}
\end{equation}
with $c_0 = \max\{c_1,c_3\} > 0$ which depends only on $g$ and $d$. Hence we are done.
\end{proof}

\begin{proof}[Proof of Theorem~\ref{MainThm3} with $F = \IQbar$]
Let $A$ be an abelian variety of dimension $g$ and $L$ be an ample line bundle on $A$. Let $X$ be a closed irreducible subvariety of $A$. Assume that all varieties are defined over $\IQbar$. Set $l = \deg_L A$ and $d = \deg_L X$. Let $\Gamma$ be a subgroup of $A(\IQbar)$ of finite rank.

Notice that $\deg_{L\otimes [-1]^*L} X = 2^{\dim X}\deg_L X \le 2^g d$ and $\deg_{L\otimes [-1]^*L} A = 2^g \deg_L A = 2^g l$.

Let $X^{\circ}$ be the complement of the Ueno locus of $X$. 

As before, we start by reducing to the case where $X$ generates $A$. 
Indeed, let $A'$ be the abelian subvariety of $A$ generated by $X-X$. Then $X \subseteq A' + Q$ for some $Q \in A(\IQbar)$. The subgroup $\Gamma'$ of $A(\IQbar)$ generated by $\Gamma$ and $Q$ has rank $\le \mathrm{rk}\Gamma + 1$. We have $(X-Q)^{\circ} = X^{\circ} - Q$ by definition of the Ueno locus, $(X^{\circ}(\IQbar)-Q) \cap \Gamma  \subseteq  (X^{\circ}(\IQbar) - Q) \cap \Gamma' = X^{\circ}(\IQbar) \cap \Gamma'$ and $\deg_L(X-Q) = \deg_L X$. 
If \eqref{EqBoundLatticePointsOutsideUeno} holds true for $X-Q \subseteq A'$, $L|_{A'}$ and $\Gamma' \cap A'(\IQbar)$, then $\#X^{\circ}(\IQbar) \cap \Gamma \le c(g,d)^{\mathrm{rk}\Gamma' + 1} \le  c(g,d)^{\mathrm{rk}\Gamma + 2}$. So we can conclude by replacing $c$ with $c^2$. Thus we are reduced to the case where $X$ generates $A$. In particular, we have $l \le c'(g,d)$ by Lemma~\ref{LemmaSmallestAbVarGenerated}.

\medskip

By Proposition~\ref{PropHypPack} and Theorem~\ref{ThmRemond}, we have $\#X^{\circ}(\IQbar) \cap \Gamma \le c(g,l,d)^{\mathrm{rk}\Gamma + 1}$. Thus \eqref{EqBoundLatticePointsOutsideUeno} holds true because $l \le c'(g,d)$. Hence we are done.
\end{proof}

\subsection{From $\IQbar$ to arbitrary $F$ of characteristic $0$}\label{subsectionspecialization}
The following lemma of specialization allows us to pass from $\IQbar$ to $F$. 
\begin{lemma}\label{LemmaSpecialization}
Assume Theorem~\ref{MainThm3} holds true for $F = \IQbar$. Then Theorem~\ref{MainThm3} holds true, under the extra assumptions that $\Gamma$ is finitely generated, for arbitrary $F$ of characteristic $0$ with $F = \overline{F}$.
\end{lemma}
\begin{proof} 
Assume Theorem~\ref{MainThm3} holds true for $F = \IQbar$. Then we obtain a function $c \colon \mathbb{N}^2 \rightarrow \mathbb{N}$ such that Theorem~\ref{MainThm3} holds true with this function $c$ viewed as a constant depending only on $g$ and the degree of the subvariety in question.

Now let $F$ be an arbitrary algebraic closed field of characteristic $0$. 
Let $A$, $L$, $X$ and $\Gamma$ be as in Theorem~\ref{MainThm3} with $\Gamma$ finitely generated. 
Write $\rho = \mathrm{rk}\Gamma$. Let $\gamma_1,\ldots,\gamma_{r} \in
A(F)$ be generators of $\Gamma$
with $\gamma_{\rho+1},\ldots,\gamma_{r}$ torsion.

There exists a field $K$, finitely generated over $\IQbar$, such that
$A$, all elements of $\mathrm{End}(A)$,\footnote{We can do this because $\mathrm{End}(A)$ is a finitely generated group.} $X$, and
 $\gamma_1,\ldots,\gamma_{r}$ are
defined over $K$. Then $K$ is the function field of
some regular, irreducible quasi-projective
variety $V$ defined over $\IQbar$.

Up to replacing $V$ by a Zariski open dense subset, we have
\begin{itemize}
\item $A$ extends to an abelian scheme $\cA \rightarrow V$ of relative dimension $g$,
\item $L$ extends to a relatively ample line bundle $\cL$ on $\cA / V$,
\item each element of $\mathrm{End}(A)$ extends to an element of $\mathrm{End}(\cA/V)$ (this can be achieved since $\mathrm{End}(A)$ is a finitely generated group); in particular, each abelian subvariety $B$ of $A$ extends to an abelian subscheme $\cB$ of $\cA \rightarrow V$,
\item $X$ extends to a flat family $\mathfrak{X} \rightarrow V$ (\textit{i.e.}, $X$ is the generic
fiber of $\mathfrak{X} \rightarrow V$),
\item $\gamma_1,\ldots,\gamma_{r}$ extend to sections of $\cA \rightarrow V$; we retain the symbols $\gamma_1,\ldots,\gamma_{r}$ for these sections. Use $\Gamma_V$ to denote the sub-$V$-group of $\cA$ generated by $\gamma_1,\ldots,\gamma_r$.
\end{itemize}
Moreover, up to replacing $V$ by a Zariski open dense subset, we may choose $\mathfrak{X}$ over $V$ to be a ``Good model'' as explained in \cite[p.~219--220]{Mazur:00} (which uses \cite[App.~1, Lemma~A]{Hindry:Lang}) such that the Ueno locus of $\mathfrak{X}_v$ is the specialization of the Ueno locus of $X$ for each $v \in V(\IQbar)$. Hence $\mathfrak{X}_v^{\circ}$ is the specialization of $X^{\circ}$.

As $\mathfrak{X} \rightarrow V$ is flat, we have $\deg_L X = \deg_{\cL_v}\mathfrak{X}_v$ for each $v \in V(\IQbar)$.

We define the specialization of $\Gamma$ at $v$, which we
denote with $\Gamma_v$, to be the subgroup of
$\cA_v(\IQbar)$ 
generated by
$\gamma_1(v),\ldots,\gamma_{r}(v)$. There exists then a specialization homomorphism $\Gamma \rightarrow \Gamma_v$ for each $v \in V(\IQbar)$.  Note that
$\mathrm{rk}\Gamma_v \le \rho$.

The extension of elements of $\mathrm{End}(A)$ to elements of $\mathrm{End}(\cA/V)$ yields a specialization $\mathrm{End}(A) \rightarrow \mathrm{End}(\cA_v)$ for each $v \in V(\IQbar)$. Denote this map by $\alpha \mapsto \alpha_v$. 

Set
\[
\Theta := \{v \in V(\IQbar) : \Gamma \rightarrow \Gamma_v\text{ is injective and }\mathrm{End}(A) \cong \mathrm{End}(\cA_v)\}.
\]
Masser \cite[Main Theorem and Scholium~1]{masser1989specializations} and \cite[Main Theorem]{masser1996specializations} guarantee that $\Theta$ is  Zariski dense in $V$.

Let $v \in \Theta$. Then $\#X^{\circ}(F) \cap \Gamma \le \#\mathfrak{X}_v^{\circ}(\IQbar) \cap \Gamma_v$. By the result over $\IQbar$, we have $\#\mathfrak{X}_v^{\circ}(\IQbar) \cap \Gamma_v \le c(g,\deg_{\cL_v}\mathfrak{X}_v)^{1+\mathrm{rk}\Gamma_v} \le c(g,\deg_L X)^{1+\rho}$. Hence we are done.
\end{proof}

Now we are ready to finish the proof of Theorem~\ref{MainThm3}.

\begin{proof}[Proof of Theorem~\ref{MainThm3} for arbitrary $F$] 
Let $F$ be an arbitrary algebraic closed field of characteristic $0$. 
Let $A$, $L$ and $X$ be as in Theorem~\ref{MainThm3}. Let $\Gamma$ be a subgroup of $A(F)$ of finite rank $\rho$. 

By the definition of a finite rank group,
there exists a finitely generated subgroup $\Gamma_0$ of
$A(F)$ with rank $\rho$  such that
\[
\Gamma \subseteq \{x \in A(F) : [N]x \in \Gamma_0\text{ for some }N\in \mathbb{N}\}.
\]
Moreover, we may choose such a $\Gamma_0$ satisfying that $\Gamma_0 = \mathrm{End}(A)\cdot \Gamma_0$.

For each $n\in \mathbb{N}$, define
\[
\frac{1}{n}\Gamma_0 := \{ x \in A(F) : [n]x \in \Gamma_0\}.
\]
Then $\frac{1}{n}\Gamma_0$ is again a finitely generated subgroup of $A(F)$ of rank $\rho$, and is invariant under $\mathrm{End}(A)$.

We have proved Theorem~\ref{MainThm3} over $\IQbar$ in $\mathsection$\ref{SubsectionUMLIQbar}. Thus Theorem~\ref{MainThm3}  holds true for $\frac{1}{n}\Gamma_0$ and our $F$ by the specialization result above (Lemma~\ref{LemmaSpecialization}). So there exists a constant $c= c(g, \deg_L X) > 0$ such that
\begin{equation}
  \label{EqBoundLeveln}
  \#X^{\circ}(F) \cap \frac{1}{n}\Gamma_0 \le c^{1+\rho}.
\end{equation}

Note that  $\{\frac{1}{n}\Gamma_0\}_{n\in \mathbb{N}}$ is a filtered system and $\Gamma \subseteq \bigcup_{n\in \mathbb{N}}\frac{1}{n}\Gamma_0$. But the bound \eqref{EqBoundLeveln} is independent of $n$. So
\[
  \#X^{\circ}(F) \cap \Gamma \le c^{1+\rho}.
\]
This is precisely Theorem~\ref{MainThm3}. Hence we are done.
\end{proof}

\subsection{From Theorem~\ref{MainThm3} to Theorem~\ref{MainThm2}}\label{SubsectionMainThm3toMainThm2}
Now that we have proved Theorem~\ref{MainThm3}, we can conclude for Theorem~\ref{MainThm2} with the following lemma.
\begin{lemma}\label{LemmaConjHighDimMLStrWeakEqui}
Assume Theorem~\ref{MainThm3} holds true for all $(A,L)$, $X$, and $\Gamma$. Then Theorem~\ref{MainThm2} also holds true.
\end{lemma}
\begin{proof}
Using the same argument as in the proof of Theorem~\ref{MainThm3} with $F=\IQbar$ at the end of $\mathsection$\ref{SubsectionUMLIQbar}, we may and do assume that $X$ generates $A$.

We start with the following finer description of the Ueno locus of $X$. Let $\Sigma(X;A)$  be the set of abelian subvarieties $B \subseteq A$ with $\dim B > 0$ satisfying: $x+B \subseteq X$ for some $x \in A(F)$, and $B$ is maximal for this property. Then for each $B \in \Sigma(X;A)$, there exists a closed subvariety $X_B$ of $X$ such that the Ueno locus of $X$ is $\bigcup_{B \in \Sigma(X;A)}(X_B+B)$.

The union above can be expressed in a quantitative way. First, by Bogomolov \cite[Thm.~1]{BogomolovUeno}, each $B \in \Sigma(X;A)$ satisfies $\deg_L B \le c_3$ for some constant $c_3 = c_3(\dim A,\deg_L X) > 0$, and hence $\# \Sigma(X;A) \le c_4 = c_4(\dim A,\deg_L X)$ by \cite[Prop.~4.1]{Remond:Decompte}; here we used Lemma~\ref{LemmaSmallestAbVarGenerated}. Next, $X_B$ can be constructed as follows. Let $B^\perp$ be a complement of $B$, \textit{i.e.} $B \cap B^\perp$ is finite and $B+B^\perp = A$. It is possible to choose such a $B^\perp$ with $\deg_L B^\perp \le c_5'(g, \deg_L A, \deg_L B)$; see \cite{MW:Complement}. Then we can choose $X_B:=\bigcap_{b \in B(F)}(X-b) \bigcap B^\perp$. Notice that by dimension reasons, 
 this intersection must be the component of a finite intersection of at most $\dim X \le \dim A$ members. So $\deg_L X_B \le c_5(\dim A,\deg_L X)$ by B\'{e}zout's Theorem and Lemma~\ref{LemmaSmallestAbVarGenerated}. In particular $X_B$ has $\le c_5$ irreducible components $X_{B,1},\ldots,X_{B,m_B}$.

As the $B_i$'s in \eqref{EqHighDimML} satisfies $x_i+B_i \subseteq X$ and $\dim B_i > 0$, we may and do assume $B_i \in \Sigma(X;A)$ by definition of the Ueno locus.  Now \eqref{EqHighDimML} becomes
\begin{equation}\label{EqConjHighDimMLStrOrg}
X(F) \cap \Gamma = \bigcup_{B\in \Sigma(X;A)} \bigcup_{j=1}^{n_B} (x_{B,j}+B)(F) \cap \Gamma  \coprod X^{\circ}(F)\cap \Gamma.
\end{equation}
Moreover, each $x_{B,j}$ can be chosen to be in $X_B^{\circ}(F) \cap \Gamma$, where $X_B^{\circ} = \bigcup_{k=1}^{m_B}X_{B,k}^{\circ}$. See \cite[Lem.~4.6]{Remond:Decompte}; notice that $p|_{X_B}$ is finite for the quotient $p \colon A \rightarrow A/B$. In particular, $n_B \le \#X_B^{\circ}(F) \cap \Gamma$.

Let us bound $n_B$ for each $B \in \Sigma(X;A)$. Applying Theorem~\ref{MainThm3} to each irreducible component $X_{B,k}$ of $X_B$, we get $\#X_{B,k}^{\circ}(F) \cap \Gamma \le c^{1+\mathrm{rk}\Gamma}$ for some $c = c(\dim A,\deg_L X_{B,k}) > 0$. But we have seen that $X_B$ has $\le c_5$ components and that $\deg_L X_{B,k} \le \deg_L X_B \le c_5$. So $\#X_B^{\circ}(F) \cap \Gamma$, and hence $n_B$, is $\le c_6(\dim A,\deg_L X)^{1+\mathrm{rk}\Gamma}$.

From the bounds on $\Sigma(X;A)$ and $n_B$ above, we get from \eqref{EqConjHighDimMLStrOrg} that $N \le c_4 c_6^{1+\mathrm{rk}\Gamma} + \#X^\circ(F)\cap \Gamma$. Hence Theorem~\ref{MainThm2} holds true by applying Theorem~\ref{MainThm3} again to $X^\circ(F)\cap \Gamma$.
\end{proof}

\section{Proof of the uniform Bogomolov conjecture (Theorem \ref{ThmUBC})}\label{SectionUBC}

\begin{proof}[Proof of Theorem~\ref{ThmUBC}]
Let $A$ be an abelian variety of dimension $g$, let $L$ be an ample line bundle, and let $X$  be an irreducible subvariety. 
Assume all these objects are defined over $\IQbar$.

Write $d = \deg_L X$ and $r = \dim X$. Let $c = c(g)$ be the constant from Lemma \ref{LemmaSmallestAbVarGenerated}. Let $c_2''' = c_2''(g,cd^{2g},r,d) > 0$ and $c_3''' = c_3''(g,cd^{2g},r,d) > 0$ be from Proposition \ref{PropEqDistStep1}. Moreover, set $c_3 := \min_{1 \le r \le g}\{c_3''(g,cd^{2g},r,d)\} > 0$ and  $c_2:=\max_{1\le r\le g} \{c_2''(g,cd^{2g},r,d)\} > 0$; both $c_3$ and $c_2$ depend only on $g$ and $d$.

We prove the theorem by induction on $r$. The base step $r = 0$ trivially holds true.

For arbitrary $r$, assume the theorem is proved for $0,\ldots,r-1$. 

Let $A'$ be the abelian subvariety of $A$ generated by $X-X$, then $\deg_L A' \le cd^{2g}$ by Lemma \ref{LemmaSmallestAbVarGenerated}. We can apply Proposition \ref{PropEqDistStep1} to $X$ and $(A', L|_{A'})$ to conclude that the set
\begin{equation}\label{EqUBCSet}
\Sigma:=\left\{P \in X^{\circ}(\IQbar) : \hat{h}_{L\otimes L_{-}}(P) \le c_3 \right\},
\end{equation}
where $L_{-} = [-1]^*L$, 
is contained in $X'(\IQbar)$, for some proper Zariski closed $X' \subsetneq X$ with $\deg_L(X') < c_2$. 
Each irreducible component of $X'$ has dimension $\le r-1$, and $X'$ has $\le c_2$ irreducible components. Hence the conclusion  follows by applying the induction hypothesis to each irreducible component of $X'$ and appropriately adjusting $c_3$ and $c_2$.
\end{proof}

\appendix
\renewcommand{\thesection}{\Alph{section}}
\setcounter{section}{0}

\section{R\'{e}mond's theorem revisited}\label{AppendixRemond}

The goal of this appendix is to give a more detailed proof of R\'{e}mond's theorem, which we cited as Theorem~\ref{ThmRemond}, to make the current paper more complete. We will explain how R\'{e}mond's generalized Vojta's Inequality, generalized Mumford's Inequality, and the technique to remove the height of the subvariety together imply the desired Theorem~\ref{ThmRemond}. The proof follows closely the arguments presented in \cite{Remond:Decompte}.

We work over $\IQbar$.

Let $A$ be an abelian variety and let $L$ be a \textit{symmetric} ample line bundle on $A$. To ease notation, we may and do assume $L$ is very ample and gives a projectively normal closed immersion into some projective space, by replacing $L$ by $L^{\otimes 4}$.

Let us restate Theorem~\ref{ThmRemond}.

Let $X$ be an irreducible subvariety of $A$, and $\Gamma$ be a finite rank subgroup of $A(\IQbar)$. We say that \textit{the assumption {\tt (Hyp pack)} holds true for $(A,L)$, $X$ and $\Gamma$}, if there exists a constant $c_0 = c_0(g, \deg_L X) > 0$ satisfying the following property: for each $P_0 \in X(\IQbar)$,
\begin{equationrepeat}{EqRemondAssumption}\label{EqRemondAppAssumptionApp}
\left\{P \in (X^{\circ}(\IQbar) -  P_0) \cap \Gamma : \hat{h}_L(P - P_0) \le c_0^{-1} \max\{1,h_{\mathrm{Fal}}(A)\} \right\} \le c_0^{\mathrm{rk}\Gamma + 1}.
\end{equationrepeat}

\begin{thmrepeat}{ThmRemond}\label{ThmRemondApp}
Assume that {\tt (Hyp pack)} holds true for all $(A,L)$, $X$, $\Gamma$ (as above) such that $X$ generates $A$.

Then for each polarized abelian variety $(A,L)$ with $L$ symmetric and very ample, each irreducible subvariety $X$ of $A$ and each  finite rank subgroup $\Gamma$ of $A(\IQbar)$,  we have
\begin{equationrepeat}{EqRemond}\label{EqRemondApp}
\#X^{\circ}(\IQbar) \cap \Gamma \le c(g,\deg_L X, \deg_L A)^{\mathrm{rk}\Gamma + 1}.
\end{equationrepeat}
\end{thmrepeat}

Moreover, we may and do assume that {\it the function $c$ is increasing in all three invariables}, by replacing $c(g,d,l)$ by $\max_{g'\le g, d'\le d, l' \le l} c(g',d',l')$.

In what follows, we will introduce many constants $c_4, c_5, \ldots$. All these constants are assumed to \textit{depend only on $g$, $\deg_L X$, and $\deg_L A$} unless stated otherwise.

\subsection{Preliminary setup}
We have $H^0(A, L) = \deg_L A / g!$ by Lemma~\ref{LemmaAbIsogPPAV}. Thus $A$ can be embedded into the projective space $\mathbb{P}^{\deg_L A / g! - 1}$ using global sections of $H^0(A, L)$. Thus the integer $n$ in \cite{Remond:Vojtasup, Remond:Decompte} can be taken to be $\deg_L A / g! - 1$.

The closed immersion $A \subseteq \mathbb{P}^{\deg_L A / g! - 1}$ defines a height function $h \colon A(\IQbar) \rightarrow \mathbb{R}$. The Tate Limit Process then gives rise to a height function
\[
\hat{h}_L \colon A(\IQbar) \rightarrow [0,\infty), \quad P \mapsto \lim_{N\rightarrow \infty}\frac{h([N^2]P)}{N^4}.
\]

For $P,Q\in A(\IQbar)$  we set
$\langle P,Q\rangle = (\hat{h}_L(P+Q)-\hat{h}_L(P)-\hat{h}_L(Q))/2$
and often abbreviate $|P| = \hat{h}_L(P)^{1/2}$. The notation $|P|$ is
justified by the fact that it induces a norm after tensoring with the reals.

Let $c_{\mathrm{NT}}(A)$ and $h_1(A)$ be as from \cite[Thm.~6.8]{DaPh:07}. It is known that 
\begin{equation}\label{EqcNTh1FaltingsHeight}
c_{\mathrm{NT}}(A), h_1(A) \le c' \max\{1,h_{\mathrm{Fal}}(A)\}
\end{equation}
 for some $c' = c'(g,\deg_L A) >0$; see \cite[equation (6.41)]{DaPh:07}.

Finally for any irreducible subvariety $X$ of the projective space $\mathbb{P}^{\deg_L A / g! - 1}$, one can define the height $h(X)$; see \cite{BGS}.

\subsection{Generalized Vojta's Inequality}
\begin{thm}[{\!\!\cite[Thm.~1.1]{Remond:Vojtasup}}]
  \label{thm:vojta}
  There exist constants $c_4 = c_4(g, \deg_L X, \deg_L A)> 0$ and $c_5 =c_5(g, \deg_L X, \deg_L A) > 0$ with the
  following property. 
If $P_0,\ldots,P_{\dim X} \in X^{\circ}(\IQbar)$ satisfy
\begin{equation*}
 \langle P_i,P_{i+1}\rangle \ge \left(1-\frac{1}{c_4}\right) |P_i||P_{i+1}|
\quad\text{and}\quad |P_{i+1}| \ge c_4 |P_i|,
\end{equation*}
then 
\begin{equation*}
|P_0|^2 \le c_5
    \max\{1,h(X),h_{\mathrm{Fal}}(A)\}.
\end{equation*}
\end{thm}
\begin{proof} This follows immediately from \cite[Thm.~1.1]{Remond:Vojtasup} (with $n = \deg_L A / g! - 1$) and \eqref{EqcNTh1FaltingsHeight} and $\dim X \le g$.
\end{proof}

\subsection{Generalized Mumford's Inequality}
Let $D_{\dim X} \colon X^{\dim X+1} \rightarrow A^{\dim X}$ be the morphism defined by $(x_0,x_1,\ldots,x_{\dim X}) \mapsto (x_1-x_0, \ldots, x_{\dim X}-x_0)$.

\begin{prop}[{\!\!\cite[Prop.~3.4]{Remond:Decompte}}]
\label{thm:mumford}  There exist constants $c_4 =c_4(g, \deg_L X, \deg_L A) > 0$ and $c_5 =c_5(g, \deg_L X, \deg_L A) > 0$ with the
  following property. 
Let $P_0 \in X(\IQbar)$. Suppose  $P_1,\ldots, P_{\dim X}\in X(\IQbar)$ with $(P_0, P_1,\ldots,P_{\dim X})$ isolated in the fiber of $D_{\dim X} \colon X^{\dim X+1} \rightarrow A^{\dim X}$. If 
  \begin{equation*}
 \langle P_0,P_i\rangle \ge \left(1-\frac{1}{c_4}\right)  |P_0|
 |P_i|\quad\text{and}\quad
 \bigl| |P_0| -  |P_i|\bigr| \le \frac{1}{c_4}  |P_0|
  \end{equation*}
  then
  \begin{equation*}
|P_0|^2  \le c_5 \max\{1,h(X), h_{\mathrm{Fal}}(A)\}.
  \end{equation*}
\end{prop}
\begin{proof} This follows immediately from \cite[Prop.~3.4]{Remond:Decompte} (with $n = \deg_L A / g! - 1$) and \eqref{EqcNTh1FaltingsHeight} and $\dim X \le g$.
\end{proof}

\begin{prop}[{\!\!\cite[Prop.~3.3]{Remond:Decompte}}]\label{PropAlternativeMumfordIneq}
Let $\Xi \subseteq X^{\circ}(\IQbar)$. 
We are in one of the following alternatives.
\begin{enumerate}
\item[(i)] Either for any $x \in X(\IQbar)$, there exist pairwise distinct $x_1,\ldots,x_{\dim X} \in \Xi$ such that $(x,x_1,\ldots,x_{\dim X})$ is isolated in the fiber of $D_{\dim X} \colon X^{\dim X+1} \rightarrow A^{\dim X}$;
\item[(ii)] or $\Xi$ is contained in a proper Zariski closed subset $X' \subsetneq X$ with $\deg X' < (\deg_L X)^{2\dim X}$.
\end{enumerate}
\end{prop}
\begin{proof}
Denote by $0$ the origin of the abelian variety $A$. We may and do assume that the stabilizer of $X$ in $A$, denoted by $\mathrm{Stab}(X)$ has dimension $0$; otherwise $X^{\circ} = \emptyset$ and the proposition trivially holds true.

The points in the fiber of $D_{\dim X}$ in question can be written as $(x+a, x_1+a, \ldots, x_r+a)$ with $a$ running over the $\IQbar$-points of $(X-x)\cap (X-x_1) \cap \cdots \cap (X-x_{\dim X})$. Thus $(x,x_1,\ldots,x_{\dim X})$ is isolated in the fiber of the image of $X^{\dim X+1} \rightarrow A^{\dim X}$ if and only if
\begin{equation}
\dim_0 (X-x) \cap (X-x_1) \cap \cdots (X-x_{\dim X}) = 0.
\end{equation}

Assume we are \textit{not} in case (i). Then there exists $i_0 \le \dim X-1$ satisfying the following property. 
There are pairwise distinct points $x_1, \ldots, x_{i_0}$ such that for $W : =  (X-x) \cap (X-x_1) \cap \cdots \cap (X-x_{i_0})$, we have
\begin{equation}\label{EqInterDimNonDecrease}
\dim_0 W = \dim_0  W \cap (X-y) = \dim X-i_0 \quad \text{ for all } y \in \Xi.
\end{equation}

Let $C_1,\ldots, C_s$ be the irreducible components of $W$ passing through $0$ with $\dim C_j = \dim X-i_0 \ge 1$. Then $s \le \sum_{j=1}^s \deg C_j \le (\deg X)^{i_0+1} \le (\deg X)^{\dim X}$ by B\'{e}zout's Theorem. Moreover,
\small
\begin{align*}
\dim_0 W\cap (X-y) = \dim_0 W = \dim X-i_0&  \Leftrightarrow C_j \subseteq X-y \text{ for some } j \in \{1,\ldots,s\} \\
& \Leftrightarrow y \in \bigcap_{c \in C_j(\IQbar)}(X-c) \text{ for some } j \in \{1,\ldots,s\}.
\end{align*}
\normalsize
So \eqref{EqInterDimNonDecrease} is equivalent to 
$\Xi \subseteq \bigcup_{j=1}^s \bigcap_{c \in C_j(\IQbar)}(X-c)$. 
Each $\bigcap_{c \in C_j(\IQbar)}(X-c)$ is a finite intersection of at most $\dim X$ members because of dimension reasons. So $\deg \bigcap_{c \in C_j(\IQbar)}(X-c) \le (\deg X)^{\dim X}$ by B\'{e}zout's Theorem. Moreover, each irreducible component of $\bigcap_{c \in C_j(\IQbar)}(X-c)$ has dimension $< \dim X$ because $\dim \mathrm{Stab}(X) = 0$. Hence we are in case (ii) by setting $X' = \bigcup_{j=1}^s \bigcap_{c \in C_j}(X-c)$. So we are done.
\end{proof}

\subsection{Removing $h(X)$}
\begin{lemma}[{\!\!\cite[Lem.~3.1]{Remond:Decompte}}]\label{LemmaSupressionHeightX}
Assume $S \subseteq X(\IQbar)$ a finite set. Assume that each equidimensional subvariety $Y \supseteq S$ of $X$ of dimension $\dim X-1$ satisfies $\deg_L Y > \deg_L A (\deg_L X)^2 / g!$. Then
\[
h(X) \le (\deg_L A / g! + 1)^{\dim X+1} \deg_L X \left( \max_{x\in S}h(x) + 3\log(\deg_L A / g!) \right).
\]
\end{lemma}
\begin{proof} This is precisely \cite[Lem.~3.1]{Remond:Decompte} with $n = \deg_L A / g! - 1$.
\end{proof}

\subsection{Proof of Theorem~\ref{ThmRemondApp}}
Let $X$ be an irreducible subvariety of $A$, and let $\Gamma$ be a subgroup of $A(\IQbar)$ of finite rank.

We start by reducing to the case where   $X$ generates $A$. 
Indeed, let $A'$ be the abelian subvariety of $A$ generated by $X-X$. Then $X \subseteq A' + Q$ for some $Q \in A(\IQbar)$. The subgroup $\Gamma'$ of $A(\IQbar)$ generated by $\Gamma$ and $Q$ has rank $\le \mathrm{rk}\Gamma + 1$. We have $(X-Q)^{\circ} = X^{\circ} - Q$ by definition of the Ueno locus, $(X^{\circ}(\IQbar)-Q) \cap \Gamma  \subseteq  (X^{\circ}(\IQbar) - Q) \cap \Gamma' = X^{\circ}(\IQbar) \cap \Gamma'$ and $\deg_L(X-Q) = \deg_L X$. By Lemma~\ref{LemmaSmallestAbVarGenerated}, {\tt (Hyp)} yields $\deg_L A' \le c'(g,\deg_L X)$. Therefore 
if \eqref{EqRemond} holds true for $X-Q \subseteq A'$, $L|_{A'}$ and $\Gamma' \cap A'(\IQbar)$, then $\#X^{\circ}(\IQbar) \cap \Gamma \le c(g,\deg_L X)^{\mathrm{rk}\Gamma' + 1} \le  c(g,\deg_L X)^{\mathrm{rk}\Gamma + 2}$. So we can conclude by replacing $c$ with $c^2$. Thus we are reduced to the case where $X$ generates $A$.

From now on, assume $X$ generates $A$. We prove \eqref{EqRemondApp} by induction on
\begin{equation}
r := \dim X.
\end{equation}

The base step is $r = 0$, in which case trivially holds true.

For an arbitrary $r \ge 1$. Assume \eqref{EqRemondApp} holds true for $0, 1, \ldots, r-1$.

Observe that both Theorem~\ref{thm:vojta} and Proposition~\ref{thm:mumford}
hold with $c_4$ and $c_5$ replaced by some larger value.  
We let $c_4$ (resp. $c_5$) denote the maximum of both constants $c_4$ (resp. $c_5$) from these two
theorems. Both constants depend only on $g$, $\deg_L X$ and $\deg_L A$.

\medskip
\noindent\underline{\textbf{Step~1:} Handle large points by both inequalities of R\'{e}mond.} The goal of this step is to prove the following bound: there exists a constant $c_6 = c_6(g, \deg_L X, \deg_L A) > 0$ such that
\begin{equation}\label{EqLargePointBoundWithHeightX}
\#\left\{ P \in X^{\circ}(\IQbar) \cap \Gamma : \hat{h}_L(P) \ge c_5 \max\{1,h(X),h_{\mathrm{Fal}}(A)\} \right\} \le c_6^{\mathrm{rk}\Gamma + 1}.
\end{equation}

The proof follows a standard classical argument involving the inequalities of Vojta and Mumford. Consider the $\mathrm{rk}\Gamma$-dimensional real vector space $\Gamma\otimes\mathbb{R}$ endowed with the Euclidean norm $| \cdot | = \hat{h}_L^{1/2}$. 
 We may and do assume
$\mathrm{rk}\Gamma \ge 1$. By elementary geometry, the
vector space can be covered by at most  $\lfloor
(1+({8c_4})^{1/2})^{\mathrm{rk}\Gamma} \rfloor$ cones on which $\langle P,Q\rangle \ge
(1-1/c_4) |P| |Q|$ holds. 

Let $P_0, P_1,P_2,\ldots,P_N \in X^{\circ}(\IQbar)\cap \Gamma$ be pairwise
distinct points in one such cone such that
\begin{equation}
  \label{eq:Piincreasing}
   c_5 \max\{1,h(X),h_{\mathrm{Fal}}(A) \} < |P_0|^2 \le |P_1|^2 \le |P_2|^2
 \le \cdots \le |P_N|^2.
  \end{equation}

Notice that 
\begin{equation}
  \label{eq:PiPjincone}
  \langle P_i,P_j\rangle \ge \left(1-\frac{1}{c_4}\right) |P_i||P_j| \quad \text{ for all }i,j\in \{0,\ldots,N\}.
\end{equation}

Set $N' := (\deg_L X)^{2g} c(g, (\deg_L X)^{2g}, \deg_L A)^{\mathrm{rk}\Gamma + 1} + 1$, with $c$ the constant from \eqref{EqRemondApp}.

Consider the subset $\Xi_j = \{P_{j+1},\ldots,P_{j+N'}\}$ with $j \in \{0,\ldots, N-N'\}$; it has $N'$ pairwise distinct elements. We claim that $\Xi$ cannot be contained in a proper Zariski closed subset $X' \subsetneq X$ with $\deg_L X' \le (\deg_L X)^{2\dim X}$. Indeed if such an $X'$ exists, then by definition of the Ueno locus we have $(X')^{\circ} \supseteq X^{\circ} \cap X'$. So $\Xi \subseteq (X')^{\circ}(\IQbar)  \cap \Gamma$. As $\dim X' < \dim X= r$, we can apply the induction hypothesis \eqref{EqRemondApp} to each irreducible component of $X'$. As $X'$ has $\le (\deg_L X)^{2r}$ irreducible components and each component has degree $\le (\deg_L X)^{2r}$, we then get $\# \Xi \le (\deg_L X)^{2r} c(g, (\deg_L X)^{2r} , \deg_L A)^{\mathrm{rk}\Gamma + 1}$. This contradicts our choice of $N'$ because $c$ is increasing in all the three variables.

 By Proposition~\ref{PropAlternativeMumfordIneq} applied to $\Xi_j$ and $P_j$, we then get pairwise distinct points $P_{i_1},\ldots,P_{i_r} \in \Xi_j$ such that  $(P_j,P_{i_1},\ldots,P_{i_r})$ is isolated in the fiber of $D_r \colon X^{r+1} \rightarrow A^r$. 
But the hypotheses of Proposition~\ref{thm:mumford} cannot hold true by \eqref{eq:Piincreasing} and \eqref{eq:PiPjincone}. So there exists $k \in \{i_1,\ldots,i_r\}$ such that $|P_k| -  |P_j| > \frac{1}{c_4}  |P_j|$. As $k  \le j+N'$, we then have 
\[
|P_{j+N'}|  > (1+\frac{1}{c_4})|P_j|.
\]
This holds true for each $j \in \{0,\ldots, N-N'\}$. So 
\begin{equation}\label{EqLaLaLaApp}
|P_{j+N'k}| > (1+\frac{1}{c_4})^k|P_j|
\end{equation}
for all $j \ge 0$ and $k \ge 1$.

Next we choose an integer $M \ge 0$ such that $(1+1/c_4)^M \ge c_4$. We may and do assume that $M$ depends only on $g$, $\deg_L X$ and $\deg_L A$ (since $c_4$ does). Then by \eqref{EqLaLaLaApp},  $|P_{(k+1) M N'}| > (1+1/c_4)^M |P_{kMN'}| \ge c_4 |P_{kMN'}|$ for each $k \ge 0$.

We claim $N < rMN' \le g MN'$. Indeed, assume $N \ge rMN'$. Then we have $r+1$ pairwise distinct points $P_0, P_{MN'}, P_{2MN'}, \ldots, P_{rMN'}$. 
The hypotheses of  Theorem~\ref{thm:vojta} for these points cannot hold true by \eqref{eq:Piincreasing} and \eqref{eq:PiPjincone}. Thus there exists $k$ such that $|P_{(k+1)MN'}| < c_4 |P_{kMN'}|$. But this contradicts the conclusion of last paragraph. So we much have $N < rMN' \le gMN'$.

Recall that we have covered $\Gamma\otimes \mathbb{R}$ by at most $\lfloor
(1+({8c_4})^{1/2})^{\mathrm{rk}\Gamma} \rfloor$ cones and each cone contains $< gMN'$ points  $P \in X^{\circ}(\IQbar) \cap \Gamma$ with $\hat{h}_L(P) =  |P|^2 \ge c_5 \max\{1,h(X),h_{\mathrm{Fal}}(A)\}$. Thus 
\begin{align*}
\#\left\{ P \in X^{\circ}(\IQbar) \cap \Gamma : \hat{h}_L(P) \ge c_5 \max\{1,h(X),h_{\mathrm{Fal}}(A)\} \right\} & \le 
(1+({8c_4})^{1/2})^{\mathrm{rk}\Gamma} gMN'
\end{align*}
All constants on the right hand side depend only on $g$, $\deg_L X$ and $\deg_L A$. 
So \eqref{EqLargePointBoundWithHeightX} holds true by choosing $c_6$ appropriately.

\medskip
\noindent\underline{\textbf{Step~2:} Remove the dependence on $h(X)$.} More precisely, set 
\small
\begin{equation}\label{Eqc7Nprimeprime}
c_7 := \deg_L A (\deg_L X)^2 / g! \cdot c(g, \deg_L A (\deg_L X)^2 / g! , \deg_L A) \quad \text{and}\quad N'':=c_7^{\mathrm{rk}\Gamma+1} + 1. 
\end{equation}
\normalsize
The goal of this step is to prove: There exist positive constants $c_8$, $c_9$, $c_{10}$, depending only on $g$, $\deg_L X$, and $\deg_L A$ with the following property. 
If $P_0, \ldots, P_{N''}$ are pairwise distinct points in $X^{\circ}(\IQbar) \cap \Gamma$, then
\begin{equation}\label{EqSupressionHauteurXEq1}
\small \#\left\{P \in X^{\circ}(\IQbar) \cap \Gamma : \hat{h}_L(P-P_0) \ge c_8^2 \max_{1\le i \le N''} \hat{h}_L(P_i-P_0)  + c_9\max\{1,h_{\mathrm{Fal}}(A)\} \right\} \le c_{10}^{\mathrm{rk}\Gamma+1}.
\end{equation}
\normalsize
The proof follows closely \cite[Prop.~3.6]{Remond:Decompte}. 
We wish to apply Lemma~\ref{LemmaSupressionHeightX} to $X-P_0$ and the set $S = \{ P_i - P_0 : 0 \le i \le N''\}$. Let us verify the hypothesis. Let $Y \subseteq X - P_0$ with $Y$ equidimensional of dimension $\dim X - 1 = r-1$ and $S \subseteq Y(\IQbar)$, and set $Y' := Y + P_0$. Then $P_i \in (Y')^{\circ}(\IQbar) \cap \Gamma$.\footnote{As $Y' \subseteq X$, we have $(Y')^{\circ} \supseteq X^{\circ} \cap Y'$ by definition of the Ueno locus. So $P_i \in (Y' \cap X^{\circ})(\IQbar) \cap \Gamma \subseteq (Y')^{\circ}(\IQbar) \cap \Gamma$.} Each irreducible component of $Y'$ has dimension $\le r-1$. Thus we can apply induction hypothesis \eqref{EqRemondApp} to each irreducible component of $Y'$. So $N'' \le \sum_{Y''} c(g,\deg_L Y'', \deg_L A)^{\mathrm{rk}\Gamma + 1}$, with $Y''$ running over all irreducible components of $Y'$. Since $\deg_L Y = \deg_L Y' = \sum_{Y''} \deg_L Y''$ and $c$ is increasing in all the three variables, this bound implies 
\[
N'' \le \deg_L Y \cdot c(g,\deg_L Y, \deg_L A)^{\mathrm{rk}\Gamma + 1}.
\]
We claim $\deg_L Y > \deg_L A (\deg_L X)^2 / g!$. Indeed, assume $\deg_L Y \le \deg_L A (\deg_L X)^2 / g!$. Then $N'' \le \deg_L A (\deg_L X)^2 / g! \cdot c(g, \deg_L A (\deg_L X)^2 / g! , \deg_L A)^{\mathrm{rk}\Gamma+1}$ because $c$ is increasing in all the three variables. This contradicts the definition of $N''$ from \eqref{Eqc7Nprimeprime}.

Thus the assumption of Lemma~\ref{LemmaSupressionHeightX} is satisfied. 
So 
\[
h(X) \le c_{11}(g, \deg_L X, \deg_L A) \left( \max_{1\le i \le N''}\hat{h}_L(P_i-P_0) + \max\{1,h_{\mathrm{Fal}}(A)\}\right).
\] 
Here we also used \eqref{EqcNTh1FaltingsHeight}. Thus by \eqref{EqLargePointBoundWithHeightX}, we can find the desired constants $c_8$, $c_9$, $c_{10}$ such that \eqref{EqSupressionHauteurXEq1} holds true.

\medskip
\noindent\underline{\textbf{Step~3:} Prove the following alternative.}
\begin{enumerate}
\item[(i)] Either $\#X^{\circ}(\IQbar) \cap \Gamma \le N''(8c_8+1)^{\mathrm{rk}\Gamma} + c_{10}^{\mathrm{rk}\Gamma+1}$;
\item[(ii)] or there exists $Q \in X^{\circ}(\IQbar) \cap \Gamma$ such that $\#\{P \in X^{\circ}(\IQbar) \cap \Gamma : \hat{h}_L(P-Q) \ge 2c_9\max\{1,h_{\mathrm{Fal}}(A)\} \} \le c_{10}^{\mathrm{rk}\Gamma+1}$.
\end{enumerate}

\medskip
The proof follows closely \cite[Prop.~3.7]{Remond:Decompte}. Assume we are \textit{not} in case (i), \textit{i.e.} $\#X^{\circ}(\IQbar) \cap \Gamma > N'' (8c_8+1)^{\mathrm{rk}\Gamma}+ c_{10}^{\mathrm{rk}\Gamma+1}$. Let $c_{12}$ be the smallest real number such that there exists $Q \in X^{\circ}(\IQbar) \cap \Gamma$ with
\begin{equation}\label{EqMinimalityc12}
\#\{P \in X^{\circ}(\IQbar) \cap \Gamma : |P-Q| \ge c_{12} \} \le c_{10}^{\mathrm{rk}\Gamma+1}.
\end{equation}

Consider the set $\Xi:=\{P \in X^{\circ}(\IQbar) \cap \Gamma : |P-Q| \le c_{12} \}$ in the $\mathrm{rk}\Gamma$-dimensional Euclidean space  
$(\Gamma\otimes\mathbb{R}, |\cdot|)$. Then  $\#\Xi \ge  \#X^{\circ}(\IQbar) \cap \Gamma - \#\{P \in X^{\circ}(\IQbar) \cap \Gamma : |P-Q| \ge c_{12} \}> N'' (4c_8+1)^{\mathrm{rk}\Gamma}$.  
By an elementary ball packing argument, $\Xi$ (being a subset of $\Gamma \otimes\mathbb{R}$ contained in a closed ball of radius $c_{12}$ centered at $Q$) is covered 
 by at most $(8c_8+1)^{\mathrm{rk}\Gamma}$ balls of radius $c_{12} / 4c_8$ centered at points in $X^{\circ}(\IQbar)  \cap \Gamma$; see \cite[Lem.~6.1]{Remond:Decompte}.  
By the Pigeonhole Principle, one of the balls contains $\ge N''+1$ points in $X^{\circ}(\IQbar) \cap \Gamma$, say $P_0, P_1, \ldots, P_{N''}$. We have $|P_i - P_0| \le c_{12} / 2c_8$ for each $i$. Thus \eqref{EqSupressionHauteurXEq1} yields
\[
\#\left\{P \in X^{\circ}(\IQbar) \cap \Gamma : \hat{h}_L(P-P_0) \ge c_{12}^2/4 + c_9\max\{1,h_{\mathrm{Fal}}(A)\} \right\} \le c_{10}^{\mathrm{rk}\Gamma+1}.
\]
Therefore by the minimality of $c_{12}$, we have $c_{12}^2 \le c_{12}^2/4 + c_9\max\{1,h_{\mathrm{Fal}}(A)\}$, and hence $c_{12}^2 \le 2c_9\max\{1,h_{\mathrm{Fal}}(A)\}$. So we are in case (ii) by \eqref{EqMinimalityc12}.

\medskip

\noindent\underline{\textbf{Step~4:} Conclude by the standard packing argument.}

Recall our assumption that $X$ generates $A$. The assumption of Theorem~\ref{ThmRemondApp} says that {\tt (Hyp pack)} holds true for $X$ and $\Gamma$, \textit{i.e.} we have \eqref{EqRemondAppAssumptionApp}.

Assume we are in case (i) from Step~3. Recall the definition of $N'' = c_7^{\mathrm{rk}\Gamma+1} + 1$ from \eqref{Eqc7Nprimeprime}. Then $\#X^{\circ}(\IQbar) \cap \Gamma \le c^{\mathrm{rk}\Gamma+1}$ for $c: = 2\max\{(c_7+1)(8c_8+1), c_{10}\}$. Hence we can conclude for this case.

Assume we are in case (ii) from Step~3. Set $R = (2 c_9\max\{1,h_{\mathrm{Fal}}(A)\})^{1/2}$ and $R_0 = (c_0^{-1}\max\{1,h_{\mathrm{Fal}}(A)\})^{1/2}$. 
By an elementary ball packing argument, any subset of $\Gamma \otimes\mathbb{R}$ contained in a closed ball of radius $R$ centered at $Q$ is covered by at most $(1+2R/R_0)^{\mathrm{rk}\Gamma}$ closed balls of radius $R_0$ centered at the elements $P - P_0$ with $P$ from the given subset \eqref{EqRemondAppAssumptionApp}; see \cite[Lem.~6.1]{Remond:Decompte}. Thus the number of balls in the covering is at most $(1+2\sqrt{2 c_9  c_0})^{\mathrm{rk}\Gamma}$. 
 But each closed ball of radius $r$ centered at some $P - P_0$ in \eqref{EqRemondAppAssumptionApp} contains at most $c$ elements by \eqref{EqRemondAppAssumptionApp}. So
\begin{equation}\label{EqRemondAppDP3}
 \#\left\{ P \in X^{\circ}(\IQbar) \cap \Gamma : \hat{h}_L(P-Q) \le 2 c_9\max\{1,h_{\mathrm{Fal}}(A)\}  \right\} \le c_0(1+2\sqrt{2 c_9  c_0})^{\mathrm{rk}\Gamma}.
\end{equation}
Thus we have $\#X^{\circ}(\IQbar) \cap \Gamma \le c_0(1+2\sqrt{2 c_9  c_0})^{\mathrm{rk}\Gamma} + c_{10}^{\mathrm{rk}\Gamma+1}$. So $\#X^{\circ}(\IQbar) \cap \Gamma \le c^{\mathrm{rk}\Gamma+1}$ for $c: = 2\max\{c_0, 1+2\sqrt{2c_9c_0}, c_{10}\}$. Hence we can conclude are this case.

Therefore, it suffices to take $c = 2\max\{(c_7+1)(8c_8+1), c_0, 1+2\sqrt{2c_9c_0}, c_{10}\}$, which is a constant depending only on $g$, $\deg_L X$ and $\deg_L A$. We are done.

\bibliographystyle{alpha}
\def\cprime{$'$}

\vfill

\end{document}